\documentclass[11pt,leqno]{amsart}

\usepackage{csquotes}
\usepackage{amssymb,amsmath, amscd,upgreek}
\usepackage{mathrsfs}
\usepackage{euscript, url}

\usepackage{endnotes}
\usepackage[usenames,dvipsnames]{color}

\usepackage{hyperref}
 \hypersetup{colorlinks,  citecolor=OliveGreen,  linkcolor=Mahogany, urlcolor=black}
\numberwithin{equation}{section}
\theoremstyle{plain}

\theoremstyle{plain}
\newtheorem{theo}{Theorem}[section]
\newtheorem{exam}[theo]{Example}
\newtheorem{lemm}[theo]{Lemma}
\newtheorem{prop}[theo]{Proposition}
\newtheorem{coro}[theo]{Corollary}

\theoremstyle{definition}
\newtheorem{rema}[theo]{Remark}
\newtheorem{examp}[theo]{Example}

\def\Corps{{\mathfrak F}}

\def\F{{F}}

\def\G{\overline{\mathbf{G}}_{x_0} }

\def\I{{\rm I}}

\def\I{{\rm I}}

\def\K{{\EuScript K}}

\def\N{{\overline{\mathbf N}}}

\def\R{{\mathbb R}}

\def\T{{\overline{\mathbf T}}}

\def\V{{\rm V}}
\def\W{{\rm W}}

\def\Y{{\rm Y}}
\def\Z{{\mathbb Z}}

\def\Aa{\mathscr{A}}

\def\Dd{\EuScript{D}}

\def\Ff{\mathscr{F}}

\def\Hh{{\rm H}}

\def\Kk{{\mathcal K}}

\def\Oo{\mathfrak{O}}
\def\Pp{\EuScript{P}}

\def\Uu{\EuScript{U}}

\def\Xx{\mathscr{X}}

\makeatletter
\def\@tocline#1#2#3#4#5#6#7{\relax
  \ifnum #1>\c@tocdepth 
  \else
    \par \addpenalty\@secpenalty\addvspace{#2}%
    \begingroup \hyphenpenalty\@M
    \@ifempty{#4}{%
      \@tempdima\csname r@tocindent\number#1\endcsname\relax
    }{%
      \@tempdima#4\relax
    }%
    \parindent\z@ \leftskip#3\relax \advance\leftskip\@tempdima\relax
    \rightskip\@pnumwidth plus4em \parfillskip-\@pnumwidth
    #5\leavevmode\hskip-\@tempdima
      \ifcase #1
       \or\or \hskip 1em \or \hskip 2em \else \hskip 3em \fi%
      #6\nobreak\relax
    \dotfill\hbox to\@pnumwidth{\@tocpagenum{#7}}\par
    \nobreak
    \endgroup
  \fi}
\makeatother

\def\d{{\rm d}}

\def\i{{\rm i}}

\def\m{{\mathfrak m}}

\def\lp{\langle}
\def\rp{\rangle}
\def\>{\geqslant}
\def\<{\leqslant}

\def\Hom{{\rm Hom}}
\def\End{{\rm End}}

\def\GL{{\rm GL}}

\def\Ker{{\rm Ker}}

\def\id{{\rm id}}

\def\ind{{\rm ind}}

\def\1{{\bf 1}}

\def\Iw{{\rm I'}}
\def\Tp{{\rm T}}
\def\Gp{ { \rm G}}

\newcommand\cal{\mathcal}

\def\val{\operatorname{val}}

\def\Root{\Phi}
\def\Coroot{\check\Phi}
\def\coroot{\check\alpha}

\def\root{{\alpha}}

\def\beq{\begin{equation}}
\def\eeq{\end{equation}}
\def\barr{\begin{eqnarray}}
\def\earr{\end{eqnarray}}
\def\de{\begin{equation*}}
\def\fe{\end{equation*}}
\def\dt{\begin{eqnarray*}}
\def\ft{\end{eqnarray*}}

\newcommand{\Ext}{\operatorname{Ext}}
\newcommand{\cX}{\underline{\underline{\mathbf{X}}}}
\newcommand{\cV}{\underline{\underline{\mathbf{V}}}}
\def\chain{f}
\def\pr{{\rm pr}}

\usepackage{comment}

\title{Pro-p Iwahori-Hecke algebras are Gorenstein}
\date{July 12, 2012}
\author{Rachel Ollivier  and Peter Schneider}

\thanks{RO was partially supported by the NSF grant DMS-1201376 and  the project TheHoPad ANR-2011-BS01-005}

\address{Columbia University, Mathematics, 2990 Broadway, New York, NY 10027, USA}
\email{ollivier@math.columbia.edu}
\urladdr{http://www.math.columbia.edu/~ollivier}

\address{ Universit\"at M\"unster,  Mathematisches Institut,  Einsteinstr. 62, 48291 M\"unster, Germany}
\email{pschnei@uni-muenster.de}

\urladdr{http://www.uni-muenster.de/math/u/schneider/}
\medskip


\begin{document}

\def\smfbyname{}
\maketitle

\begin{abstract}
Let $\mathfrak F$ be a locally compact nonarchimedean field with residue characteristic $p$  and     $\Gp$ the group of $\mathfrak{F}$-rational points of a connected split reductive group   over $\mathfrak{F}$. For $k$  an arbitrary field, we study  the homological properties  of the Iwahori-Hecke $k$-algebra  $\Hh'$  and of the pro-$p$ Iwahori-Hecke $k$-algebra $\Hh$ of $\Gp$. We prove that both of these algebras are Gorenstein rings with self-injective dimension bounded above by the rank of $\Gp$.  If $\Gp$ is semisimple, we
also show that this upper bound is sharp, that both $\Hh$ and $\Hh'$ are Auslander-Gorenstein   and that there is a duality functor on the finitely generated modules of $\Hh$ (respectively $\Hh'$). We obtain the analogous  Gorenstein and Auslander-Gorenstein properties
for the graded rings associated to $\Hh$ and $\Hh'$.

When $k$ has characteristic $p$, we prove that  in \enquote{most} cases  $\Hh$ and $\Hh'$ have infinite global dimension. In particular, we deduce that the category of smooth $k$-representations of $\Gp={\rm PGL}_2(\mathbb Q_p)$ generated by their invariant vectors under the pro-$p$-Iwahori subgroup has infinite global dimension (at least if $k$ is algebraically closed).

\end{abstract}

\setcounter{tocdepth}{1}

\tableofcontents

\medskip

\section*{Introduction}

\subsection{Framework}

The affine Hecke algebra with parameter $\mathbf q$ attached to an affine Coxeter system $(W, S)$ is ubiquitous in representation theory: it is the $\Z[\mathbf q]$-algebra $\cal H$ with basis $(t_w)_w$ indexed by the elements  $w\in W$ with the braid relations
\begin{equation*}
     t_{w} t_{w'}= t_{ww'} \qquad\text{if $\ell(ww')=\ell(w)+\ell(w')$},
\end{equation*}
where $\ell$ is the length function of
$(W,S)$, and the quadratic relations
\begin{equation*}
    t_{s}^2=({\mathbf q}-1)t_s +\mathbf q \qquad\text{for $s\in  S$}.
\end{equation*}
It has been extensively studied (\cite{Lu}, \cite{KL} for example) over $\Z[\mathbf q^{\pm 1}]$,
that is to say when the parameter  is  an invertible indeterminate.

In the representation theory of $p$-adic reductive groups, specializations of the  affine Hecke algebra appear naturally  as follows. Let $\mathfrak F$ be a locally compact nonarchimedean field with residue characteristic $p$ and residue field with $q$ elements,   and  let    $\Gp$ be the group of $\mathfrak{F}$-rational points of a connected split reductive group  $\mathbf G$ over $\mathfrak{F}$. Fix an Iwahori subgroup $\I'$ of $\Gp$ and consider the (extended) affine Coxeter system corresponding to the choice of  an apartment in the semisimple building of $\Gp$ containing the chamber fixed by $\I'$.  An  arbitrary field  $k$ is a $\Z[\mathbf q]$-module via the specialization of $\mathbf q$ to the image of the integer $q$ in the field $k$. The Iwahori-Hecke $k$-algebra of $\Gp$  is  then isomorphic to (the extended version of) $\cal H\otimes_{\Z[\mathbf q]} k$.

This article studies the homological properties  of the Iwahori-Hecke $k$-algebra  $\Hh'$ of $\Gp$, as well as the properties  of an  extension of $\Hh'$ : the  Hecke $k$-algebra $\Hh$ of the pro-$p$-Sylow  subgroup $\I$ of $\I'$.  We are motivated by the case when $k$ has characteristic $p$ where very little is known, since the parameter $\mathbf q$ specializes to $0$. Nevertheless, our methods and most of our results are valid over an arbitrary field $k$.

The case when $k$ is the field of complex numbers appears most often in the literature.
The representation theory of complex affine Hecke algebras  is closely related to the  complex representation theory of $\Gp$. The blocks of the category  of smooth complex representations of $\Gp$ (\cite{BD}) are called Bernstein components. By the program of Bushnell and Kutzko (\cite{BK}), they should be  parametrized by  certain pairs $(\rm U, \rho)$  called \emph{types}, where  $\rho$ is an irreducible smooth representation of an open compact   subgroup $\rm U$ of $\Gp$. For various groups $\Gp$ and various types, it is known that the corresponding block is equivalent to a category of modules over a tensor product of complex affine Hecke algebras (see also \cite{Hei}).

For example, the trivial representation of $\I'$ is a type and the corresponding block is equivalent to the category of modules over  the complex Iwahori-Hecke algebra of $\Gp$ (\cite{Bo}). Since each  Bernstein component has finite global dimension (by a result by Bernstein), the complex Iwahori-Hecke algebra also has finite global dimension.
Furthermore, the  irreducible representations of the complex Iwahori-Hecke algebra are well understood  (\cite{KL}, \cite{Rog}).
 To classify these representations,  it is crucial  that the  algebra  has a presentation due to Bernstein, which can be seen as an analog of the loop presentation for affine Lie algebras. In \cite{KL},  this presentation is the key for  the geometric realization of the  affine Hecke $\mathbb C[\mathbf q^{\pm 1}]$-algebra and  its regular representation in terms  of the equivariant $K$-theory of the Steinberg  variety of triples, with respect to the natural action of ${\mathbf G}\times \mathbb C^*$ (\cite{CG}, \cite{KL}). In their description, the parameter $\mathbf q$ appears naturally as the generator of the Grothendieck group of $\mathbb C^*$.

If $\mathbf q$ is not invertible the affine Hecke algebra has been studied in \cite{Vign}  (and \cite{Vig}). But it does not have a Bernstein presentation and has no known geometric realization. Hence, the Iwahori-Hecke algebra $\Hh'$ over a field $k$ with characteristic $p$  is much less understood than in the complex case.
However, it appears naturally in the context of the mod $p$ representation theory
of $\Gp$, or rather, the pro-$p$ Iwahori-Hecke algebra $\Hh$ does: there is a faithful functor from a full subcategory of the smooth $k$-representations of $\Gp$  (containing all irreducibles) into the category of $\Hh$-modules.   Since it is not in general an equivalence of categories (\cite{Inv}), the link between Hecke modules and representations in characteristic $p$ is more subtle  than in the complex case, and  is not yet fully understood. Nevertheless, it has motivated the study of the representations of the $k$-algebra $\Hh$ which,  in the case when $\mathbf G$ is the general linear group,   has revealed a numerical coincidence between  certain irreducible mod $p$ Galois representations and a family of irreducible Hecke modules called \emph{supersingular} and characterized by the fact that they are annihilated by certain central Hecke operators (this was conjectured in \cite{Vig} and proved in \cite{Oparab}). The meaning and consequences of this coincidence, which suggests that the $\Hh$-modules in characteristic $p$ carry some number theoretic information, still have to be explored. More generally, the recent developments about mod $p$ and $p$-adic Langlands correspondence (\cite{Pas}) for ${\rm GL}_2(\mathbb Q_p)$ show that, as opposed to the complex case, the focus on irreducible representations of $\Gp$ is not going to be enough  to grasp the depth of the phenomena. Thus, on the side of the Hecke modules, one also needs to understand the geometry and homological properties of $\Hh$. The current article  is a first step in this direction.

\medskip

\subsection{\label{results}Results}

Until mentioned otherwise, $k$ is a field with arbitrary characteristic and $\Hh$ and $\Hh'$ are respectively the pro-$p$ Iwahori-Hecke $k$-algebra and the Iwahori-Hecke $k$-algebra of the split $p$-adic reductive group $\Gp$ (notations in Section \ref{sec:Hecke}). In the semisimple building $\mathscr X$ of $\Gp$, we fix an apartment $\Aa$  containing the chamber $C$ fixed by the chosen Iwahori subgroup $\I'$. We denote by $d$ the semisimple rank of $\Gp$.
In this introduction, we describe the results  for $\Hh$. They are equally valid for $\Hh'$ as proved throughout the article.

\subsubsection{}

After some preliminaries  and notations (Sections \ref{sec:gene} and \ref{sec:Hecke}), we describe in Section \ref{sec:univ} the main steps for defining  a natural resolution of $\Hh$ as an $(\Hh,\Hh)$-bimodule which is a  key ingredient in the rest of the article.  To help the flow,  some  details  are postponed to the technical Section \ref{BTtheory}. The result is the following (Theorem \ref{theo:freeresolution}): there is   a natural exact sequence of $(\Hh,\Hh)$-bimodules of the form
\begin{equation}\label{f:H-H-bimod-intro}
    0 \longrightarrow \bigoplus_{F \in \mathscr{F}_d} \Hh (\epsilon_F) \otimes_{\Hh_F^\dagger} \Hh \longrightarrow \ldots \longrightarrow \bigoplus_{F \in \mathscr{F}_0} \Hh (\epsilon_F) \otimes_{\Hh_F^\dagger} \Hh \longrightarrow \Hh \longrightarrow 0 \
\end{equation}
which yields a free resolution of $\Hh$ as a left and as a right $\Hh$-module. The sets $\Ff_i$  are made of certain facets of dimension $i$ contained in the closure $\overline C$ of $C$. For such a facet $F$, the algebra $\Hh_F^\dagger$ is the pro-$p$ Iwahori-Hecke algebra of the stabilizer of $F$ in $\Gp$: it is a subalgebra of $\Hh$ so that $\Hh$ can be seen as a left and right $\Hh_F^\dagger$-module. Note that if $\Gp$ is semisimple, then $\Hh_F^\dagger$ is finite dimensional as a vector space.
(In the above resolution, the right action of
$\Hh_F^\dagger$ on $\Hh$ is twisted by a certain orientation character $\epsilon_F$, hence  the notation  $\Hh (\epsilon_F)$).

As explained in Section \ref{sec:gene}, this resolution ensures that certain homological properties of $\Hh$ are controlled by the \enquote{small} pro-$p$ Iwahori-Hecke algebras
$\Hh_F^\dagger$. In Section \ref{sec:Frobenius}, we establish that each $\Hh_F^\dagger$ is left and right noetherian and has  self-injective dimension equal to the rank $r$ of the center  of $\Gp$.  It implies the following result (Theorem \ref{theo:Gorenstein}):

\begin{theo}  $\Hh$ is a Gorenstein ring of self-injective dimension bounded above by the rank $d+r$ of the group $\Gp$. The same statement is valid for the Iwahori-Hecke algebra $\Hh'$.
\end{theo}

We now explain the idea behind the resolution \eqref{f:H-H-bimod-intro}.   The cornerstone of our argument is the following construction  in \cite{SS}.
For an arbitrary $k$-representation $\mathbf V$ of $\Gp$ generated by its $\I$-invariant subspace,
consider the attached $\Gp$-equivariant homological coefficient system of level zero $\cV$ on the building  $\mathscr{X}$ and  the associated augmented chain complex $C_c^{or} (\mathscr{X}_{(\bullet)}, \cV)\to\mathbf V$.
If $k=\mathbb C$, it yields a resolution for $\mathbf V$ (\emph{loc.cit}).
This fails in general if  $k$ has characteristic $p$ (Remark \ref{rema:about-exactness}). However, we are able to prove, independently of the characteristic of $k$, that passing to the $\I$-invariants vectors in  $C_c^{or} (\mathscr{X}_{(\bullet)}, \cV)\to\mathbf V$ always yields an exact resolution  $(C_c^{or} (\mathscr{X}_{(\bullet)}, \cV))^\I\to\mathbf V^\I$ of $\mathbf V^\I$ as a left $\Hh$-module.
This is done by proving that  the latter complex    is isomorphic to another complex   $C_c^{or}(\mathscr{A}_{(\bullet)}, \cV^\I) \to \mathbf V^\I$ associated to a certain homological coefficient system $\cV^\I$ on  the apartment $\Aa$ (\ref{subsec:univreso-modules}).
This observation was inspired by a similar statement in \cite{Bro}  valid for complex representations.
It remains to show that the new complex is exact, and this is obtained by a classical argument of contractibility of certain subcomplexes of $\Aa$, together with the crucial fact that the coefficient system $\cV^\I$ on $\Aa$ has additional \enquote{local constancy} properties (Theorem  \ref{I-resolutionV}).

Let $\mathbf X$ denote the universal representation of $\Gp$ compactly induced from the trivial character of $\I$ with values in $k$ (it is naturally a right $\Hh$-module). In the case where $\mathbf V=\mathbf X$, we have obtained an exact resolution $(C_c^{or} (\mathscr{X}_{(\bullet)}, \cX))^\I\to\mathbf X^\I$ of $\Hh=\mathbf X^\I$ as an $(\Hh, \Hh)$-bimodule: it is the resolution  \eqref{f:H-H-bimod-intro}. Note that, if the field $k$ has characteristic $0$, then for $\Hh'$ such a result is already contained (with a different proof) in \cite{OpSo}.

\subsubsection{}

Suppose that $\Gp$ is semisimple, meaning $r=0$. Then  by tensoring \eqref{f:H-H-bimod-intro} with any left $\Hh$-module $\m$, we obtain  an exact resolution
\begin{equation}\label{gpr}
Gpr_\bullet(\mathfrak{m}) \longrightarrow \mathfrak{m}
\end{equation}
of $\m$ by Gorenstein projective modules (Lemma \ref{Gorproj-res}), which can be used to compute the right $\Hh$-modules $\Ext_\Hh^i (\mathfrak{m},\Hh)$.

In the process of  proving that the upper bound $d$  in the previous theorem is sharp, we first investigated the case of the two following $\Hh$-modules: the trivial character $\chi_{triv}$ and the sign character $\chi_{sign}$ (in the setting of the complex Iwahori-Hecke algebra and complex representations of $\Gp$, these characters correspond respectively to the trivial and the Steinberg representation of $\Gp$). We computed  that, for $\chi\in\{\chi_{triv}, \chi_{sign}\}$, we have
$\Ext_\Hh^d(\chi, \Hh)=\chi\circ \iota_C $  where $\iota_C$ is an involutive automorphism of $\Hh$ which,  in the  particular case where $\Gp$ is simply connected,  coincides with the canonical involution  exchanging  $\chi_{triv}$ and  $\chi_{sign}$.
This proves that the self-injective dimension of $\Gp$ is equal to $d$. Moreover, it suggests that (under the assumption that $\Gp$ is semisimple) there is a duality functor on finitely generated $\Hh$-modules. This is part of the following result (see \ref{subsec:duality}).

\begin{theo}
Suppose that $\Gp$ is semisimple.
\begin{itemize}
\item[i.] For any left $\Hh$-module of finite length, there is a natural isomorphism of right $\Hh$-modules $$\Ext_\Hh^d(\mathfrak{m},\Hh) \cong \mathfrak{m}^d$$ where $\m^d:=\Hom_k(\iota_C^\ast\mathfrak{m},k)$ and the action of $\Hh$ on $\iota_C^\ast\mathfrak{m}$ is through the automorphism $\iota_C$.
 \item[ii.] In particular, $\Hh$ has injective dimension equal to $d$.
 \item[iii.] $\Hh$ is Auslander-Gorenstein. In particular, we have $\Ext_\Hh^i(\mathfrak{m},\Hh) = 0$ for any $i < d$ and any $\Hh$-module of finite length $\mathfrak{m}$.
\end{itemize}
\end{theo}

We conclude Section \ref{subsec:comparison} by noticing that in the case of the sign and trivial characters, the resolution \eqref{gpr} is in fact a resolution by projective $\Hh$-modules, under certain hypotheses on $k$ which are satisfied for example if $k$ has characteristic  $0$ or $p$: these characters have projective dimension $d$  (Proposition \ref{prop:equal-to-d}).

\subsubsection{}

Without any assumption on $\Gp$, we
study in Section \ref{subsec:projdim} the global dimension of $\Hh$ (and $\Hh'$) when $k$ has characteristic $p$. We prove that in \enquote{most} cases, $\Hh$ and $\Hh'$ both admit a simple module with infinite projective dimension. But, for example, if $\Gp={\rm PGL}_2(\mathbb Q_p)$ the algebra $\Hh$ has infinite global dimension whereas $\Hh'$, for $p\neq 2$, has global dimension $1$. As an application, using \cite{Inv}, we obtain that the category of smooth $k$-representations of $\Gp={\rm PGL}_2(\mathbb Q_p)$   generated by their $\I$-invariant vectors has infinite global dimension, at least if $k$ is algebraically closed.

\subsubsection{}
In Section \ref{sec:graded} we introduce the graded ring $gr_\bullet \Hh$ associated to a natural filtration on $\Hh$. The graded ring similarly associated to $\Hh'$  in the case when $\Gp$ is simply connected is known as the affine nil Coxeter algebra. It is investigated in the theory of affine Grassmannians and noncommutative symmetric functions (compare, for example, \cite{BBPZ,KM,TLa}).
We prove the following.

\begin{theo}\phantomsection
\begin{itemize}
\item[i.] $gr_\bullet \Hh$ is a Gorenstein ring of self-injective dimension bounded above by the rank of the group $\Gp$.
\item[ii.] If furthermore  $\Gp$ is semisimple, then $gr_\bullet \Hh$ is Auslander-Gorenstein.
\end{itemize}
\end{theo}

\bigskip

\bigskip

\section{Some general algebra\label{sec:gene}}

Let $k$ be any field and $A$ be some $k$-algebra.

First we assume that $A$  is a $(A_0, A_0)$-bimodule over a $k$-algebra $A_0$,
where the left, resp.\ right, $A_0$-module structure of $A$ is given by   an algebra homomorphism $\lambda : A_0 \rightarrow A$, resp.\ $\rho : A_0 \rightarrow A$.
We consider the following three properties:
\begin{enumerate}
  \item[1.] $A_0$ is left and right noetherian.
  \item[2.] $A_0$ has finite injective dimension $r$ as a left and as a right module over itself.
  \item[3.] $A$ is free as a left as well as a right $A_0$-module.
\end{enumerate}
In the following a left, resp.\ right, $A$-module will be viewed as an $A_0$-module via the map $\lambda$, resp.\ $\rho$.

\begin{lemm}\phantomsection\label{induced}
\begin{itemize}
\item[i.] Assuming 1.--3., $A$ as a left and as a right $A_0$-module has injective dimension $\leq r$.
\item[ii.] Assuming 3., let $N$ be a left $A$-module which as an $A_0$-module has injective dimension $\leq r$. For any left $A_0$-module $M$ we have $\Ext_A^i (A \otimes_{A_0} M,N) = 0$ for any $i > r$.
\item[iii.] Assume 3., $\lambda = \rho$, and that $A_0$ is a direct summand of $A$ as an $(A_0,A_0)$-bimodule. If the $A_0$-module $M$ has infinite projective dimension then the $A$-module $A \otimes_{A_0} M$ has infinite projective dimension as well.
\end{itemize}
\end{lemm}

\begin{proof}
i. We write $A = \oplus_{i \in I} A_0$, and we choose an injective resolution $0 \rightarrow A_0 \rightarrow E_0 \rightarrow \ldots E_r \rightarrow 0$ of $A_0$ as an $A_0$-module. Since over a noetherian ring arbitrary direct sums of injective modules are injective it follows that $0 \rightarrow A \rightarrow \oplus_{i \in I} E_0 \rightarrow \ldots \rightarrow \oplus_{i \in I} E_r \rightarrow 0$ is an injective resolution of $A$ as an $A_0$-module.

ii. We choose a projective resolution $P_\bullet \longrightarrow M$ of the $A_0$-module $M$. Since $A$ is free as an $A_0$-module by assumption it follows that $A \otimes_{A_0} P_\bullet \longrightarrow A \otimes_{A_0} M$ is a projective resolution of the $A$-module $A \otimes_{A_0} M$. We compute
\begin{equation*}
    \Ext_A^i (A \otimes_{A_0} M,N) = h^i(\Hom_A(A \otimes_{A_0} P_\bullet, N)) = h^i(\Hom_{A_0}(P_\bullet, N)) = \Ext_{A_0}^i (M,N) = 0
\end{equation*}
for $i > r$.

iii. Let $M'$ be any $A_0$-module. As above we have $\Ext^i_A(A \otimes_{A_0} M, A \otimes_{A_0} M') = \Ext^i_{A_0}( M, A \otimes_{A_0} M')$. The third assumption implies that $\Ext^i_{A_0}(M,M')$ is a direct summand of $\Ext^i_{A_0}( M, A \otimes_{A_0} M')$.
\end{proof}

Let us now assume that we have an integer $r \geq 0$ and an exact sequence of $(A,A)$-bimodules
\begin{equation}\label{f:resol}
     0\longrightarrow B_d \longrightarrow \ldots \longrightarrow B_0 \longrightarrow A \longrightarrow 0
\end{equation}
where each term $B_j$ is isomorphic to a finite direct sum of bimodules of the form $A \otimes_{A_0} A$ with $A_0$ as above (but varying with $B_j$ and the respective direct summand). We say that \eqref{f:resol} satisfies one of the properties 1., 2., or 3. if each of the occurring bimodules has this property.

\begin{prop}\phantomsection{\label{injdim}}
\begin{itemize}
\item[i.] Assuming that \eqref{f:resol} satisfies 1.--3., the injective dimension of $A$ as a left and as a right $A$-module is $\leq d+r$.
\item[ii.] Assuming that \eqref{f:resol} satisfies 3., let $N$ be a left $A$-module such that, for all the algebras $A_0$ appearing in the bimodules of  \eqref{f:resol}, the injective dimension of $N$ as an $A_0$-module is $\leq r$. Then $N$ has injective dimension $\leq d+r$.
\item[iii.] Assuming that \eqref{f:resol} satisfies 3., let $M$ be a left $A$-module such that, for all the algebras $A_0$ appearing in the bimodules of \eqref{f:resol}, the projective dimension of $M$ as an $A_0$-module is $\leq r$. Then $M$ has projective dimension $\leq d+r$.
\end{itemize}
\end{prop}

\begin{proof}
By Lemma \ref{induced}.i the first assertion is a special case of the second one. To prove ii. we let $M$ be an arbitrary left $A$-module. By the property 3. the sequence \eqref{f:resol} is an exact sequence of free right $A$-modules. Hence tensoring by $M$ gives the exact sequence of left $A$-modules
\begin{equation*}
     0\longrightarrow B_d \otimes_A M \longrightarrow \ldots \longrightarrow B_0 \otimes_A M \longrightarrow M \longrightarrow 0 \ .
\end{equation*}
Viewing this as a quasi-isomorphism between the complexes $B_d \otimes_A M \longrightarrow \ldots \longrightarrow B_0 \otimes_A M$ and $M$ the corresponding hyper-Ext spectral sequence is a first quadrant spectral sequence of the form
\begin{equation*}
    E_1^{s,t} = \Ext_A^t (B_s \otimes_A M ,N) \Longrightarrow \Ext_A^{s+t} (M,N) \ .
\end{equation*}
Each $A$-module $B_j \otimes_A M$ is a finite direct sum of $A$-modules of the form $A \otimes_{A_0} A \otimes_A M = A \otimes_{A_0} M$. It therefore follows from Lemma \ref{induced}.ii that
\begin{equation*}
    \Ext_A^t (B_s \otimes_A M,N) = 0 \qquad\text{for any $t > r$ and $0 \leq s \leq d$}.
\end{equation*}
Inserting this information into the spectral sequence gives $\Ext_A^i(M,N) = 0$ for any $i > d+r$. iii. By assumption and the property 3. each term $B_i \otimes_A M$ in the above exact sequence has a projective resolution of length $\leq r$. Passing to the total complex in the corresponding double complex leads to a projective resolution of length $\leq d+r$ of the $A$-module $M$.
\end{proof}

Obvious statements analogous to Proposition \ref{injdim}.ii. and iii. hold for right $A$-modules $N$ and $M$, respectively.

\section{Pro-$p$ Iwahori-Hecke algebras\label{sec:Hecke}}

We fix a locally compact nonarchimedean field $\mathfrak{F}$ (of arbitrary characteristic) with ring of integers $\mathfrak{O}$ and prime element $\pi$. We choose the valuation $\val_{\mathfrak{F}}$ on   $\mathfrak{F}$  normalized by $\val_{\mathfrak{F}}(\pi)=1$. Let $\Gp := \mathbf{G}(\mathfrak{F})$ be the group of $\mathfrak{F}$-rational points of a connected reductive group $\mathbf{G}$ over $\mathfrak{F}$ which we always assume to be $\mathfrak{F}$-split.
The residue field $\mathfrak{O}/\pi \mathfrak{O}$ of $\mathfrak{F}$ is $\mathbb{F}_q$ for some power $q = p^f$ of the residue characteristic $p$.
Let $k$ denote an arbitrary field.
Let $\mathscr{X}$, resp.\ $\mathscr{X}^1$, be the semisimple, resp.\ enlarged (\cite[4.2.16]{BT2}), building of $\Gp$. There is the canonical projection map $\pr : \mathscr{X}^1 \longrightarrow \mathscr{X}$. We fix a chamber $C$ in $\mathscr{X}$ as well as a hyperspecial vertex $x_0$ of $C$. The stabilizer of $x_0$ in $\Gp$ contains a good maximal compact subgroup $\K$ of $\Gp$. The pointwise stabilizer $\I' \subseteq \K$ of $\pr^{-1}(C)$ is an Iwahori subgroup. In fact, let $\mathbf{G}_{x_0}$ and $\mathbf{G}_C$ denote the Bruhat-Tits group schemes over $\mathfrak{O}$ whose $\mathfrak{O}$-valued points are $\K$ and $\I'$, respectively. (Note that $\mathbf{G}_{x_0}(\Corps) = \mathbf{G}_C(\Corps) = \Gp$). Their reductions over the residue field $\mathbb{F}_q$ are denoted by $\overline{\mathbf{G}}_{x_0}$ and $\overline{\mathbf{G}}_C$. By \cite{Tit} 3.4.2, 3.7, and 3.8 we have:
\begin{itemize}
  \item[--] $\overline{\mathbf{G}}_{x_0}$ is connected reductive and $\mathbb{F}_q$-split.
  \item[--] The natural homomorphism $\overline{\mathbf{G}}_C \longrightarrow \overline{\mathbf{G}}_{x_0}$ has a connected unipotent kernel and maps the connected component $\overline{\mathbf{G}}_C^\circ$ onto a Borel subgroup $\overline{\mathbf{B}}$ of $\overline{\mathbf{G}}_{x_0}$. Hence $\overline{\mathbf{G}}_C / \overline{\mathbf{G}}_C^\circ$ is isomorphic to a subgroup of $\mathrm{Norm}_{\overline{\mathbf{G}}_{x_0}}(\overline{\mathbf{B}})
      /\overline{\mathbf{B}}$. But $\mathrm{Norm}_{\overline{\mathbf{G}}_{x_0}}(\overline{\mathbf{B}}) = \overline{\mathbf{B}}$ and therefore $\overline{\mathbf{G}}_C = \overline{\mathbf{G}}_C^\circ$.
 \end{itemize}

It implies ${\mathbf{G}}_C^\circ(\Oo)={\mathbf{G}}_C(\Oo) = \I'$ and ${\mathbf{G}}_{x_0}^\circ(\Oo)={\mathbf{G}}_{x_0}(\Oo) = \Kk$.

Let $\overline{\mathbf{N}}$ denote the unipotent radical of $\overline{\mathbf{B}}$ and $\overline{\mathbf{T}}$ its Levi subgroup. We put
\begin{equation*}
    \K_1 := \Ker \big(\mathbf{G}_{x_0}(\mathfrak O) \xrightarrow{\; \pr \;} \overline{\mathbf{G}}_{x_0} (\mathbb{F}_q) \big) \quad\textrm{and}\quad \I := \{g \in \K : \pr(g) \in \overline{\mathbf{N}}(\mathbb{F}_q) \}
\end{equation*}
and obtain the chain
\begin{equation*}
    \K_1 \subseteq \I \subseteq \I' \subseteq \K
\end{equation*}
of compact open subgroups in $\Gp$ such that
\begin{equation*}
    \K/\K_1 = \G := \overline{\mathbf{G}}_{x_0} (\mathbb{F}_q) \supseteq \I/\K_1 = \N := \overline{\mathbf{N}}(\mathbb{F}_q) \ .
\end{equation*}
The subgroup $\I$ is pro-$p$ and is called the pro-$p$-Iwahori subgroup. It is a maximal  pro-$p$-subgroup in $\Kk$. The quotient $\Iw/\I$ identifies with $ \overline{\mathbf{T}}(\mathbb{F}_q)$.

The compact induction $\mathbf{X} := \ind_\I^\Gp(1)$ of the trivial $\I$-representation over $k$ is a smooth representation of $\Gp$ in a $k$-vector space. The pro-$p$ Iwahori-Hecke algebra is defined to be $\Hh := \End_{k[\Gp]}(\mathbf{X})^{\mathrm{op}}$. It is a $k$-algebra. We often will identify $\Hh$, as a right $\Hh$-module, via the map
\begin{align*}
    \Hh & \xrightarrow{\; \cong \;} \ind_\I^\Gp(1)^\I \subseteq \mathbf{X} \\
    h & \longmapsto (\mathrm{char}_\I) h
\end{align*}
(where $\mathrm{char}_\I \in \ind_\I^\Gp(1)$ denotes the characteristic function of $\I$) with the submodule $\ind_\I^\Gp(1)^\I$ of $\I$-fixed vectors in $\ind_\I^\Gp(1)$. If we replace $\I$ by $\I'$ in the latter, we define the Iwahori-Hecke algebra $\Hh'$: it identifies
with the submodule $\ind_{\I'}^\Gp(1)^{\I'}$ of $\I'$-fixed vectors in the compact induction $\mathbf X':=\ind_{\I'}^\Gp(1)$.

We recall that, for any smooth $\Gp$-representation $V$, the subspace $V^\I = \Hom_{k[\Gp]}(\ind_\I^\Gp(1),V)$ of $\I$-fixed vectors in $V$ naturally is a left $\Hh$-module. Likewise $V^{\I'}$ is an $\Hh'$-module.

\section{The universal resolution\label{sec:univ}}

\subsection{\label{subsec:univreso-representations}}

We consider  the semisimple Bruhat-Tits building $\mathscr{X}$ of $\Gp$. We recall (cf.\ \cite{SS} I.1-2 for a brief overview) that $\mathscr{X}$ is (the topological realization of) a $\Gp$-equivariant polysimplicial complex of dimension equal to the semisimple rank $d$ of $\Gp$. The (open) polysimplices are called facets and the $d$-dimensional, resp.\ zero dimensional, facets chambers, resp.\ vertices. Associated with each facet $F$ is, in a $\Gp$-equivariant way, a smooth affine $\mathfrak{O}$-group scheme $\mathbf{G}_F$ whose general fiber is $\mathbf{G}$ and such that $\mathbf{G}_F(\mathfrak{O})$ is the pointwise stabilizer in $\Gp$ of $\pr^{-1}(F)$ (denoted by $\mathscr{G}_{\pr^{-1}(F)}$ in \cite[3.4.1]{Tit}). Its connected component is denoted by $\mathbf{G}_F^\circ$ so that the reduction $\overline{\mathbf{G}}_F^\circ$ over $\mathbb{F}_q$ is a connected smooth algebraic group. The subgroup $\mathbf{G}_F^\circ(\mathfrak{O})$ of $\Gp$ is compact open. Let
\begin{equation*}
    \I_F := \{ g \in \mathbf{G}_F^\circ(\mathfrak{O}) :( g \mod \pi) \in\ \textrm{unipotent radical of $\overline{\mathbf{G}}_F^\circ$} \}.
\end{equation*}

The $\I_F$ are compact open pro-$p$ subgroups in $\Gp$ which satisfy $\I_C = \I$, $\I_{x_0} = \K_1$,
\begin{equation}\label{pr1}
    g\I_F g^{-1} = \I_{gF} \qquad\textrm{for any $g \in \Gp$},
\end{equation}
and
\begin{equation}\label{pr2}
    \I_{F'} \subseteq \I_F \qquad\textrm{whenever $F' \subseteq \overline{F}$}.
\end{equation}

Let  $\Pp_F^\dagger$ denote the stabilizer in $\Gp$ of the facet $F$. For $g \in \Pp_F^\dagger$, set  $\epsilon_F(g) = +1$, resp.\ $-1$, if $g$ preserves, resp.\ reverses, a given orientation of $F$.
Note that if $F$ is contained in the closure $\overline{C}$ of $C$, then $\I_F \subseteq \I \subseteq \I' \subseteq \Pp_F^\dagger$ with $\I_F$ being normal in $\Pp_F^\dagger$.

\begin{lemm}\label{I-orient}
Let $F'$ be any facet of $\mathscr{X}$. The elements in $\I' \cap \Pp_{F'}^\dagger$ fix $F'$  pointwise; in particular, we have $\epsilon_{F'}| (\I' \cap \Pp_{F'}^\dagger) = 1$.
\end{lemm}
\begin{proof}
Any $h \in \I' \cap \Pp_{F'}^\dagger$  stabilizes the facet $F'$ and fixes pointwise the chamber $C$. We pick points $y \in C$ and $z \in F'$, and we choose an apartment $A$ of $\mathscr{X}$ which contains $C$ and $F'$. Then $A$ also contains the geodesics $[y z]$ and $[y h(z)] = h([y z])$. Since $C$ is open in $A$ both geodesics meet $C$ in small (half open) intervals and therefore necessarily are equal. It follows that $h(z) = z$. This shows that $h$ fixes $F'$ pointwise.
\end{proof}

Let $\mathbf V$ be a smooth $k$-representation of $\Gp$.
Properties \eqref{pr1} and \eqref{pr2} imply that the family $\{\mathbf{V}^{\I_F}\}_F$ of subspaces of $\I_F$-fixed vectors in $\mathbf{V}$ forms a $\Gp$-equivariant coefficient system $\cV$ on $\mathscr{X}$.
The associated (cf. \cite{SS} II.2) augmented oriented chain complex
\begin{equation}\label{f:chain-complexV}
    0 \longrightarrow C_c^{or} (\mathscr{X}_{(d)}, \cV) \xrightarrow{\;\partial\;} \ldots \xrightarrow{\;\partial\;} C_c^{or} (\mathscr{X}_{(0)}, \cV) \xrightarrow{\;\epsilon\;} \mathbf{V} \longrightarrow 0
\end{equation}
is a complex of $\Gp$-representations.

\begin{rema}\label{rema:about-exactness}

1.  The complex \eqref{f:chain-complexV} is exact when $\mathbf V$ is one of the universal modules $\mathbf X$ and $\mathbf X'$. This is the ``universal'' part of the proof of \cite{SS} Thm.\ II.3.1 which works for completely arbitrary coefficient rings (instead of the field of complex numbers):  the arguments in loc.\ cit. Step 2 with $T := \Gp/\I$, resp.\ $T := \Gp/\I'$, and Step 3 with a special vertex of the chamber $g_0C$ go through literally.


2. If $k$ is the field of complex coefficients and $\mathbf V$ is a representation generated by its $\I$-invariant subspace, then
\eqref{f:chain-complexV} is exact. This is the level zero case of \cite[Theorem II.3.1]{SS}.


3. Suppose that $k$ has characteristic $p$ and  $\Gp={\rm GL}_2({\mathbb Q}_p)$.
Let  $\mathbf V$ be a representation of $\Gp$ generated by its $\I$-invariant subspace and with scalar action of the center.
Then \eqref{f:chain-complexV} is exact if and only if $\mathbf V^{\Kk_1}$ is equal  to the sub-$\Kk$-representation of $\mathbf V$ generated by $\mathbf V^{\I}$. It is a corollary of the main result of \cite{Inv}
 and of the exactness of the complex in the universal case as explained in \cite[6.3]{OS}.
For example, \eqref{f:chain-complexV} is exact if $\mathbf V$ is the trivial representation of $\Gp$, but it is not exact in general: if $\mathbf V$ is the supersingular representation of $\Gp$ with trivial central character described in \cite[Theorem 8.6]{BreuilNY}, then
$\mathbf V^{\K_1}$ is a  representation of $\Kk$ which is not semisimple whereas its socle is the
sub-$\Kk$-representation  generated by the $2$-dimensional space $\mathbf V^{\I}$.

\end{rema}

If we pass in \eqref{f:chain-complexV} to the $\I$-invariant vectors we obtain the complex of left $\Hh$-modules
\begin{equation}\label{f:I-complexV}
    0 \longrightarrow C_c^{or} (\mathscr{X}_{(d)}, \cV)^\I \xrightarrow{\;\partial\;} \ldots \xrightarrow{\;\partial\;} C_c^{or} (\mathscr{X}_{(0)}, \cV)^\I \xrightarrow{\;\epsilon\;} \mathbf V^\I \longrightarrow 0 \ .
\end{equation}
Surprisingly we will be able to establish the exactness of this complex in complete generality.

\subsection{\label{subsec:univreso-modules}}

In  Section \ref{subsec:apartment}, we will recall the definition of the standard apartment $\mathscr A$, which in particular, contains the chamber $C$. To  the smooth $k$-representation $\mathbf V$ of $\Gp$,
we associate the coefficient system  on $\mathscr{A}$  denoted by $\cV^\I$ and defined by
\begin{equation*}
    F\longmapsto \mathbf{V}^{\I_F (\I \cap \Pp_F^\dagger)} \:\textrm{ for any facet $F$ in $\mathscr A$}
\end{equation*}
with transition maps
\begin{align*}
t_{F'}^F    :\:\:\mathbf{V}^{\I_F (\I \cap \Pp_F^\dagger)} & \longrightarrow \mathbf{V}^{\I_{F'} (\I \cap \Pp_{F'}^\dagger)} \qquad\qquad\qquad\text{whenever $F' \subseteq \overline{F}$} \ . \\
    x & \longmapsto \sum_{g \in (\I \cap \Pp_{F'}^\dagger)/(\I \cap \Pp_F^\dagger)} gx\, .
\end{align*}
We consider the associated augmented oriented chain complex
\begin{equation}\label{f:apart-complexV}
    0 \longrightarrow C_c^{or} (\mathscr{A}_{(d)}, \cV^\I) \xrightarrow{\;\partial\;} \ldots \xrightarrow{\;\partial\;} C_c^{or} (\mathscr{A}_{(0)}, \cV^\I) \xrightarrow{\;\epsilon_\mathscr{A}\;} \mathbf V^\I \longrightarrow 0
\end{equation}
where the augmentation is given by
\begin{align*}
    \epsilon_\mathscr{A} : C_c^{or} (\mathscr{A}_{(0)}, \cV^\I) & \longrightarrow \mathbf V^\I \\
    \chain & \longmapsto \sum_{x \in \mathscr{A}_0} \sum_{g \in \I/(\I \cap \Pp_x^\dagger)} g\,\chain(x) \ .
\end{align*}
The following result is inspired by \cite{Bro}.

\begin{prop}\label{prop:isocomplex}
The complex \eqref{f:I-complexV} is isomorphic to the complex \eqref{f:apart-complexV}.
\end{prop}

\begin{proof}
We are going to check that restricting oriented chains gives isomorphisms
\begin{equation}\label{eq:restriction}
    C_c^{or} (\mathscr{X}_{(i)}, \cV)^\I \xrightarrow{\;\cong\;} C_c^{or} (\mathscr{A}_{(i)}, \cV^\I) \textrm{ for any } i\in \{0, ..., d\},
\end{equation}
which are compatible with the differential maps. By the definition of the action of $\I$ on the chains, the value of an oriented $\I$-invariant $i$-chain at an oriented facet $(F,c)\in\Aa_{(i)}$ lies in  $ \mathbf{V}^{\I_F (\I \cap \Pp_F^\dagger)}$. Therefore
\eqref{eq:restriction} is well defined. It is bijective by the Bruhat decomposition which will be recalled in Section \ref{subsec:bruhat} and implies that  any  chamber  in $\mathscr{X}$ has an $\I$-conjugate that belongs to the standard apartment $\Aa$. Since $\I$ fixes $C$ pointwise no two different points in $\mathscr{A}$ can be equivalent under the $\I$-action. Hence for any facet $F'$ in $\mathscr{X}$ there is a unique facet $F$ in $\mathscr{A}$ such that $F' = g F$ for some $g \in \I$. Lastly, by Lemma \ref{I-orient} this remains true for oriented facets.


Now we check the compatibility of the isomorphisms \eqref{eq:restriction} with the differentials. Suppose that $i\geq 1$. Consider an oriented $\I$-invariant $i$-chain $\chain\in     C_c^{or} (\mathscr{X}_{(i)}, \cV)^\I$. The image of $\chain$ by the differential $\partial$ in the complex \eqref{f:I-complexV} is an $\I$-invariant $(i-1)$-chain determined by
its  values  at $(F',c')\in \Aa_{(i-1)}$ and
$$
\partial (\chain)(F', c')= \sum_{ F\in \Aa_{i}}\sum_{\stackrel{ \: F'\subset g\overline F}{g \in \I/(\I\cap \Pp_F^\dagger)}} \chain (gF, c )
$$
where $c$ is such that it induces the orientation $c'$  as in \cite[II.1]{SS}. Since $\chain$ is $\I$-invariant and $g$ does not change the orientation (Lemma \ref{I-orient}), we have
$$
\chain (g F, c )=\chain (g (F, c) ) = g(\chain(F, c)).
$$
Moreover, by uniqueness of the facet in the $\I$-orbit of $\F'$ belonging to $\Aa$, the $g$'s in the previous sum belong to $ \I\cap\Pp_{F'}^\dagger$. Therefore
\begin{equation}\label{f:diff}
\partial (\chain)(F', c')= \sum_{ F\in \Aa_{i}, F'\subset \overline F }\:\:\sum_{g \in  (\I\cap \Pp_{F'}^\dagger) /(\I\cap \Pp_F^\dagger) } g (\chain (F, c ))= \sum_{ F\in \Aa_{i}, F'\subset \overline F } t_{F'}^F(\chain (F, c )).
\end{equation}
where $\Aa_{i}$ denotes the set of $i$-dimensional facets in $\Aa$ and where $c$ induces the orientation $c'$. Likewise, it gets easily checked that the composition of the augmentation map $\epsilon$ in the complex \eqref{f:I-complex} with
the isomorphism \eqref{eq:restriction} for $i=0$ yields the augmentation map $\epsilon_\Aa$ announced in the proposition.
\end{proof}

\begin{theo}\label{I-resolutionV}
The complexes \eqref{f:I-complexV} and \eqref{f:apart-complexV} are exact.
\end{theo}
\begin{proof}
By Proposition \ref{prop:isocomplex}, it suffices to show that \eqref{f:apart-complexV}  is exact.
Let $F$ be a facet  in the standard apartment. By Proposition \ref{prop:closest}, there is a unique chamber $C(F)$ which contains $F$ in its closure and
which is closest to the chamber $C$ in the sense of  the \emph{gallery distance}. Furthermore, it satisfies
\begin{equation}\label{f:C(F)}
    \I_F (\I \cap \Pp_F^\dagger) = \I_{C(F)}\ \textrm{ and }\   \I \cap \Pp_F^\dagger = \I \cap \Pp_{C(F)}^\dagger\ .
\end{equation}
This implies that the transition map
\begin{equation}\label{f:local}
    t_{F}^{C(F)}\::\mathbf{V}^{\I_{C(F)}} \xrightarrow{\; = \;}\mathbf{V}^{\I_F(\I \cap \Pp_F^\dagger)}
\end{equation}
from the chamber $C(F)$ to the facet $F$ is the identity map.

In order to use this to show the exactness of \eqref{f:apart-complexV} we introduce, for any $n \geq 0$, the subcomplex $\mathscr{A}(n)$ of all facets $F$ in $\mathscr{A}$ such that $C(F)$ is of distance $\leq n$ from $C$ (or equivalently, of all facets in  the closure of the chambers at distance $\leq n$ from $C$). We have $\mathscr{A}(0) = \overline{C}$. By \eqref{f:local}, the restriction $\cV^\I | \mathscr{A}(0)$ is the constant coefficient system with value $\mathbf V^\I$:  for any $F'\subset \overline F\subset \overline C$, we have indeed $C(F)=C(F')=C$ and $t_{F'}^F\circ t_{F}^C=t_{F'}^C$ where both
$ t_{F}^C$ and  $t_{F'}^C$ are the identity map on $\mathbf V^\I$, so that $t_{F'}^F$ is  also the identity map on $\mathbf V^\I$. Hence the augmented complex
\begin{equation*}
    0 \longrightarrow C_c^{or} (\mathscr{A}(0)_{(d)}, \cV^\I) \xrightarrow{\;\partial\;} \ldots \xrightarrow{\;\partial\;} C_c^{or} (\mathscr{A}(0)_{(0)}, \cV^\I) \xrightarrow{\;\epsilon_\mathscr{A}\;} \mathbf V^\I \longrightarrow 0
\end{equation*}
is exact since $\overline{C}$ is contractible. On the other hand we obviously have
\begin{equation*}
    C_c^{or} (\mathscr{A}_{(i)}, \cV^\I) = \bigcup_{n \geq 0} C_c^{or} (\mathscr{A}(n)_{(i)}, \cV^\I) \ .
\end{equation*}
Hence for the exactness of \eqref{f:apart-complexV} it suffices to show that, for any $n \geq 1$, the relative complex
\begin{equation*}
    0 \rightarrow C_c^{or} (\mathscr{A}(n)_{(d)}, \cV^\I) / C_c^{or} (\mathscr{A}(n-1)_{(d)}, \cV^\I) \xrightarrow{\;\partial\;} \ldots \xrightarrow{\;\partial\;} C_c^{or} (\mathscr{A}(n)_{(0)}, \cV^\I) / C_c^{or} (\mathscr{A}(n-1)_{(0)}, \cV^\I) \rightarrow 0
\end{equation*}
is exact. If $Ch(n)$ denotes the set of all chambers in $\mathscr A$ of distance $n$ from $C$, then Lemma \ref{disjoint} ensures that, for $n\geq 1$, we have the disjoint decomposition
\begin{equation*}
    \mathscr{A}(n) =  \mathscr{A}(n-1)\sqcup \bigsqcup_{D \in Ch(n)}  \overline D \setminus \mathscr{A}(n-1).
\end{equation*}
Setting $\sigma\overline{D} := \overline{D} \cup  \mathscr{A}(n-1)$,  the above relative complex decomposes into the direct sum over $D \in Ch(n)$ of the relative complexes
\begin{equation}\label{Dcomplex}
   0 \rightarrow C_c^{or} (\sigma \overline{D}_{(d)}, \cV^{\I}) / C_c^{or} (\mathscr{A}(n-1)_{(d)}, \cV^{\I}) \xrightarrow{\;\partial\;} \ldots \xrightarrow{\;\partial\;} C_c^{or} (\sigma\overline{D}_{(0)}, \cV^{\I}) / C_c^{or} (\mathscr{A}(n-1)_{(0)}, \cV^{\I}) \rightarrow 0 \ .
\end{equation}
On any facet in each $\overline D \setminus \mathscr{A}(n-1)$ the coefficient system $\cV^\I$ has, by \eqref{f:local}, the constant value $\mathbf{V}^{\I_D}$, and by the same argument as in the case $n=0$, the transition maps $t_{F'}^F$ for two facets $F'\subset \overline F\subset \overline{D} \setminus \mathscr{A}(n-1)$ are the identity on  $\mathbf{V}^{\I_D}$. Consider the constant coefficient  system $\mathbf{V}^{\I_D}$ on $\sigma\overline{D} = \overline{D} \cup  \mathscr{A}(n-1)$ and for any $i\in \{0, ..., d\}$
the map
$$
C_c^{or} (\sigma \overline{D}_{(i)}, \mathbf{V}^{\I_D})\rightarrow C_c^{or} (\sigma \overline{D}_{(i)}, \cV^{\I}) / C_c^{or} (\mathscr{A}(n-1)_{(i)}, \cV^{\I})
$$
defined the following way: for an oriented $i$-chain $f\in C_c^{or} (\sigma \overline{D}_{(i)}, \mathbf{V}^{\I_D})$, its restriction to the facets in $\overline D\backslash \mathscr{A}(n-1)$ can be seen, by extension by zero, as an element of $C_c^{or} (\sigma \overline{D}_{(i)}, \cV^{\I})$ and we consider the image  of this restriction in the quotient on the right hand side. This defines a bijective map
$$
C_c^{or} (\sigma \overline{D}_{(i)}, \mathbf{V}^{\I_D})/ C_c^{or} (\mathscr{A}(n-1)_{(i)}, \mathbf V^{\I_D})\longrightarrow C_c^{or} (\sigma \overline{D}_{(i)}, \cV^{\I}) / C_c^{or} (\mathscr{A}(n-1)_{(i)}, \cV^{\I})
$$
that commutes with the differentials, so that the complex \eqref{Dcomplex} is isomorphic to
\begin{equation*}
   0 \rightarrow C_c^{or} (\sigma \overline{D}_{(d)}, \mathbf{V}^{\I_D}) / C_c^{or} (\mathscr{A}(n-1)_{(d)}, \mathbf{V}^{\I_D}) \xrightarrow{\;\partial\;} \ldots \xrightarrow{\;\partial\;} C_c^{or} (\sigma\overline{D}_{(0)}, \mathbf{V}^{\I_D}) / C_c^{or} (\mathscr{A}(n-1)_{(0)}, \mathbf{V}^{\I_D}) \rightarrow 0 \ .
\end{equation*}
It is exact since $\mathscr{A}(n-1)$ and  $\sigma\overline{D}$ are contractible by Proposition \ref{prop:contractible}.
\end{proof}

\subsection{\label{subsec:bimodules}}

In the case where $\mathbf V= \mathbf X$ is the universal representation $\mathbf X$, the subspaces $\{\mathbf{X}^{\I_F}\}_F$ are right $\Hh$-invariant so that we have a coefficient system of right $\Hh$-modules and the associated spaces of chains on $\mathscr X$ are $(\Gp,\Hh)$-bimodules. The associated augmented oriented chain complex
\begin{equation}\label{f:chain-complex}
    0 \longrightarrow C_c^{or} (\mathscr{X}_{(d)}, \cX) \xrightarrow{\;\partial\;} \ldots \xrightarrow{\;\partial\;} C_c^{or} (\mathscr{X}_{(0)}, \cX) \xrightarrow{\;\epsilon\;} \mathbf{X} \longrightarrow 0
\end{equation}
therefore is a complex of $(\Gp,\Hh)$-bimodules.

\begin{rema}\label{augmentation-split}
The augmentation $C_c^{or} (\mathscr{X}_{(0)}, \cX) \xrightarrow{\;\epsilon\;} \mathbf{X}$ has a $\Gp$-equivariant section.
\end{rema}
\begin{proof}
By Frobenius reciprocity it suffices to find an $\I$-invariant chain in $C_c^{or} (\mathscr{X}_{(0)}, \cX)$ which is mapped by the augmentation $\epsilon$  to $\mathrm{char}_\I$. We may take the chain supported on the vertex $x_0$ with value $\mathrm{char}_\I$.
\end{proof}

As in Section \ref{subsec:univreso-representations} for the complex \eqref{f:chain-complexV}, we pass in \eqref{f:chain-complex} to the $\I$-invariant vectors. We now obtain the complex of $(\Hh,\Hh)$-bimodules
\begin{equation}\label{f:I-complex}
    0 \longrightarrow C_c^{or} (\mathscr{X}_{(d)}, \cX)^\I \xrightarrow{\;\partial\;} \ldots \xrightarrow{\;\partial\;} C_c^{or} (\mathscr{X}_{(0)}, \cX)^\I \xrightarrow{\;\epsilon\;} \Hh \longrightarrow 0 \ .
\end{equation}

In this setting, Theorem \ref{I-resolutionV} gives the following.

\begin{theo}\label{I-resolution}
The complex  \eqref{f:I-complex} yields an exact resolution of the  $(\Hh,\Hh)$-bimodule $\Hh$.
\end{theo}

We now identify the structure of $(\Hh, \Hh)$-bimodule of the terms in the complex \eqref{f:I-complex}.

\subsubsection{\label{subsubsec:epsilonF}}

Let $F$ be a facet contained in $\overline C$. Extending functions on $\Pp_F^\dagger$ by zero to $\Gp$ induces a $\Pp_F^\dagger$-equivariant embedding
\begin{equation}\label{XF+}
    \mathbf{X}_F^\dagger := \ind^{\Pp_F^\dagger}_\I(1) \hookrightarrow \mathbf{X} \ .
\end{equation}
We introduce the $k$-algebra
\begin{equation*}
    \Hh_F^\dagger := \End_{k[\Pp_F^\dagger]}(\mathbf{X}_F^\dagger)^{\mathrm{op}} = \ind^{\Pp_F^\dagger}_\I(1)^\I \ .
\end{equation*}
It is naturally a subalgebra of $\Hh$ via the extension by zero  embedding $\ind^{\Pp_F^\dagger}_\I(1)^\I \hookrightarrow \ind^{\Gp}_\I(1)$. Alternatively, in terms of the endomorphism rings and using that $\ind^{\Gp}_\I(1) = \ind^{\Gp}_{\Pp_F^\dagger}(\ind^{\Pp_F^\dagger}_\I(1))$ by the transitivity of induction, the inclusion $\Hh_F^\dagger \hookrightarrow \Hh$ is given by $h \mapsto \ind^{\Gp}_{\Pp_F^\dagger}(h)$.

We are  going to consider $(\Pp_F^\dagger, \Hh_F^\dagger)$-bimodules and their twists by the character $\epsilon_F$, defined in Section \ref{subsec:univreso-representations}, as follows.

\begin{itemize}
\item For any $(\Pp_F^\dagger, \Hh_F^\dagger)$-bimodule $\m$ we denote by $\epsilon_F\otimes \m$ the space $\m$ endowed with the structure of a  $(\Pp_F^\dagger, \Hh_F^\dagger)$-bimodule where the action of $\Pp_F^\dagger$ is twisted by the character $\epsilon_F$.

\item By viewing $\Hh_F^\dagger$ as the space of $\I$-bi-invariant functions on $\Pp_F^\dagger$ we see that the product of functions defines an involution $j_F: h \mapsto \epsilon_F h$ of the vector space $\Hh_F^\dagger$. The computation
\begin{equation*}
    ((\epsilon_F h_1) \cdot (\epsilon_F h_2)) (g') = \sum_{g \in \Pp_F^\dagger / \I} \epsilon_F(g)h_1(g) \epsilon_F(g^{-1}g')h_2(g^{-1}g') = (\epsilon_F(h_1 \cdot h_2))(g')
\end{equation*}
    shows that $j_F$ in fact is an automorphism of the algebra $\Hh_F^\dagger$. For any left, resp.\ right, $ \Hh_F^\dagger$-module $\m$ we denote by $(\epsilon_F)\m$, resp.\ $\m (\epsilon_F)$,  the space $\m$ endowed with the structure of a left, resp.\ right, $\Hh_F^\dagger$-module where the action of $\Hh_F^\dagger$ is composed  with $j_F$. Of course, with $\m$ also $\m (\epsilon_F)$ is a $(\Pp_F^\dagger, \Hh_F^\dagger)$-bimodule.

\end{itemize}

\begin{lemm}
The $(\Pp_F, \Hh_F^\dagger)$-bimodules $\epsilon_F \otimes \mathbf{X}_F^\dagger$ and
$\mathbf{X}_F^\dagger (\epsilon_F)$ are isomorphic.
\label{lemma:twist}
\end{lemm}
\begin{proof}
One easily checks that the map
\begin{align*}
    \epsilon_F \otimes \ind^{\Pp_F^\dagger}_\I(1) & \longrightarrow \ind^{\Pp_F^\dagger}_\I(1) (\epsilon_F) \\
    f & \longmapsto \epsilon_F f
\end{align*}
yields the required isomorphism.
\end{proof}

\begin{lemm}\label{F-map}
The map
\begin{align}\label{f:F-map}
    \mathbf{X}_F ^\dagger\otimes_{\Hh_F^\dagger} \Hh & \longrightarrow \mathbf{X}^{\I_F} \\
    f \otimes h & \longmapsto h(f) \nonumber
\end{align}
is a well defined homomorphism of $(\Pp_F^\dagger, \Hh)$-bimodules.
\end{lemm}
\begin{proof} Denote by $i_F$ the map in question.
Since
\begin{equation*}
    i_F(h_0(f) \otimes h) = h(h_0(f)) = (h \circ \ind^{\Gp}_{\Pp_F^\dagger}(h_0))(f) = i_F(f \otimes  \ind^{\Gp}_{\Pp_F^\dagger}(h_0)h)
\end{equation*}
the map $i_F$ is well defined as a map into $\mathbf{X}$. Since any $h \in \Hh$ is a $\Gp$-equivariant endomorphism of $\mathbf{X}$ we have
\begin{equation*}
    \upsilon_F(g(f \otimes h)) = i_F(gf \otimes h) = h(gf) = gh(f) = g  i_F(f \otimes h) \ .
\end{equation*}
This shows the $\Pp_F^\dagger$-equivariance of the map. But $i_F$ being normal in $\Pp_F^\dagger$ acts trivially on the left hand side. Hence it also shows that the image of $i_F$ is contained in $\mathbf{X}^{\I_F}$. The $\Hh$-equivariance of $i_F$ is obvious.
\end{proof}

\begin{exam}
$\I = \I_C$ is a normal subgroup of $\Pp_C^\dagger$. Hence $\Hh_C^\dagger = k[\Pp_C^\dagger/\I]^{\mathrm{op}} \cong k[\Pp_C^\dagger/\I]$ is the group algebra of the discrete group $\Pp_C^\dagger/\I$, and $\mathbf{X}_C ^\dagger= k[\Pp_C^\dagger/\I]$ as a right module over itself. In particular, the map $i_C$ is an isomorphism for trivial reasons.
\end{exam}

In Proposition \ref{prop:free}   we will show that $\Hh$ is free as a left as well as a right $\Hh_F^\dagger$-module. This yields the following result.

\begin{prop}\label{H-H}
$\ind_{\Pp_F^\dagger}^\Gp(\epsilon_F \otimes \mathbf{X}_F^\dagger \otimes_{\Hh_F^\dagger} \Hh)^\I = \Hh (\epsilon_F) \otimes_{\Hh_F^\dagger} \Hh$ as $(\Hh,\Hh)$-bimodules. In particular, it is free as a left as well as a right $\Hh$-module.
\end{prop}
\begin{proof}
Since the functor $\ind_{\Pp_F^\dagger}^\Gp$ commutes with arbitrary direct sums, the fact that $\Hh$ is free as a left $\Hh_F^\dagger$-module implies the first identity in the chain of isomorphisms of $(\Gp, \Hh)$-bimodules
\begin{align*}
    \ind_{\Pp_F^\dagger}^\Gp(\epsilon_F \otimes \mathbf{X}_F^\dagger \otimes_{\Hh_F^\dagger} \Hh) & = \ind_{\Pp_F^\dagger}^\Gp(\epsilon_F \otimes \mathbf{X}_F^\dagger) \otimes_{\Hh_F^\dagger} \Hh \\
    & = \ind_{\Pp_F^\dagger}^\Gp( \mathbf{X}_F^\dagger (\epsilon_F) ) \otimes_{\Hh_F^\dagger} \Hh \\
&= \ind_{\Pp_F^\dagger}^\Gp( \mathbf{X}_F^\dagger) (\epsilon_F) \otimes_{\Hh_F^\dagger} \Hh \\
&= \mathbf{X} (\epsilon_F) \otimes_{\Hh_F^\dagger} \Hh
\end{align*}
where the second one is given by Lemma \ref{lemma:twist}.
Passing to the $\I$-invariants on the right hand side commutes with the tensor product since $\Hh$ is free over $\Hh_F^\dagger$. It gives the  announced  isomorphism.  Since $\Hh$ (resp. $\Hh (\epsilon_F)$)  is a free left (resp. right) $\Hh_F^\dagger$-module (see Remark \ref{rema:libre}), we deduce that $\Hh (\epsilon_F)\otimes_{\Hh_F^\dagger} \Hh$ is free as a left as well as a right $\Hh$-module.
\end{proof}

\begin{lemm}\phantomsection\label{lemma:isom}
\begin{itemize}
\item[i.] The map \eqref{f:F-map} is an isomorphism.
\item[ii.] $\ind_{\Pp_F^\dagger}^\Gp(\epsilon_F \otimes \mathbf{X}^{\I_F})^\I = \Hh (\epsilon_F) \otimes_{\Hh_F^\dagger} \Hh$ as $(\Hh,\Hh)$-bimodules. In particular, it is free as a left as well as a right $\Hh$-module.
\end{itemize}
\end{lemm}
\begin{proof}
The first assertion will be shown in Section \ref{subsec:bijproof}. The second assertion then follows from the first and Proposition \ref{H-H}.
\end{proof}

\subsubsection{\label{subsubsec:iso}}

Let  $i\in \{0, ..., d\}$ and let $\mathscr{X}_i$ denote the set of $i$-dimensional facets. Let $F$ be one of them. Choose an orientation $(F,c)$ for $F$. There is a natural $(\Pp_F^\dagger,\Hh)$-equivariant map
$$
\epsilon_F \otimes \mathbf{X}^{\I_F}\rightarrow C_c^{or}(\mathscr{X}_{(i)}, \cX)
$$
that sends an element $x\in \mathbf{X}^{\I_F}$ onto the only $i$-chain with support $\{F\}$ and value $x$ at $(F,c)$. It induces an injective $(\Gp,\Hh)$-equivariant map
$$
\ind_{\Pp_F^\dagger}^\Gp(\epsilon_F \otimes \mathbf{X}^{\I_F})\rightarrow C_c^{or}(\mathscr{X}_{(i)}, \cX)
$$
the image of which is the space of $i$-chains supported on the $\Gp$-conjugates of $F$.

The group $\Gp$ acts on the set $\mathscr{X}_i$ of $i$-dimensional facets with finitely many orbits. Accordingly, the $(\Gp,\Hh)$-bimodule $C_c^{or}(\mathscr{X}_{(i)}, \cX)$ of oriented $i$-chains breaks up into a finite direct sum of $(\Gp,\Hh)$-bimodules of the form $\ind_{\Pp_F^\dagger}^\Gp(\epsilon_F \otimes \mathbf{X}^{\I_F})$. Since $\Gp$ acts transitively on the chambers we may assume that $F \subseteq \overline{C}$ and apply the results of Section {\ref{subsubsec:epsilonF}}.

We fix a (finite) set of representatives $\mathscr{F}_i$ for the $\Gp$-orbits in $\mathscr{X}_i$ such that every member of every set $\mathscr{F}_i$ is contained in $\overline{C}$. Theorem\ \ref{I-resolution} and Lemma \ref{lemma:isom}.ii  together imply  the following theorem.

\begin{theo}\label{theo:freeresolution}
The complex \eqref{f:I-complex} is an exact sequence of $(\Hh,\Hh)$-bimodules of the form
\begin{equation}\label{f:H-H-bimod}
    0 \longrightarrow \bigoplus_{F \in \mathscr{F}_d} \Hh (\epsilon_F) \otimes_{\Hh_F^\dagger} \Hh \longrightarrow \ldots \longrightarrow \bigoplus_{F \in \mathscr{F}_0} \Hh (\epsilon_F) \otimes_{\Hh_F^\dagger} \Hh \longrightarrow \Hh \longrightarrow 0 \ .
\end{equation}
It yields a free resolution of $\Hh$ as a left as well as a right $\Hh$-module.
\end{theo}

\begin{rema}\label{rema:sym}
There is a natural isomorphism of $(\Hh,\Hh)$-bimodules $\Hh (\epsilon_F) \otimes_{\Hh_F^\dagger} \Hh \simeq \Hh \otimes_{\Hh_F^\dagger}  (\epsilon_F) \Hh$.
\end{rema}

\subsection{\label{subsec:ourtheorem}}

We denote by $r$ the rank of the center of $\Gp$. The number $d^1 := d + r$ is the rank of the group $\Gp$.

\begin{theo}\label{theo:Gorenstein}
The injective dimension of $\Hh$  as a left as well as a right $\Hh$-module  is bounded above by the rank of the group $\Gp$.
\end{theo}
\begin{proof}
In Propositions \ref{prop:free}, \ref{F-injdim}  (see also Remark \ref{rema:libre}) we will verify the assumptions of Section \ref{sec:gene} for the complexes \eqref{f:H-H-bimod}. Hence Proposition \ref{injdim} applies.
\end{proof}

There is a natural action of the finite Weyl group $\W_0$ (see Section \ref{subsec:apartment}) on the finite torus $ \overline{\mathbf{T}}(\mathbb{F}_q)$
and on the groups of its $k$-characters. Denote by $\Gamma$ the set of orbits under this latter action.
Suppose  that the cardinality of $\T(\mathbb F_q)$, or equivalently $q-1$, is invertible in $k$. Then to any  $\gamma\in \Gamma$ one can associate  a  central idempotent $\varepsilon _\gamma\in \Hh$ as in \cite[Corollary 4]{Vig}. In particular, in the case where $\gamma=1$ is the orbit of the trivial character, then the algebra $\Hh \varepsilon _{ 1}$ with unit $ \varepsilon _{1}$ identifies with the Iwahori-Hecke algebra $\Hh'$ of $\Gp$.

\begin{coro} \label{coro:Iwahori}
Suppose  that $q-1$ is invertible in $k$. Then the injective dimension of $\Hh\varepsilon _\gamma$ as a left as well as a right $\Hh\varepsilon _\gamma$-module is $\leq d^1$. In particular,   the injective dimension of $\Hh'$ as a left as well as a right $\Hh'$-module is $\leq d^1$.
\end{coro}

In fact, in the case of $\Hh'$ the restriction on the characteristic of $k$ in the above Corollary \ref{coro:Iwahori} is unnecessary as we will explain in the next section.

In Corollary \ref{characters}.ii we will see that the above upper bounds on the self-injective dimensions are sharp if the group $\Gp$ is semisimple.

\subsection{About the Iwahori-Hecke algebra}\label{subsec:Iwahori}

We apply the arguments of \ref{subsec:bimodules} to
the representation $\mathbf X'$ defined in Section \ref{sec:Hecke}. In this  case, the subspaces $\{(\mathbf{X'})^{\I_F}\}_F$ are right $\Hh'$-invariant.   Applying Theorem \ref{I-resolutionV} to $\mathbf{V} = \mathbf{X}'$ we obtain the exact sequence of $(\Hh, \Hh')$-bimodules and $\I'$-representations
\begin{equation}\label{f:chain-complex-iwahori}
    0 \longrightarrow C_c^{or} (\mathscr{X}_{(d)}, \cX')^\I \xrightarrow{\;\partial\;} \ldots \xrightarrow{\;\partial\;} C_c^{or} (\mathscr{X}_{(0)}, \cX')^\I \xrightarrow{\;\epsilon\;} (\mathbf{X}')^\I \longrightarrow 0 \ .
\end{equation}

\begin{lemm}\label{lemma:iwahori-1}
Each term in \eqref{f:chain-complex-iwahori} is $\I'$-invariant.
\end{lemm}

\begin{proof}
Recall that $\Tp^0$ fixes the apartment $\Aa$ pointwise (see \ref{subsec:apartment}), that $\Tp^0$ normalizes $\I$, and that $\I' = \Tp^0 \I$.
For any chamber $D$ in $\Aa$, one easily checks that
$\I_D \backslash \Gp/ \I'= \Tp^0\I_D \backslash \Gp/ \I'$ (see Remark \ref{rema:transitivity} below for example). This implies $(\mathbf X')^{\I_D}=(\mathbf X')^{\Tp^0\I_D}$. For $D=C$ we, in particular, obtain $(\mathbf X')^{\I}=(\mathbf X')^{\I'}$.

Now let $f \in C_c^{or} (\mathscr{X}_{(i)}, \cX')^\I$, for any $0 \leq i \leq d$, be any oriented $i$-chain which is fixed by $\I$. We have to show that $f$ also is fixed by any $t \in \Tp^0$. Let $(F,c) \in \mathscr{X}_{(i)}$ be arbitrary and pick an element $g \in \I$ such that $(F_0,c_0) := g^{-1}(F,c) \in \mathscr{A}_{(i)}$. We have
\begin{equation*}
    t(F,c) = tg(F_0,c_0) = tgt^{-1}(F_0,c_0) \qquad\text{with $tgt^{-1} \in \I$}
\end{equation*}
and hence
\begin{equation*}
    f(t(F,c)) = f(tgt^{-1}(F_0,c_0)) = tgt^{-1} f((F_0,c_0)) = tgt^{-1} f(g^{-1}(F,c)) = tgt^{-1}g^{-1} f((F,c)) \ .
\end{equation*}
On the other hand, using Proposition \ref{prop:closest}, we obtain
\begin{equation*}
    g^{-1} f((F,c)) = f((F_0,c_0)) \in (\mathbf{X'})^{\I_{F_0}(\I \cap \Pp_{F_0}^\dagger)} = (\mathbf{X'})^{\I_{C(F_0)}} = (\mathbf{X'})^{\Tp^0 \I_{C(F_0)}} \ .
\end{equation*}
It follows that $gt^{-1}g^{-1} f((F,c)) = f((F,c))$ and therefore that $f(t(F,c)) = t f((F,c))$.
\end{proof}

We therefore have the exact resolution of $\Hh'$ as an   $(\Hh',\Hh')$-bimodule given by
\begin{equation}\label{f:I-complex-iwahori}
    0 \longrightarrow C_c^{or} (\mathscr{X}_{(d)}, \cX')^{\I'} \xrightarrow{\;\partial\;} \ldots \xrightarrow{\;\partial\;} C_c^{or} (\mathscr{X}_{(0)}, \cX')^{\I'} \xrightarrow{\;\epsilon\;} \Hh' \longrightarrow 0 \ .
\end{equation}
To further compute the $(\Hh',\Hh')$-bimodules in this resolution we define, analogously as in \ref{subsubsec:epsilonF}, the representation $\mathbf{X}'^{\dagger}_F := \ind^{\Pp_F^\dagger}_{\I'}(1) \hookrightarrow \mathbf{X'} \ $ and the corresponding
subalgebra
\begin{equation*}
    {\Hh}'^{\dagger}_F:= \End_{k[\Pp_F^\dagger]}(\mathbf{X}'^\dagger_F)^{\mathrm{op}} = \ind^{\Pp_F^\dagger}_{\I'}(1)^{\I'}
\end{equation*}
of $\Hh'$ which, by a similar argument as in Lemma \ref{lemma:iwahori-1}, coincides as a vector space with $\ind^{\Pp_F^\dagger}_{\I'}(1)^{\I}.$ In  Section \ref{subsec:bijproof}, we give the necessary arguments to adapt the proof of Lemma \ref{lemma:isom} to  $\mathbf X'$. In particular, the analog of \eqref{f:F-map} is the isomorphism of $(\Pp_F^\dagger, \Hh')$-bimodules
\begin{equation}\label{f:F-map-iwahori}
\mathbf{X}'^\dagger_F\otimes_{\Hh'^\dagger_F} \Hh'  \xrightarrow{\;\cong\;} (\mathbf{X}')^{\I_F},\:\:
    f \otimes h  \longmapsto h(f).
\end{equation}
and the argument of Proposition \ref{H-H} goes through: we obtain that
\begin{equation*}
    \ind_{\Pp_F^\dagger}^\Gp(\epsilon_F \otimes (\mathbf{X'})^{\I_F})^{\I'} = \Hh' (\epsilon_F) \otimes_{\Hh'^\dagger_F} \Hh'
\end{equation*}
as $(\Hh',\Hh')$-bimodules, and it is free as a left as well as a right $\Hh'$-module. Here
$\Hh' (\epsilon_F)$ denotes the space $\Hh'$ endowed with the structure of $(\Hh',\Hh'^\dagger_F)$-bimodule where the action of $\Hh'^\dagger_F$  on the right is composed  with $ h \mapsto \epsilon_F h$  (it is a well-defined involutive automorphism of $\Hh'^\dagger_F$
by Lemma \ref{I-orient}). Therefore (see \ref{subsubsec:iso}), the complex \eqref{f:I-complex-iwahori} is an exact sequence of $(\Hh',\Hh')$-bimodules of the form
\begin{equation}\label{f:H'-H'-bimod}
    0 \longrightarrow \bigoplus_{F \in \mathscr{F}_d} \Hh '(\epsilon_F) \otimes_{\Hh'^\dagger_F} \Hh' \longrightarrow \ldots \longrightarrow \bigoplus_{F \in \mathscr{F}_0} \Hh' (\epsilon_F) \otimes_{\Hh'^\dagger_F} \Hh' \longrightarrow \Hh '\longrightarrow 0 \ .
\end{equation}
It yields a free resolution of $\Hh'$ as a left as well as a right $\Hh'$-module. Together with Proposition \ref{F-injdim} we then deduce the analog of Theorem \ref{theo:Gorenstein}.

\begin{theo}\label{theo:Gorenstein-iwahori}
The injective dimension of $\Hh'$  as a left as well as a right $\Hh'$-module is bounded above by the rank of the group $\Gp$.
\end{theo}

\section{Exercises in Bruhat-Tits theory\label{BTtheory}}

\subsection{\label{subsec:apartment}}

We follow \cite[I.1]{SS}. Fix a maximal $\mathfrak F$-split torus $\Tp$ in $\Gp$. Consider the associated root data
$(\Root, X^*({\Tp}), \Coroot, X_*({\Tp}))$ where  $X^*({\Tp})$ and $X_*({\Tp})$ denote respectively the set of algebraic characters and cocharacters
of $\Tp$, and similarly, let $X^*({\rm Z})$ and $X_*({\rm Z})$ denote respectively the set of algebraic characters and cocharacters
of the connected center ${\rm Z}$ of $\Gp$. We note that this root system is reduced by our assumption that the group $\mathbf{G}$ is $\mathfrak F$-split. Denote by
\begin{equation*}
    \lp \, .\,,.\, \rp:X_*({\Tp})\times X^*({\Tp})\rightarrow \Z
\end{equation*}
the natural perfect pairing. Its $\R$-linear extension is also denoted by
$\lp \, .\,,.\, \rp$.
The vector space
\begin{equation*}
    \mathbb R\otimes _{\mathbb Z}(X_*({\Tp})/X_*({\rm Z})) \ ,\ \text{resp.}\ \mathbb R\otimes _{\mathbb Z} X_*({\Tp}) \ ,
\end{equation*}
considered as an affine space on itself identifies with the standard apartment $\mathscr{A}$, resp.\ $\mathscr{A}^1$ of the building $\mathscr{X}$, resp.\ $\mathscr{X}^1$. Any  root $\root$ takes value zero on $X_*({\rm Z})$ so that $\root$ defines a function on $\mathscr{A}$ which we denote  by $x\mapsto  \root(x) $. For any subset $Y$ of $\mathscr{A}$, we write $\root(Y)\geq 0$ if $\root$ takes nonnegative values on $Y$. To $\root$ is also associated a coroot $\coroot\in\Coroot$ such that $\lp\coroot, \alpha   \rp=2$ and a reflection on  $\mathscr{A}$ defined by
\begin{equation*}
    s_\root: x\mapsto x- \root( x)\coroot \mod X_*({\rm Z})\otimes_{\mathbb Z}\mathbb R\ .
\end{equation*}

The subgroup of the transformations of $\mathscr{A}$  generated by these reflections identifies with the finite Weyl group $\W_0$, defined to be the quotient  by $\Tp$ of its normalizer  $N_\Gp(\Tp)$ in $\Gp$.

To an element $g\in \Tp$ corresponds a vector $\nu(g)\in \mathbb R\otimes _{\mathbb Z}X_*({\Tp})$ defined by
\begin{equation*}
    \lp\nu(g),\, \chi\rp =-\val_{\mathfrak F}(\chi(g))  \qquad \textrm{for any } \chi\in X^*(\Tp).
\end{equation*}
The kernel of $\nu$ is the maximal compact subgroup of $\Tp$. The quotient of $\Tp$ by the kernel of $\nu$ is a free abelian group $\Lambda$ with rank equal to $\rm dim(\Tp)$, and $\nu$ induces an isomorphism $\Lambda \cong X_*(\Tp)$. The group $\Lambda$ acts by translation on $\Aa$ via $\nu$.
The extended Weyl group $\W$ is defined to be the quotient of $N_\Gp(\Tp)$ by the kernel of $\nu$. The actions of $\W_0$ and $\Lambda$ combine into an action of ${\W}$ on $\Aa$ as recalled in \cite[page 102]{SS} and on $\mathscr{A}^1$ as well. For simplicity we choose the hyperspecial vertex $x_0$ of the building to be the zero point in $\mathscr A$. The extended Weyl group $\W$  is the semi-direct product of ${\W}_0\ltimes\Lambda$ (\cite[1.9]{Tit}): $\W_0$ identifies with the subgroup of $\W$ that fixes any point of $\mathscr{A}^1$ that lifts $x_0$.

\subsection{Root subgroups}

We now recall the definition of the affine roots and the properties of the  associated root subgroups.
To a root $\root$ is attached a unipotent subgroup $\Uu_\root$ of $\Gp$ such that for any $u\in \Uu_\root-\{1\}$, the intersection $\Uu_{-\root}u \Uu_{-\root}\cap N_\Gp(\Tp)$ consists in only one element called $m_\root(u)$. The image of $m_\root(u)$ in $  \W_0$ is $s_\root$ and its translation part has the form $-\mathfrak h_\root(u)\,.\,\coroot$ for some real number $\mathfrak h_\root(u)$: the image in $\W$ of this element $m_\root(u)$ is the reflection at the affine hyperplane $\{x\in \mathscr{A}, \: \root(x)=-\mathfrak h_\root(u)\}$. Denote by $\Gamma_\root$ the discrete unbounded subset  of $\R$ given by $\{\mathfrak h_\root(u), \: u\in \Uu_\root-\{1\}\}$. Since our group $\mathbf{G}$ is $\mathfrak F$-split we, in fact, have $\Gamma_\root = \mathbb{Z}$ for any $\root$. The affine functions
\begin{equation*}
    (\root, \mathfrak h):=\root(\,.\,)+\mathfrak h,\:\: \root\in\Root, \: \mathfrak h\in \Gamma_\root
\end{equation*}
are called the affine roots. We identify an element $\root$ of $\Root$ with the affine root $(\root,0)$ so that the set of affine roots $\Root_{aff}$ contains $\Root$. From the definition of  $m_\root(u)$, it appears that the action of $\W_0$  on $\Root$ extends into an action of $\W$ on the set $\Root_{aff}$ of affine roots. Explicitly, if $w=w_0  t_\lambda \in \W$ is the composition of the translation by $\lambda\in  \Lambda$ with $w_0\in\W_0$, then the action of $w$ on the affine root $(\root,\mathfrak h_\root(u))$ with $u\in \Uu_\root$ is
\begin{equation*}
    (w_0(\root),  \mathfrak h_\root (u) + (\val_{\mathfrak F} \circ \root)(\lambda) ) = (w_0(\root),  \mathfrak h_\root (u) - \langle \nu(\lambda), \root\rangle)
\end{equation*}
and we can check that $\mathfrak h_\root (u)+ (\val_{\mathfrak F} \circ \root)(\lambda) =\mathfrak h_{w_0(\root)}(w uw^{-1})$
so that the latter element is indeed an affine root.

Define a filtration of $\Uu_\root$, $\root\in \Root$ by
\begin{equation*}
    \Uu_{\root , r}:=\{u\in \Uu_\root-\{1\}, \: \mathfrak h_\root(u)\geq r\}\cup\{1\}\textrm{  for } r\in \R \ .
\end{equation*}
Set $\Uu_{\root, \infty}=\{1\}$.

By abuse of notation we write throughout the paper $wUw^{-1}$, for some $w \in \W$ and some subgroup $U \subseteq \Gp$, whenever the result of this conjugation is independent of the choice of a representative of $w$ in $N_\Gp(\Tp)$.

\begin{rema}\label{inclus}
Let $(\root, \mathfrak h)$ be an affine root and put $\Uu_{(\root, \mathfrak h)} := \Uu_{\root, \mathfrak h}$.
  \begin{enumerate}
  \item For any $w\in \W$, we have $w \Uu_{(\root, \mathfrak h)}w^{-1}= \Uu_{w(\root, \mathfrak h)}$.
  \item For $r\in\R$ a real number,  $\mathfrak h\geq r$ is equivalent to $\Uu_{(\root, \mathfrak h)}\subset \Uu_{\root, r}$.

  \end{enumerate}
\end{rema}

For any non empty subset $\Y\subset \Aa$, define
\begin{align*}
    f_\Y : \Root & \longrightarrow \R\cup \{\infty\} \\ \root & \longmapsto  -\inf_{x\in \Y} \root(x) \ .
\end{align*}
and the subgroup of $\Gp$
\begin{equation}\label{defiU}
\Uu_\Y= \; < \Uu_{\root, f_\Y(\root)}, \:\: \root\in \Root>
\end{equation}
generated by  all $\Uu_{\root, f_\Y(\root)}$ for $ \root\in \Root$.

Choosing a chamber $C$  such that $\overline C$ contains $x_0$ in the standard apartment amounts to choosing the subset $\Root^+$ of $\Root$ of the roots $\root$ such that $\root(C)\geq 0$.  These roots are then called the positive roots. Denote by $\Pi $ a basis for $\Root^+$.
The set of positive affine roots ${\Root_{aff}^+}$ is defined to be the set of  affine roots taking nonnegative values on the standard chamber $C$. An affine root $(\root, \mathfrak h)$ is  called negative if $(-\root, -\mathfrak h)$ is positive.

\begin{lemm}\phantomsection\label{UC}
\begin{itemize}
\item[i.] We have
  \begin{equation*}
    \Uu_{\root, f_C(\root)} =
    \begin{cases}
    \Uu_{\root,0} & \textrm{if $\ \root \in \Phi^+$}, \\
    \Uu_{\root,1} & \textrm{if $\ \root \in \Phi^-$}.
    \end{cases}
  \end{equation*}
\item[ii.] The group $\Uu_C$ is generated by all $\Uu_{a}$ for $a\in \Root^+_{aff}$.
\end{itemize}
\end{lemm}

\begin{proof}
Let $\root\in \Root$ and let $\mathfrak h_0\in  \{0,1\} \subseteq\Gamma_\root$ be the minimal element in $\Gamma_\root = \mathbb{Z}$ such that $(\root,\mathfrak h_0)$ is a positive affine root. We prove that $\Uu_{\root, f_C(\root)} = \Uu_{\root, \mathfrak h_0}$.
Since  $\root+\mathfrak h_0$ takes nonnegative values on $C$, we have $f_C(\root)\leq  \mathfrak h_0$ and
$\Uu_{\root, \mathfrak h_0}$ is contained in $\Uu_{\root, f_C(\root)}$.
Now let $u\in \Uu_{\root, f_C(\root)} -\{1\}$. We have $\mathfrak h_\root(u)\geq  f_C(\root)$ and it implies that
$\root+\mathfrak h_\root(u)$ takes nonnegative values on $C$ so that
$\mathfrak h_\root(u)\geq \mathfrak h_0$ and $u\in \Uu_{\root,\mathfrak h_0}$.
\end{proof}



\subsection{}

The finite Weyl group $\W_0$ is a Coxeter system generated by  the set $S := \{s_\root : \root \in \Pi\}$ of reflections associated to the simple roots $\Pi$. It is endowed with a length function denoted by $\ell$.
This length extends to $\W$ in such a way that the length of an element $w\in \W$ is the cardinality of $\{ A\in\Root^+_{aff},\: w(A)\in {\Root_{aff}^-}\}$  (see \cite[1]{Lu}). For any affine root $(\root, \mathfrak h)$, we have in $\W$ the reflection at the affine hyperplane $\root(\, .\,) = - \mathfrak{h}$ given by
\begin{equation*}
    s_{(\root, \mathfrak h)} := \textrm{image in $\W$ of $m_\root(u)$},
\end{equation*}
where $u \in \Uu_\root$ is such that $\mathfrak{h} = \mathfrak{h}_\root(u)$. The affine Weyl group is defined as the subgroup
\begin{equation*}
\W_{aff} := \; < s_A , \:\: A \in \Phi_{aff}>
\end{equation*}
of $\W$.
There is  a partial order on $\Root$ given by $\root\leq \beta$ if and only if $\beta -\root$ is a linear combination with (integral) nonnegative coefficients of elements in $\Pi$. Let
\begin{equation*}
    \Phi^{min} := \{\root \in \Phi : \root\ \textrm{is minimal for $\leq$}\}
\end{equation*}
and
\begin{equation*}
    \Pi_{aff} := \Pi \cup \{(\root,1) : \alpha \in \Phi^{min}\} \subseteq \Phi^+_{aff} \ \textrm{and}\   S_{aff} := \{s_A : A \in \Pi_{aff}\} \ .
\end{equation*}
The length satisfies (\cite[1]{Lu}) the following formula, for every $A \in \Pi_{aff}$ and $w\in \W$:
\begin{equation}\label{add}
   \ell(w s_A)=
   \begin{cases}
       \ell(w)+1 & \textrm{ if }w (A)\in {\Root_{aff}^+},\\  \ell(w)-1 & \textrm{ if }w (A)\in {\Root_{aff}^-}.
    \end{cases}
\end{equation}
The pair $(\W_{aff}, S_{aff})$ is a Coxeter system (\cite[V.3.2 Thm.\ 1(i)]{Bki-LA} or \cite[Satz 2.2.16]{Born}), and the length function $\ell$ restricted to $\W_{aff}$ coincides with the length function of this Coxeter system (\cite[Cor.\ 2.2.12]{Born}). We moreover have (\cite[Satz 2.3.1, Lemma 2.3.2]{Born} or \cite[1.5]{Lu}):
\begin{itemize}
  \item[--] $\W_{aff}$ is a normal subgroup of $\W$.
  \item[--] $\Omega := \{w \in \W : \ell(w) = 0\}$ is an abelian subgroup of $\W$.
  \item[--] $\W$ is the semi-direct product $\W = \Omega \ltimes \W_{aff}$.
  \item[--] The length $\ell$ is constant on the double cosets $\Omega w \Omega$ for $w \in \W$. In particular $\Omega$ normalizes $S_{aff}$ and $\Omega C = C$.
\end{itemize}
One easily deduces that
\begin{equation}\label{f:subadd}
    \ell(vw) \leq \ell(v) + \ell(w) \qquad\textrm{for any $v, w \in \W$}.
\end{equation}

We fix a facet $F$ contained in  $\overline C$. Attached to it is the subset $\Pi_F$ of the affine roots in $\Pi_{aff}$ taking value zero on $F$, or equivalently the subset $S_F$ of those reflections in $S_{aff}$ fixing $F$ pointwise. We let
\begin{equation*}
    \Phi_F := \{ (\root,\mathfrak{h}) \in \Phi_{aff} : (\root,\mathfrak{h})|F = 0\},\ \Root_F^+ := \Root_F \cap \Root^+_{aff},\ \Root_F^- : = \Root_F \cap \Root^-_{aff},
\end{equation*}
and
\begin{equation*}
    \W_F := \ \textrm{subgroup of $\W_{aff}$ generated by all $s_{(\root,\mathfrak{h})}$ such that $(\root,\mathfrak{h})|F = 0$}.
\end{equation*}
The pair $(\W_F, S_F)$ is a Coxeter system with $\W_F$ being the pointwise stabilizer in $\W$ of $\pr^{-1}(F)$ (\cite[V.3.3, IV.1.8 Thm.\ 2(i)]{Bki-LA}), the restriction $\ell | \W_F$ coincides with its length function (\cite[IV.1.8 Cor.\ 4]{Bki-LA}), and $\W_F$ is finite (\cite[V.3.6 Prop.\ 4]{Bki-LA}).

\begin{rema}\label{facet}
The closure $\overline{F}$ of a facet $F$ consists exactly of the points of $\overline C$ that are fixed by the reflections in $S_F$ (\cite[V.3.3 Proposition  1]{Bki-LA}).
\end{rema}

\begin{rema}\label{rema:transitivity}
Recall that $\W_{aff}$ acts simply transitively on the chambers of the standard apartment $\Aa$ (\cite[V.3.2 Th\'eor\`eme 1]{Bki-LA}).
Moreover, for any facet $F'$ in $\Aa$, there is a unique facet $F$ contained in $\overline C$ which is $\W_{aff}$-conjugate to $F'$  (\cite[1.3.5]{BT1}).
\end{rema}

\begin{lemm}\label{F-basis}
Any $A \in \Phi^+_F$ has a decomposition $A = \sum_{B \in \Pi_F} m_B B$ with uniquely determined nonnegative integers $m_B$.
\end{lemm}
\begin{proof}
By decomposition into irreducible components we may assume that $\Phi$ is irreducible so that the chamber $C$ is a simplex and the cardinality of $\Pi_{aff}$ is equal to $\dim \mathscr{A} + 1$. By \cite[Satz 2.2.8]{Born} the positive affine root $A$ has a unique decomposition $A = \sum_{B \in \Pi_{aff}} m_B B$ with integers $m_B \geq 0$. We have
\begin{equation*}
    (\sum_{B \in \Pi_{aff} \setminus \Pi_F} m_B B)| \mathscr{F} = 0 \ .
\end{equation*}
If $F = \{y\}$ is a vertex then $\Pi_{aff} \setminus \Pi_F$ consists of a single element $B_y$. It follows that $m_{B_y} B_y(y) = 0$. But $B_y(y) \neq 0$ and hence $m_{B_y} = 0$. In general we have
\begin{equation*}
    \Pi_{aff} \setminus \Pi_F = \{B_y : y\ \textrm{is a vertex of $\overline{F}$}\}.
\end{equation*}
But if $y$ is a vertex of $\overline{F}$ then $A \in \Phi^+_{\{y\}}$ and therefore $m_{B_y} = 0$.
\end{proof}

The map
\begin{align*}
    \Phi_{aff} & \longrightarrow \Phi \\
    (\root, \mathfrak{h}) & \longmapsto \root
\end{align*}
is surjective and equivariant with respect to the projection map $\W \twoheadrightarrow \W_0$. Its restriction $\Phi_F \rightarrow \Phi$ clearly is injective; let $\Pi'_F \subseteq \Phi'_F$ denote the image  of $\Pi_F \subseteq \Phi_F$, respectively, in $\Phi$. Since $\W_F$ is finite and $\ker(\W \twoheadrightarrow \W_0)$ is torsionfree the restriction $\W_F \rightarrow \W_0$ is injective as well; let $S'_F \subseteq \W'_F$ denote the image of $S_F \subseteq \W_F$, respectively, in $\W_0$. Obviously $S'_F$ is a generating set of $\W'_F$. The subgroup $\W_F$ leaves the subset $\Phi_F$ invariant. Hence $\W'_F$ leaves $\Phi'_F$ invariant. It follows that $\Phi'_F$ is a subroot system of $\Phi$ with Weyl group $\W'_F$ (see \cite[1.9]{Tit}) and, using Lemma \ref{F-basis}, that $\Pi'_F$ is a basis of the root system $\Phi'_F$. For example, if $F= \{x_0\}$, then $\W_{\{x_0\}}= \W_0$ and $\Root_F=\Root$. If $F= C$, then $\W_C=\{1\}$ and $\Root_C=\emptyset$. We have the following result (compare with \cite[Propositions 2.1, 2.5, 2.7]{Oparab}).


\begin{prop}\phantomsection\label{representants}
\begin{itemize}
\item[i.] The set $\Dd_F$ of elements $d\in \W$ satisfying
\begin{equation*}
    d(\Root_F^+)\subset {\Phi}_{aff}^+
\end{equation*}
is a system of representatives of the left cosets $\W/ \W_F$. It satisfies
\begin{equation}\label{additive0}
\ell( dw_F)=\ell(d)+\ell(w_F)
\end{equation}
for any $w_F\in \W_F$ and $d\in \Dd_F$. In particular, $d$ is the unique element with minimal length in $d\W_F$.
\item[ii.] For  $s\in S_F$ and $d\in \Dd_F$, we are in one of the following situations:
\begin{itemize}
\item[--] $\ell(s d)=\ell(d)-1$ in which case $sd\in \Dd_F$.
\item[--] $\ell(sd)=\ell(d)+1$ in which case either $sd\in \Dd_F$ or $sd\in d\W_F$.
\end{itemize}
\end{itemize}
\end{prop}

\begin{proof}
i. First check that the cosets $d\, \W_F$  are pairwise disjoint for $d \in \Dd_F$. Let $ d_1, d_2\in\Dd_F$  such that $d_1\, \W_F = d_2\, \W_F$. In particular, $d _1^{-1}d_2$ is an element of $\W_F$, and we suppose that it is nontrivial. Then there exists an affine root $A \in \Pi_F \subseteq \Phi^+_F$ such that $\ell( d _1^{-1}d_2 s_A )=\ell(d _1^{-1}d_2)-1$. By \eqref{add} we obtain $d _1^{-1}d_2(A)\in {{\Root_F^-}}$ and hence $d_2 A \in d _1(\Root_F^-)\subset \Root^-$ which contradicts the fact that $d_2 \in \Dd_F$.

For $w\in \W$, we  now prove by induction on the length of $w$  that there exists  a (obviously unique) $(d, w_F)\in \Dd_F \times  \W_F$
such that $  w=dw_F $ and that it satisfies $\:\ell(d w_F )=\ell(d)+\ell(w_F )$.

Applying \eqref{add} and Lemma \ref{F-basis}, first note that an element $w\in \W$ belongs to $\Dd_F$ if and only if  for any affine root $A \in \Pi_F$
we have  $\ell(w s_A)=\ell(w)+1$. In particular, any element with length $0$  belongs to $\Dd_F$. Suppose now that $\ell(w)>0$ and that it does not belong to $\Dd_F$. Then there is an $A \in \Pi_F$ such that $\ell(w s_A)=\ell(w)-1$ and, by induction, we can write $w s_A  =d w_F $ with $(d, w_F)\in \Dd_F \times  \W_F$ and $\ell(dw _F)=\ell(d)+\ell(w_F )$.
We have $w=d w_Fs_A$ and $w_Fs_A\in  \W_F$.
It remains to check that $\ell(w)=\ell(d)+\ell(w_Fs_A)$ which amounts to proving that
$\ell(w_F s_A)=\ell(w_F)+1$. The latter is true because otherwise we would have $\ell(w_Fs_A)<\ell(w_F)$ which, using \eqref{f:subadd}, leads to the contradiction  $\ell(w)\leq \ell(d)+\ell(w_Fs_A)<\ell(d)+\ell(w_F)=\ell(w s_A)=\ell(w)-1$.

We have proved the assertions of i., remarking for the last one that $1$ is the only element with length $0$ in $\W_F$.

ii. Let $d\in \Dd_F$ and $s = s_A$ with $A \in \Pi_F$. From \eqref{add} we know that $\ell(s_A d) = \ell(d) \pm 1$. If $s_A d \not\in \Dd_F$, then there is a $B \in \Root_F^+$ such that $s_Ad(B) \in \Phi^-_{aff}$ while $d(B) \in \Phi_{aff}^+$. Since $s_A( \Phi_{aff}^+ \setminus \{A\}) = \Phi_{aff}^+ \setminus \{A\}$ (cf. \cite[1.4]{Lu} or \cite[Lemma 2.2.9]{Born}) we must have $d(B) = A$. In particular, $d^{-1}(A) \in \Phi_{aff}^+$ and  hence $\ell(s_Ad)= \ell(d)+1$ by \eqref{add}. Furthermore, $B = d^{-1}(A) \in \Root_F$ implies that $d^{-1}s d = s_B \in \W_F$ by definition of $\W_F$.
\end{proof}

Let $(\root,\mathfrak{h}) \in \Root_F^+$. Then $f_F(\root)=\mathfrak{h}$ so that the group
\begin{equation*}
    \Uu_F^0 = \; <\Uu_{\root,\mathfrak{h}}, \:\: (\root, \mathfrak{h}) \in \Root_F^+>
\end{equation*}
generated by all $\Uu_{\root,\mathfrak{h}}$, $ (\root, \mathfrak{h}) \in \Root_F^+$, is contained in $\Uu_F$. If $F=C$ then $\Uu_C^0$ is trivial. \begin{lemm}
Let $d\in  \W$. The condition $d(\Root_F^+)\subset  {\Phi}_{aff}^+$ of Proposition \ref{representants} is equivalent to
\begin{equation} \label{rootsubgroup}
     d\: \Uu_F^0 \:d^{-1}\subset \Uu_C \ .
\end{equation}
\end{lemm}

\begin{proof}
Let $d\in \W$ such that $d(\Root_F^+)\subset {\Phi}_{aff}^+$. By Remark \ref{inclus} (1), we have $d \Uu_A d^{-1}=\Uu_{d(A)}$ for any affine root $A$, so that  Lemma \ref{UC}.ii implies $d\: \Uu_F^0 \:d^{-1}\subset \Uu_C$. Now suppose that $d\in \W$ satisfies $d\: \Uu_F^0 \:d^{-1}\subset \Uu_C$ and let $(\root,\mathfrak{h}) \in \Root_F^+$. Denote by $d= w_0 t_\lambda$ the decomposition of $d$ in the semi-direct product $\W=  \W_0.\Lambda$. Recall that
\begin{equation*}
    d(\root, \mathfrak{h})=(w_0(\root), \mathfrak{h} + (\val_{\mathfrak F} \circ \root)(\lambda)) \ .
\end{equation*}
Denote by $\mathfrak h_0 \in \{0, 1\}$ the minimal element in $\Gamma_{w_0(\root)}$ such that
$(w_0(\root), \mathfrak h_0) \in \Root_{aff}^+$.
We have $d \Uu_{\root, \mathfrak{h}} d^{-1}=\Uu_{d(\root, \mathfrak{h})}\subset \Uu_C\cap \Uu_{w_0(\root)}$. The latter is equal to $ \Uu_{w_0(\root), \mathfrak h_0}$ after \cite[I.1, property 3 on p.\ 103]{SS} and Lemma \ref{UC}.i. By Remark \ref{inclus} (2), it implies that $\mathfrak{h} + (\val_{\mathfrak F} \circ \root)(\lambda) \geq \mathfrak h_0$ and $d(\root,\mathfrak{h})\in \Root_{aff}^+$.
\end{proof}

\subsection{\label{subsec:inclusion}}

In Section \ref{sec:univ} we introduced the pro-$p$-subgroup $\I_F$  of ${\mathbf G}_F^\circ(\mathfrak O)$. It is the group denoted by ${\rm R}_F$ in \cite{SS}.
In the following we abbreviate $\Tp^0 := \ker (\nu)$ and we let $\Tp^1$ denote the pro-$p$ Sylow subgroup of $\Tp^0$. Then ${\mathbf G}_F^\circ(\mathfrak O) = \Uu_F \Tp^0$ (\cite[5.2.1, 5.2.4]{BT2}).

\begin{lemm} \label{IC}
$\Uu_C\;\subseteq\; \I_C =\; < \Uu_F^0, \I_F>\;=\Uu_F^0\I_F$.
\end{lemm}

\begin{proof}
For any real number $r$ let $r+$ denote the smallest integer $> r$. We also define the function
\begin{equation*}
    f^*_F(\root) :=
    \begin{cases}
    f_F(\root)+ & \textrm{if $\root | F$ is constant}, \\
    f_F(\root) & \textrm{otherwise}.
    \end{cases}
\end{equation*}
After \cite[Proposition I.2.2]{SS}, the product map induces the following  bijections
\begin{gather*}
     \prod_{\root\in \Root^-} \Uu_{\root, f_C(\root)}\times  \Tp^1 \times \prod_{\root\in \Root^+} \Uu_{\root, f_C(\root)}\overset{\sim}\longrightarrow \I_C \\
     \prod_{\root\in \Root^-} \Uu_{\root, f^*_F(\root)} \times
     \Tp^1 \times \prod_{\root\in \Root^+} \Uu_{\root, f^*_F(\root)} \overset{\sim}\longrightarrow \I_F
\end{gather*}
where the products on the left hand side are ordered in some arbitrarily chosen way. We immediately see that $\Uu_C\;\subseteq\; \I_C$ and hence that $\Uu^0_F\;\subseteq\; \Uu_F\;\subseteq\; \Uu_C\;\subseteq\; \I_C$. Using Lemma \ref{UC}.i the first bijection becomes
\begin{equation*}
    \prod_{\root\in \Root^-} \Uu_{\root,1} \times  \Tp^1 \times \prod_{\root\in \Root^+} \Uu_{\root,0} \overset{\sim}\longrightarrow \I_C \ .
\end{equation*}
Since moreover $\I_F$ is normal in $\I_C$ it remains to show that $\I_C \subseteq \; < \Uu_F^0, \I_F >$. As $f_C(\root) \geq f_F(\root)$ we have $\Uu_{\root, f_C(\root)}\subseteq \Uu_{\root, f^*_F(\root)}$ whenever $\root | F$ is not constant. Suppose therefore that $\root | F$ has the constant value $r$. If $r$ is not an integer then $f_C(\root) \geq f_F(\root)+$ and again $\Uu_{\root, f_C(\root)}\subseteq \Uu_{\root, f^*_F(\root)}$. Suppose now in addition that $r$ is an integer. Then $(\root, -r) = (\root, f_F(\root)) \in \Phi_F$. If this affine root is positive then $\Uu_{\root, f_C(\root)}\subseteq \Uu_{\root, f_F(\root)} \subseteq \Uu_F^0$. If it is negative then $f_F(\root) \leq 0$, resp.\ $\leq -1$, if $\root \in \Phi^-$, resp.\ $\root \in \Phi^+$; hence $f^*_F(\root) \leq 1$, resp.\ $\leq 0$, if $\root \in \Phi^-$, resp.\ $\root \in \Phi^+$.
\end{proof}

\subsection{Parahoric subgroups and Bruhat decompositions\label{subsec:bruhat}}

There is an action of the group $\Gp$ on the building $\Xx$ recalled  in \cite[page 104]{SS} that extends the action of $N_\Gp(\rm T)$
on the standard apartment described in section \ref{subsec:apartment}. This action is compatible  \emph{via} the projection $\rm{pr}$ with the action of $\Gp$ on the enlarged building $\mathscr{X}^1$ described in \cite[4.2.16]{BT2}.
Recall that $\Pp_F^\dagger$ denotes the stabilizer in $\Gp$ of a facet $F$, and denote by $\Pp_{\pr^{-1}(F)} := \mathbf{G}_F(\mathfrak{O})$ the pointwise stabilizer in $\Gp$ of $\pr^{-1}(F)$. We have
\begin{equation*}
    \Uu_F \subseteq  {\mathbf G}_F^\circ(\mathfrak O) \subseteq \Pp_{\pr^{-1}(F)} \subseteq\Pp^\dagger_F
\end{equation*}
with $\Uu_F$ being normal in $\Pp^\dagger_F$. The pro-unipotent radical $\I_F$ of ${\mathbf G}_F^\circ(\mathfrak O)$ is normal  in $\Pp^\dagger_F$. Recall that we set $\Iw:= {\mathbf G}_C(\mathfrak O) = {\mathbf G}_C^\circ(\mathfrak O)$ (see Section \ref{sec:Hecke})  and
$\I:= \I_C$.

A parahoric subgroup of $\Gp$, by definition, is a  subgroup of the form ${\mathbf G}_F^\circ(\mathfrak O)$ for a facet $F$. It contains  an Iwahori subgroup, that is to say a $\Gp$-conjugate of $\Iw={\mathbf G}_C^\circ(\mathfrak O)$ (\cite[3.4.3]{Tit})

We recall the following Bruhat decompositions:
\begin{itemize}
\item  We have the decomposition $\Gp = \Iw N_\Gp(\rm T)\Iw$ and two cosets $ \Iw n_1\Iw$ and $ \Iw n_2\Iw$  are equal if and only if $n_1$ and $n_2$ have the same projection in $\W$.
In other words, let $\{\hat w\}_{w\in \W}$ denote a system of representatives in $N_\Gp(\rm T)$ of the elements in $\W$: it provides a system of representatives of the double cosets of $\Gp$ modulo $\Iw$. This follows from \cite[3.3.1]{Tit} since in our situation we have $\I' = \Pp_{\pr^{-1}(C)}$. The Bruhat decomposition implies that for any chamber $D$ in $\mathscr X$, there is a $\Iw$-conjugate of $D$ that belongs to the standard apartment. The same holds for the $\I$-conjugates of $D$, since
$\Iw N_\Gp({\rm T})\Iw = \I N_\Gp({\rm T})\Iw$.
\item The subgroup $\Gp_{aff}$ of $\Gp$ generated by all parahoric subgroups  is the disjoint union of the double cosets ${\Iw} \hat w {\Iw}$ where $w$ runs over the affine Weyl group $\W_{aff}$. {This follows from \cite[5.2.12]{BT2} since $\I' = {\mathbf G}_C^\circ(\mathfrak O) \subseteq \Gp_{aff}$.} \end{itemize}

We again fix a facet  $F$ contained in $\overline C$. Denote by $\Omega_F$ the subgroup of the elements in $\Omega$ stabilizing $F$.
Using Remark \ref{facet}, we have
$$\Omega_F=\{ w\in \Omega, \: w S_F w^{-1}= S_F\}$$ so that
 $ \W_F$   is normalized by $\Omega_F$. Denote by $ \W_F^\dagger$ the subgroup of $\W$
generated by  $ \W_F$ and  $\Omega_F$. It is a semi-direct product of these two subgroups. Remark that $\W_C^\dagger=\Omega$ and
$\W_F^\dagger\cap \W_{aff}=\W_F$.

\begin{lemm} \label{bruhat-finite}
The stabilizer
 $\Pp_F^\dagger$  of $F$ is the distinct union of the double cosets ${\Iw} \hat w {\Iw}$ for all $ w$ in  $ \W_F^\dagger$. In particular, $\W_F^\dagger$ is the stabilizer of $F$ in $\W$. The  parahoric subgroup ${\mathbf G}_F^\circ(\mathfrak O)$
is the distinct union of the double cosets ${\Iw} \hat w {\Iw}$ for all $w$ in   $ \W_F$.
\end{lemm}

\begin{proof}
Both  $\Pp_F^\dagger$ and ${\mathbf G}_F^\circ(\mathfrak O)$ contain $\I'$ because $F$ is  contained in $\overline C$.
The group $\Pp_F^\dagger$ is the reunion of the double cosets ${\Iw} \hat w{\Iw}$ such that $w\in \W$ stabilizes $F$. An element $w$ in $\W$ can be written $w=w_{aff}\omega $ with $\omega\in \Omega$ and $w\in \W_{aff}$. Suppose that $w$ stabilizes $F$.
Since $\omega F$ is  contained in $\overline C$ and is $\W_{aff}$-conjugate to $F$, we have
$\omega\in \Omega_F$ and $w_{aff} F=F$ by Remark \ref{rema:transitivity}.  By \cite[V.3.3 Proposition 1]{Bki-LA}, it implies that $w_{aff}\in \W_F$ which finishes the  proof of the first assertion.

For the second part we first of all note that $\Tp^0 \subseteq {\mathbf G}_F^\circ(\mathfrak O)$ by \cite[4.6.4(ii)]{BT2}. Moreover, \cite[4.6.7(iii)]{BT2} implies that the reduction map ${\mathbf G}_F^\circ(\mathfrak O) \longrightarrow \overline{\mathbf G}_F^\circ(\mathbb{F}_q)$ is surjective. We now deduce from \cite[4.6.33 and 1.1.12]{BT2} that $\I'/\I_F$ is a Borel subgroup of the finite reductive group ${\mathbf G}_F^\circ(\mathfrak O) /\I_F$. By the first part of our assertion we have the decomposition ${\mathbf G}_F^\circ(\mathfrak O) = \coprod_w \I' \hat{w} \I'$ with $w$ running over $({\mathbf G}_F^\circ(\mathfrak O) \cap N_\Gp (\Tp)) /\Tp^0$. This means, by the Bruhat decomposition of ${\mathbf G}_F^\circ(\mathfrak O) /\I_F$ with respect to its Borel subgroup $\I'/\I_F$, that $({\mathbf G}_F^\circ(\mathfrak O) \cap N_\Gp (\Tp)) /\Tp^0$ maps isomorphically onto the Weyl group of ${\mathbf G}_F^\circ(\mathfrak O) /\I_F$ with respect to the torus $\Tp^0 \I_F /\I_F$. Since the elements in ${\mathbf G}_F^\circ(\mathfrak O)$ stabilize $\pr^{-1}(F)$ pointwise we have $({\mathbf G}_F^\circ(\mathfrak O) \cap N_\Gp (\Tp)) /\Tp^0 \subseteq \W_F$. But it follows from \cite[3.5.1 and 1.9]{Tit} that $\W_F$ maps isomorphically onto the very same Weyl group. Hence we must have $({\mathbf G}_F^\circ(\mathfrak O) \cap N_\Gp (\Tp)) /\Tp^0 = \W_F$.
\end{proof}

We deduce from the latter the following.

\begin{lemm} \label{inters} The intersection of $\Pp_F^\dagger$ with the subgroup  $\Gp_{aff}$ of $\Gp$ generated by all parahoric subgroups is equal to ${\mathbf G}_F^\circ(\mathfrak O)$.
\end{lemm}
\noindent

We now describe systems of representatives for the double cosets $\I'\backslash\Gp/{\mathbf G}_F^\circ(\mathfrak O)$ and $\I'\backslash\Gp/\Pp_F^\dagger$.


\begin{lemm} \label{stable} The set $\Dd_F$ is stable by right multiplication by an element in $\Omega_F$.
\end{lemm}

\begin{proof} Let $\omega\in \Omega_F$. In particular it satisfies $\omega(\Phi_F)= \Phi_F$. Moreover, it is the minimal length element in $\omega \W_0$: in other words, $\omega\in \Dd_{x_0}$ so that $\omega(\Phi^+)\subset \Phi_{aff}^+$.
It proves that $\omega(\Phi_F^+)=\Phi_{aff}^+\cap \Phi_F=\Phi_F^+$.
\end{proof}

Denote by $\Dd_F^\dagger$ a system of representatives of the orbits in $\Dd_F$ under the right action of $\Omega_F$.

\begin{lemm}\phantomsection\label{DF+}
\begin{itemize}
\item[i.] The set   $\Dd_F^\dagger$ is a system of representatives of the left cosets $\W/ \W_F^\dagger$. It satisfies
$\ell(dw_F)=\ell(d)+\ell(w_F)$ for any $d\in \Dd_F^\dagger$, $w_F\in  \W_F^\dagger$.
\item[ii.] The sets $\{\hat d\,\}_{d\in \Dd_F}$ and $\{\hat d\,\}_{d\in \Dd^\dagger_F}$  are  respective  systems
of representatives  of the double cosets  $\I'\backslash\Gp/{\mathbf G}_F^\circ(\mathfrak O)$ and $\I'\backslash\Gp/\Pp_F^\dagger$.
\end{itemize}
\end{lemm}

\begin{proof}  (1) That $\W$ is the reunion of the $d\W_F^\dagger$ for $d \in \Dd_F^\dagger$ is clear from Prop.\ \ref{representants} because $\Omega_F\subset \W_F^\dagger$.
Let $d_1, d_2\in \Dd_F^\dagger$. If $d_2\in  d_1\W_F^\dagger$, then there is $\omega\in \Omega_F$ and $w_F\in  \W_F$ such that $d_2=d_1\omega w_F$. But $d_1\omega\in \Dd_F$ by Lemma \ref{stable}, so that $d_2= d_1\omega$ by Prop.\ \ref{representants} and $\omega=1$ by definition of $\Dd_F^\dagger$.
The length equality is clear again by Prop.\ \ref{representants}. (2) Let $w_1, w_2\in \W$ such that $\ell(w_1)+\ell(w_2)=\ell(w_1 w_2)$. Then $\Iw \hat w_1 \Iw\hat w_2\Iw = \Iw \hat w_1 \hat w_2\Iw$ (for example \cite[Thm.\ 3.6]{Cartier}). Together with Lemma \ref{bruhat-finite}, this remark finishes the proof.
\end{proof}

\subsection{\label{subsec:distance}}

We recall the geometric interpretation of the set $\Dd_F^\dagger$. The \emph{gallery distance} $\d(D, D')$ between two chambers $D$ and $D'$ of the standard apartment is defined to be the minimal length of a gallery connecting  $D$ and $D'$. By \cite[2.3.10]{BT1} and keeping in mind that $\W_{aff}$ acts simply transitively on the chambers of $\Aa$ we have, for any $w, w'\in \W_{aff}$:
\begin{equation}\label{eq:distance}
     \d(wC, w'C)= \ell(w^{-1}w').
\end{equation}

\begin{prop}\label{prop:closest}
Let $F$ be a  facet $F\subset\overline C$ and $d\in\Dd_F^\dagger$.
\begin{itemize}
\item[i.] Among the chambers of $\Aa$    containing $dF$ in their closure, the chamber
 $C(dF):= dC$ is the unique one  which is closest to the chamber $C$.
\item[ii.] It satisfies $\I_{dF} (\I \cap \Pp_{dF}^\dagger) = \I_{C(dF)}$
and  $\I \cap \Pp_{dF}^\dagger = \I \cap \Pp_{C(dF)}^\dagger$.
\end{itemize}
\end{prop}

\begin{proof}
i. A chamber of $\Aa$ has the form $w C$ for some $w\in \W$  and we can multiply $w$ on the right by an element in $\Omega$  so as to have $w^{-1} d\in \W_{aff}$.  Suppose that $dF\subset w\overline C$. Then $w^{-1}dF=F$ by Remark \ref{rema:transitivity}. Therefore $d\in w \W_F^{\dagger}$ after Lemma \ref{bruhat-finite}. But $d$ is, up to right multiplication by an element in $\Omega_F\subset \Omega$, the unique element with minimal length in $d \W_F^\dagger$. So $\d(wC, C)>\d(dC, C)$ if $wC\neq dC$.

ii. It is immediate from Lemma \ref{I-orient} that $\I \cap \Pp_{C(dF)}^\dagger \subset \I \cap \Pp_{dF}^\dagger$. For the second point of the proposition we therefore need to prove that
the subgroup of $\Pp_F^\dagger$ generated by $\I_F \cup (\hat d^{-1}\I \hat d \cap \Pp_F^\dagger)$ is equal to $\I$ and that $\hat d^{-1}\I \hat d \cap \Pp_F^\dagger \subset \hat d^{-1}\I \hat d \cap \Pp_C^\dagger$.

The intersection $\hat d^{-1}\I \hat d \cap \Pp_F^\dagger $ is a pro-$p$-subgroup of  $\Pp_F^\dagger$ contained in $\Gp_{aff}$ because $\hat d^{-1}\I \hat d$ is contained in a parahoric subgroup. Thus $\hat d^{-1}\I \hat d \cap \Pp_F^\dagger $
is a pro-$p$-subgroup of the parahoric subgroup ${\mathbf G}_F^\circ(\mathfrak O)$ by Lemma \ref{inters}, and  it is contained in a  pro-$p$-Sylow subgroup of the latter (\cite[I.1.4 Prop.\ 4]{Ser}).
But any pro-$p$-Sylow subgroup of ${\mathbf G}_F^\circ(\mathfrak O)$ contains its unipotent pro-radical $\I_F$. Therefore the subgroup of $\Pp_F^\dagger$ generated by $\I_F \cup (\hat d^{-1}\I \hat d \cap \Pp_F^\dagger)$ is a pro-$p$-subgroup of ${\mathbf G}_F^\circ(\mathfrak O)$. To prove that it is equal to $\I$, it is therefore enough to prove that it contains $\I$ (because $\I$ is a maximal pro-$p$-subgroup of ${\mathbf G}_F^\circ(\mathfrak O)$), and by the identity in Lemma \ref{IC}, that it contains $\Uu_F^0$.

But $\Uu_F^0$ is contained in $\hat d^{-1}\Uu_C \hat d\,\cap \, \Pp_F^\dagger$ by \eqref{rootsubgroup}, hence it is contained in $\hat d^{-1}\I \hat d\,\cap \, \Pp_F^\dagger$ by the inclusion in Lemma \ref{IC}. It proves the first equality, from which we deduce
that $\hat d^{-1}\I \hat d \cap \Pp_F^\dagger\subset \I$ so that $\hat d^{-1}\I \hat d \cap \Pp_F^\dagger\subset \hat d^{-1}\I \hat d\cap \I\subset \hat d^{-1}\I \hat d\cap \Pp_C^\dagger$.
\end{proof}

\begin{rema}\label{rema:Omega}
We deduce from i. in Proposition \ref{prop:closest} that if $d F$ is contained in $\overline C$ for $d\in \Dd_F^\dagger$, then $d\in \W_C^\dagger=\Omega$.
\end{rema}

Let $n\geq 0$ and  $Ch(n)$ be the set of the chambers in $\Aa$ at distance $n$ from $C$.
Denote by $\Aa(n)$  the set  of the facets in $\Aa$  contained in the closure of the chambers at distance $\leq n$.

\begin{lemm}\label{disjoint}
For $n\geq 1$, the set $\Aa(n)$  is the disjoint union
\begin{equation*}
    \mathscr{A}(n) =  \mathscr{A}(n-1)\sqcup\bigsqcup_{D \in Ch(n)} {\overline D} \setminus \mathscr{A}(n-1)
\end{equation*}
\end{lemm}
\begin{proof}
Let $D$ and $D'$ be  two distincts chambers in $\Aa$ at distance $n$ from $C$ and    $\tilde F$  a facet contained in  both $\overline D$ and $\overline D'$. We have to prove that $\tilde F$ is contained in the closure of a chamber at distance $<n$.

Fix  $F\subset \overline C$  a facet in the closure of the standard chamber such that $\tilde F$ is $\W$-conjugate to $F$. There is a unique  $d\in \Dd_{F}^\dagger$ such that $\tilde F= d F$. We shall prove that $\ell(d)<n$ which gives the required result by \eqref{eq:distance}. Let $w, w'\in \W$, unique up to right multiplication by an element in $\Omega$, such that $D=w C$ and $D'=w'C'$ respectively. We have
$\ell(w)=\ell(w')=n$ by hypothesis. By the same argument as  in the proof of the previous proposition, we can choose $w$ and $w'$ such that $w, w'\in d \W_F^\dagger$. In particular, there are two elements $w_0, w_0'\in \W_{F}^\dagger$ such that
$w= d w_0$, $w'= dw_0'$ and $\ell(w_0)=\ell(w_0')= n-\ell(d)$. If $\ell(d)=n$, then $\ell(w_0)=\ell(w_0')=0$ and $w_0, w_0'\in \Omega$ which contradicts the fact that $D\neq D'$. Therefore
$\ell(d)<n$.
\end{proof}

\begin{prop}\label{prop:contractible}
Let $n\geq 1$ and $D \in Ch(n)$. The simplicial subcomplexes $\Aa(n-1)$ and $\overline D\cup\Aa(n-1)$ of $\Aa$  are contractible for the structure of affine Euclidean space on $\Aa$.
\end{prop}

\begin{proof}
First recall that  for any two points $x, y\in \Aa$ there is a unique geodesic
$[x,y]$  in $\Aa$  connecting them. It is simply the segment $\{tx+(1-t)y, \: t\in [0, 1]\}$ (\cite[2.5.4 (ii) and 2.5.13]{BT1}).

Let $D$ be a chamber in $\Aa$. We denote by $cl(C,D)$ the union of the closures of the chambers in $\Aa$ that belong to a minimal gallery between $C$ and $D$. This set is called the \emph{enclos} of $C$ and $D$ in \cite[2.4]{BT1}. It is  convex by \cite[Proposition 2.4.5]{BT1}.

Now let $\mathscr{C}$ denote any of the two subcomplexes in the assertion. Obviously $\overline{C} \subseteq \mathscr{C}$. Let $D$ be any chamber in $\mathscr{C}$ and let $(C = C_0, C_1, ..., C_\ell = D)$ be any minimal gallery. Since $\d(C,C_i) = i$ for any $0 \leq i \leq \ell$ we see that all chambers $C_0, \ldots, C_\ell$ belong to $\mathscr{C}$. It follows that the whole enclos $cl(C, D)$ is contained in $\mathscr{C}$. This ensures that $\mathscr{C}$ is star-like and hence contractible.
\end{proof}

\subsection{\label{pro-p}}

We denote by $\tilde\W$ the quotient of $N_\Gp(\Tp)$   by $\Tp^1$ (notation in section \ref{subsec:inclusion}) and obtain the exact sequence
$0\rightarrow \Tp^0/\Tp^1 \rightarrow \tilde\W\rightarrow \W\rightarrow 0$. It is convenient in the following to fix a lift $\tilde{v} \in \tilde\W$ of any $v \in \W$ as well as a lift $\hat{w} \in N_\Gp(\Tp)$ for any $w \in \tilde\W$. By \cite[Theorem 1]{Vig} the group $\Gp$ is the disjoint union of the double cosets $\I\hat w \I$ for all $w\in \tilde \W$. The length function $\ell$ on $\W$ pulls back to a length function $\ell$ on $\tilde\W$  (\cite[Proposition 1]{Vig}).

Fix a facet $F$  contained in $\overline C$.
Define $\tilde{ \W}_F$, resp.\  $\tilde{ \W}_F^\dagger$, to be the preimage in $\tilde \W$ of $\W_F$, resp.\ $ \W_F^\dagger$.
Their  lifts provide respective systems of representatives of the double cosets in $\I\backslash {\mathbf G}_F^\circ(\mathfrak O)/\I$
and  $\I\backslash \Pp_F^\dagger/\I$. We also let $\tilde{\Omega}_F \subseteq \tilde{\W}$ denote the preimage of $\Omega_F$.

\begin{rema}\phantomsection\label{repr-I}
\begin{itemize}
\item[1.] Note that the set $\{\tilde{d}\}_{d \in \Dd_F}$ is a system of representatives of the left cosets in $\tilde\W/  \tilde{ \W}_F$ and  that $\{\hat{\tilde{d}}\}_{d \in \Dd_F}$ is a system of representatives of the double cosets in $\I\backslash \Gp/{\mathbf G}_F^\circ(\mathfrak O)$. Similarly, $\Dd_F^\dagger$  provides a system of representatives of the double cosets in $\I\backslash \Gp/\Pp_F^\dagger$.

\item[2.] Together with Remark \ref{rema:transitivity} and the Bruhat decomposition for $\Gp_{aff}$, this decomposition of $\Gp$ into double cosets modulo $\I$ and $\Pp_F^\dagger$ implies that any facet in $\mathscr X$ is $\I$-conjugate to a unique facet in $\Aa$.
\end{itemize}
\end{rema}

\subsection{\label{subsec:defining-relations}}

A basis for $\Hh$ (resp. $\Hh'$) is given by the characteristic functions of the double cosets $\I \backslash \Gp/\I$ (resp. $\I' \backslash \Gp/\I'$). For $w\in \tilde\W$ (resp. $w\in \W$), denote by $\tau_w$ (resp. $\tau'_w$) the corresponding characteristic function.

The defining relations of $\Hh'$ are the braid relations (\cite[2.1 and 3.2 (a)]{Lu})
\begin{equation}\label{braid-iwahori}
\tau'_{w} \tau'_{w'} = \tau'_{ww'} \qquad\textrm{ for $w,w'\in \W$ such that } \ \ell(ww')=\ell(w)+\ell(w')
\end{equation}
together with the quadratic relations
\begin{equation}\label{quadratic-iwahori}
(\tau'_s)^2= (q-1)\tau'_s+q
\end{equation}
for all $s\in S_{aff}$ (\cite[3.2]{Lu}).

The defining relations of $\Hh$ are the
braid relations (see \cite[Theorem 1]{Vig}) \begin{equation}\label{braid}
\tau_{w} \tau_{w'} = \tau_{ww'} \qquad\textrm{ for $w,w'\in \tilde \W$ such that } \ \ell(ww')=\ell(w)+\ell(w')
\end{equation}
together with a modified form of the
quadratic relations \eqref{quadratic} which we introduce below after some preliminaries. As part of a Chevalley basis we have (cf.\ \cite[3.2]{BT2}, \cite[II.1.3]{Jan}, \cite{Ste}), for any root $\alpha \in \Phi$, a homomorphism $\varphi_\alpha : {\rm SL}_2 \longrightarrow \mathbf{G}_{x_0}$ of $\mathfrak{O}$-group schemes which restricts to isomorphisms
\begin{equation*}
    \{ \begin{pmatrix}
    1 & \ast \\ 0 & 1
    \end{pmatrix}
    \}
    \xrightarrow{\;\cong\;} \Uu_\alpha
    \qquad\text{and}\qquad
    \{ \begin{pmatrix}
    1 & 0 \\ \ast & 1
    \end{pmatrix}
    \}
    \xrightarrow{\;\cong\;} \Uu_{-\alpha} \ .
\end{equation*}
Moreover, one has $\coroot (t) = \varphi_\alpha( \bigl( \begin{smallmatrix}
t & 0 \\ 0 & t^{-1}
\end{smallmatrix}
\bigr) )$. For any $s = s_{(\alpha,\mathfrak{h})} \in S_{aff}$ we put
\begin{equation*}
    n_s := \varphi_\alpha (
    \begin{pmatrix}
    0 & \pi^\mathfrak{h} \\
    - \pi^{-\mathfrak{h}} & 0
    \end{pmatrix}
    ) \in N_\Gp(\Tp) \ .
\end{equation*}
We have $n_s^2 = \coroot(-1) \in \Tp^0$ and $n_s \Tp^0 = s \in \W$. In the following we make the choice $\tilde s := n_s \Tp^1$ and $\hat{\tilde{s}} := n_s$. We also let the subtorus $\Tp_s \subseteq \Tp$ denote the image (in the sense of algebraic groups) of the cocharacter $\coroot$. Its group of $\mathbb{F}_q$-rational points $\Tp_s(\mathbb{F}_q)$ is a subgroup of $\Tp^0 / \Tp^1$ isomorphic to $\mathbb{F}_q^\times$.
The quadratic relations in $\Hh$ are:
\begin{equation}\label{quadratic}
    \tau_{\tilde s}^2 = q \tau_{{\tilde s}^2} + \tau_{\tilde s} \theta_s \qquad\text{for any $s \in S_{aff}$},
\end{equation}
where
\begin{equation*}
    \theta_s := \sum_{t \in \Tp_s(\mathbb{F}_q)} \tau_t
\end{equation*}
(cf.\ \cite[Thm. 1]{Vig}).
A general element $w \in \tilde\W$ can be decomposed into $w = \omega \tilde{s_1} \ldots \tilde{s_\ell}$ with $\omega \in \tilde\Omega$, $s_i \in S_{aff}$, and $\ell = \ell(w)$.  The braid relations imply $\tau_w = \tau_\omega \tau_{\tilde{s_1}} \ldots \tau_{\tilde{s_\ell}}$.

\begin{rema} \label{rema:invertible}
By \eqref{braid},  a basis element $\tau_w$  for $w\in \tilde \W$ with length zero is invertible in $\Hh$. For $s\in S_{aff}$, we have $\tau_{\tilde s}(\tau_{\tilde s}-\theta_s) =q \tau_{\check\alpha(-1)}$ and therefore $ \tau_{\tilde s}$ is invertible if and only if $q$ is invertible in $k$. All basis elements $\tau_w$  for $w\in \tilde \W$ are invertible in $\Hh$ if and only if $q$ is invertible in $k$. Likewise, all basis elements $\tau'_w$  for $w\in  \W$ are invertible in $\Hh'$ if and only if $q$ is invertible in $k$.

\end{rema}
We set $\tau_{\tilde s}^\ast := \tau_{{\tilde s}^2} (\tau_{\tilde s} - \theta_s) = \tau_{\tilde{s}^{-1}} - \theta_s$.
Define
\begin{equation}\label{f:ast}
    \tau_w^\ast := \tau_{\tilde{s_\ell}}^\ast \ldots \tau_{\tilde{s_1}}^\ast \tau_{\omega^{-1}} \ .
\end{equation}
According to \cite[Cor. 2]{Vig} the element $\tau_w^\ast$ only depends on $w$ (and not on its decomposition) and there is a unique involutive automorphism $\upiota$ of $\Hh$ such that
\begin{equation*}
    \upiota(\tau_w) = (-1)^{\ell(w)} \tau_{w^{-1}}^\ast \qquad\text{for any $w \in \tilde\W$}.
\end{equation*}
Note that
\begin{equation*}
    \upiota(\tau_\omega) = \tau_\omega \quad\text{for any $\omega \in \tilde\Omega$ and}\quad \upiota(\tau_{\tilde s}) = \theta_s - \tau_{\tilde s} \quad\text{for any $s \in S_{aff}$}.
\end{equation*}
Obviously, $\upiota$ restricts to an automorphism of  $\Hh_F^\dagger$ for any facet $F$ contained in $\overline C$.

\begin{rema}\label{rema:involution-iwahori}
The corresponding involution $\upiota'$ on $\Hh'$  carries $\tau'_s$ to $-(\tau'_s+1-q)$ for $s\in S_{aff}$ and fixes $\tau'_\omega$ for all $\omega\in \Omega$. It coincides with the well-known Iwahori-Matsumoto involution \cite[2.2]{JanIM} (which is usually defined when $q$ is invertible in $k$).  For $w\in \W$, there is also an element $\tau'^\ast_w$ defined similarly as \eqref{f:ast} and satisfying $\iota'(\tau'_w)=(-1)^{\ell(w)}\tau'^\ast_{w^{-1}}$.
\end{rema}

\subsection{\label{subsec:bijproof}}

We now study the maps \eqref{f:F-map} and  \eqref{f:F-map-iwahori} and will first introduce a variant of them in Proposition \ref{F-mapBis}. We still work with the fixed facet $F$  contained in $\overline C$.
Extending functions on  ${\mathbf G}_F^\circ(\mathfrak O)$ by zero to $\Gp$ induces ${\mathbf G}_F^\circ(\mathfrak O)$-equivariant embeddings
\begin{equation*}
    \mathbf{X}_F := \ind^{{\mathbf G}_F^\circ(\mathfrak O)}_\I(1) \hookrightarrow \mathbf{X} \quad \textrm{and} \quad   \mathbf{X}'_F := \ind^{{\mathbf G}_F^\circ(\mathfrak O)}_{\I'}(1) \hookrightarrow \mathbf{X}' \ .
\end{equation*}
We introduce the $k$-algebras
\begin{equation*}
    \Hh_F := \End_{k[{\mathbf G}_F^\circ(\mathfrak O)]}(\mathbf{X}_F)^{\mathrm{op}} = [\ind^{{\mathbf G}_F^\circ(\mathfrak O)}_\I(1)]^\I \quad  \textrm{and} \quad  \Hh'_F := \End_{k[{\mathbf G}_F^\circ(\mathfrak O)]}(\mathbf{X}'_F)^{\mathrm{op}} = [\ind^{{\mathbf G}_F^\circ(\mathfrak O)}_{\I'}(1)]^{\I'} \ .
\end{equation*}
They are naturally subalgebras of $\Hh$ and $\Hh'$, respectively, via the extension by zero embedding.
In fact, $\mathbf{X}_F$ is a sub-${\mathbf G}_F^\circ(\mathfrak O)$-representation of the representation $\mathbf{X}_F^\dagger$
of $\Pp_F^\dagger$ introduced in \eqref{XF+}, and $\Hh_F$ is naturally a subalgebra of $\Hh_F^\dagger$. Likewise $\mathbf{X}'_F$ is a sub-${\mathbf G}_F^\circ(\mathfrak O)$-representation of  ${\mathbf{X}'_F}^{\!\!\!\dagger}$, and $\Hh'_F$ is naturally a subalgebra of ${\Hh_F'}^{\!\!\!\dagger}$.

A basis for $\Hh_F$ (resp.\ $\Hh'_F$) is given by $(\tau_w)_{w\in \tilde{  \W}_F}$ (resp.\  $(\tau'_w)_{w\in {\W}_F}$); a basis for $\Hh^\dagger_F$ (resp.\ ${\Hh_F'}^{\!\!\!\dagger}$) is given by $(\tau_w)_{w\in \tilde{\W}_F^\dagger}$ (resp.\ $(\tau'_w)_{w\in {\W}_F^\dagger}$).

\begin{lemm}\phantomsection\label{HF+HF}
\begin{itemize}
\item[i.] The Hecke algebra  $\Hh^\dagger_F$ is free as a left as well as a right $\Hh_F$-module, with basis $(\tau_{\tilde\omega})_{\omega\in \Omega_F}$. In fact, we have $\Hh^\dagger_F = \Hh_F \otimes_{k[\Tp^0/\Tp^1]} k[\tilde{\Omega}_F]$ as $(\Hh_F, k[\tilde{\Omega}_F])$-bimodules. As an algebra, $\Hh^\dagger_F$ is isomorphic to the twisted tensor product where the product is given by
$$
(\tau_w\otimes \omega)(\tau_{w'}\otimes \omega')=\tau_w \tau_{\omega w' \omega^{-1}}\otimes \omega\omega'\textrm{ for $w,w'\in \tilde \W_F$ and $\omega, \omega'\in \tilde{\Omega}_F$}
$$
\item[ii.] The algebra $\Hh_F$ is a direct summand of $\Hh_F^\dagger$ as an $(\Hh_F, \Hh_F)$-bimodule.
\end{itemize}
\end{lemm}
\begin{proof}
We have $\tilde{\W}_F^\dagger / \tilde{\W}_F = \W_F^\dagger / \W_F = \Omega_F$. Moreover, the elements in $\tilde{\Omega}_F$ have length $0$. It proves i. using \eqref{braid}.
To prove ii, note that the direct sum of all $\Hh_F\tau_{\tilde\omega}
$ for $\omega\in \Omega_F$, $\omega\neq 1$ is a  $(\Hh_F, \Hh_F)$-bimodule by definition of $\Omega_F$.
\end{proof}

\begin{prop}\phantomsection\label{prop:free}
\begin{itemize}
\item[i.] The Hecke algebra  $\Hh$ is free as a left as well as a right $\Hh_F^\dagger$-module.
It is also free as a left as well as a right $\Hh_F$-module.
\item[ii.] The algebra $\Hh_F^\dagger$ (resp.\ $\Hh_F$) is a direct summand of $\Hh$ as an $(\Hh_F^\dagger, \Hh_F^\dagger)$-bimodule (resp.\  a $(\Hh_F, \Hh_F)$-bimodule).
\end{itemize}
\end{prop}

\begin{proof}
i. By Lemma \ref{HF+HF}, it is enough to prove that $\Hh$ is free as a module over $\Hh_F^\dagger$.
The braid relations and the property of the system of representatives $\Dd_F^\dagger$ (Lemma \ref{DF+}.i) ensure that $\Hh$ is a free right module over $\Hh_F^\dagger$ with basis
$({\tau_{\tilde{d}}})_{d\in \Dd_F^\dagger}$. Likewise, it is  a free left module over $\Hh_F^\dagger$ with basis
$({\tau_{\tilde{d}^{-1}}})_{d\in \Dd_F^\dagger}$.

ii. We write the proof in the case of $\Hh_F$. Let $M_F$ be the direct sum of the $\tau_{\tilde d}\Hh_F$ for $d\in \Dd_F$, $d\neq 1$.
We prove that it is an  $(\Hh_F, \Hh_F)$-bimodule  by checking that it is stable under left multiplication by $\Hh_F$.
Because of the braid relations \eqref{braid}, $\Hh_F$ is generated by all $\tau_{\tilde w}
$ for $\tilde w\in\tilde \W_F$ and $\ell(\tilde w)\in \{0, 1\}$.
Let $d\in \Dd_F$ with $d\neq 1$.
If $\ell(\tilde w)=0$, then $\tilde w$ is a lift in $\tilde \W$ of  $1$ and   $\tau_{\tilde w}\tau_{\tilde d}=\tau_{\tilde d}\tau_{\tilde d^{-1} \tilde w\tilde d }\in M_F$.
If $\ell(\tilde w)=1$, then $\tilde w = \tilde s$ is a lift in $\tilde \W$ of an element $s$ in $S_{F}$.
\begin{itemize}
\item[a)]  Suppose $\ell(sd)=\ell(d)+1$. Then $\tau_{\tilde s}\tau_{\tilde d} = \tau_{\tilde {s}\tilde d}$ and $sd \neq 1$. If $sd \in \Dd_F$ then $\tau_{\tilde {s}\tilde d} \in M_F$  .
If $sd \not\in \Dd_F$ then, by Proposition \ref{representants} ii., there is a $v \in \W_F$ such that $sd=dv$ and $\ell(sd)=\ell(d)+\ell(v)$ so
that $\tau_{\tilde {s}\tilde {d}}=\tau_{\tilde d}\tau_{\tilde v}$ belongs to $M_F$.

\item[b)] Suppose $\ell(sd)=\ell(d)-1$. Note that $sd$ cannot be $1$ since $d\in \Dd_F$ and $d\neq 1$. Then by
 \cite[Thm.\ 1(ii)]{Vig}  (see also \eqref{quadratic}) we have $\tau_{\tilde s}^2 \in k[\Tp^0 / \Tp^1] \tau_{ \tilde s}+ k[\Tp^0 / \Tp^1]$ and hence
 $\tau_{\tilde s}\tau_{\tilde d}\in k[\Tp^0 / \Tp^1] \tau_{ \tilde d}+ k[\Tp^0 / \Tp^1]\tau_{ \tilde s\tilde d}$.
 Both $d$ and $sd$ lie in $\Dd_F-\{1\}$ by Proposition \ref{representants} ii. So  $\tau_{\tilde s}\tau_{\tilde d}\in M_F$.
\end{itemize}

Similarly, one proves  that
$\Hh$ is the direct sum of the
$(\Hh_F^\dagger,  \Hh_F^\dagger)$-bimodules $\Hh_F^\dagger$ and
$M_F^\dagger$  where $M_F^\dagger$ is defined to be the direct sum  of the $\tau_{\tilde d}\Hh_F^\dagger$ for $d\in \Dd_F^\dagger$, $d\neq 1$.
\end{proof}

\begin{rema}\label{rema:libre}
Recall that the right $\Hh_F^\dagger$-module $\Hh (\epsilon_F)$ was defined in section {\ref{subsubsec:epsilonF}} by
twisting the action of $\Hh_F^\dagger$ by the involution $j_F$. Proposition \ref{prop:free}  implies that the right $\Hh_F^\dagger$-module $\Hh (\epsilon_F)$  is free.
\end{rema}

\begin{rema}\label{rema:iwahori-3}
The results of Lemma \ref{HF+HF} and Proposition \ref{prop:free} are valid with $\Hh'$, $\Hh'_F$, ${\Hh'_F}^{\!\!\!\dagger}$ instead of, respectively, $\Hh$, $\Hh_F$, ${\Hh}_F^\dagger$. In particular,
${\Hh'_F}^{\!\!\!\dagger}$ is free as a left as well as a right $\Hh'_F$-module, with basis $(\tau'_{\omega})_{\omega\in \Omega_F}$, and $\Hh'$ is a free right module over ${\Hh}'^\dagger_F$ with basis $(\tau'_{d})_{d\in \Dd_F^\dagger}$. In the proof of Proposition \ref{prop:free}.ii.b) the relation \eqref{quadratic} is replaced by the classical quadratic relation \eqref{quadratic-iwahori}.
\end{rema}

There are the natural maps
\begin{equation*}
   \eta_F :  \mathbf{X}_F \otimes_{\Hh_F} \Hh_F^\dagger \longrightarrow  \mathbf{X}_F ^\dagger, \:\:\:  \:\:
   \eta'_F :  \mathbf{X}'_F \otimes_{\Hh'_F} {{\Hh}'_F}^{\!\!\!\dagger}  \longrightarrow  {\mathbf{X}'_F}^{\!\!\!\dagger}
\end{equation*}
of $({\mathbf G}_F^\circ(\mathfrak O), \Hh_F^\dagger)$- and
$({\mathbf G}_F^\circ(\mathfrak O), {\Hh'_F}^{\!\!\!\dagger})$-bimodules, respectively,
defined by $f \otimes h  \longmapsto h(f)$.

\begin{lemm} \label{eta}
The maps $\eta_F$ and $\eta'_F$ are bijective.
\end{lemm}

\begin{proof}
By Lemma \ref{HF+HF} it suffices to check that the maps
\begin{equation*}
   \mathbf{X}_F \otimes_{k[\Tp^0 /\Tp^1]} k[\tilde{\Omega}_F]  \longrightarrow  \mathbf{X}_F ^\dagger \quad\textrm{and}\quad
  \mathbf{X}'_F \otimes_{k} k[{\Omega}_F]  \longrightarrow  {\mathbf{X}'_F} ^{\!\!\!\dagger}
\end{equation*}
defined by $f \otimes \omega  \longmapsto \tau_\omega(f)$ and $f \otimes \omega  \longmapsto \tau'_\omega(f)$, respectively, are bijective. This follows from the bijectivity of the maps
$$
\begin{array}{rl}
    {\mathbf G}_F^\circ(\mathfrak O) /\I \times \Omega_F & \longrightarrow \Pp_F^\dagger / \I \cr
    (g\I, \omega) & \longmapsto g \I \hat{\omega} \I = g \hat{\omega} \I\cr
\end{array} \quad\textrm{and}\quad
\begin{array}{rl}
    {\mathbf G}_F^\circ(\mathfrak O) /\I' \times \Omega_F & \longrightarrow \Pp_F^\dagger / \I' \cr
    (g\I', \omega) & \longmapsto g \I '\hat{\omega} \I '= g \hat{\omega} \I'\cr
\end{array}
$$
since $\hat\omega$ normalizes $\I'$ and $\I$ (\cite[Lemma 6]{Vig}) and since $\Pp_F^\dagger \cong {\mathbf G}_F^\circ(\mathfrak O) \rtimes \Omega_F$ by Lemma \ref{bruhat-finite}.
\end{proof}

We now prove the following proposition, which, combined with Lemma  \ref{eta} gives the bijectivity of the maps \eqref{f:F-map} in Lemma \ref{F-map} and of \eqref{f:F-map-iwahori}.

\begin{prop}\label{F-mapBis}
The maps
\begin{equation}\label{f:F-mapBis}
\begin{array}{rl}
\mathbf{X}_F \otimes_{\Hh_F} \Hh & \longrightarrow \mathbf{X}^{\I_F} \\
    f \otimes h & \longmapsto h(f)
\end{array}
\quad\textrm{and}\quad
\begin{array}{rl}  \mathbf{X}'_F \otimes_{\Hh'_F} \Hh'& \longrightarrow \mathbf{X'}^{\I_F} \\
    f \otimes h & \longmapsto h(f)
\end{array}
\end{equation}
are well defined isomorphisms of, respectively, $({\mathbf G}_F^\circ(\mathfrak O), \Hh)$- and $({\mathbf G}_F^\circ(\mathfrak O), \Hh')$-bimodules.
\end{prop}
\begin{proof}
That the maps are well defined, ${\mathbf G}_F^\circ(\mathfrak O)$-equivariant, and  $\Hh$- and $\Hh'$-equivariant, respectively, is obtained
exactly as in Lemma \ref{F-map} for \eqref{f:F-map}.

We prove that the first one is injective. Recall that, by Proposition \ref{prop:free}, the space
$\mathbf{X}_F \otimes_{\Hh_F} \Hh$ decomposes as the direct sum of $\mathbf{X}_F \otimes {\tau_{\tilde{d}^{-1}}}$ for all $d\in \Dd_F$.
Each space $\mathbf{X}_F \otimes \tau_{\tilde{d}^{-1}}$ is mapped by \eqref{f:F-mapBis} onto a space of functions with support in
{${\mathbf G}_F^\circ(\mathfrak O) \hat{\tilde{d}}^{-1} \I$}. By  Remark \ref{repr-I} it is therefore enough to check that
${\tau_{\tilde{d}^{-1}}} :  \mathbf{X}_F\rightarrow \mathbf{X}$ is an injective map for any $d\in \Dd_F$: if $k$ has characteristic $p$ it comes (since a nonzero kernel would have nonzero $\I$-invariants), from the fact that its restriction to $\Hh_F=(\mathbf{X}_F)^\I$ is  injective by \eqref{braid} and  \eqref{additive0};  if $k$ has characteristic different from $p$, it comes from the fact that the basis element $\tau_{\tilde{d}^{-1}}$ is invertible in $\Hh$ (Remark \ref{rema:invertible}).

Now check the surjectivity. An element in $\mathbf{X}^{\I_F}$ is a linear combination of characteristic functions of double cosets of the form
$\I_F \backslash \Gp/ \I$. Since ${\{\hat{\tilde{d}}^{-1}\}_{d \in \Dd_F}}$ is a system of representatives of
${\mathbf G}_F^\circ(\mathfrak O) \backslash \Gp / \I$ and $\I_F$ is normal in   ${\mathbf G}_F^\circ(\mathfrak O)$, such a double coset can be written $g \I_F \hat{\tilde{d}}^{-1} \I$  with $g\in {\mathbf G}_F^\circ(\mathfrak O)$. We have to check that its characteristic function  belongs to the image of \eqref{f:F-mapBis}, which amounts to proving that
the characteristic function of $\I_F \hat{\tilde{d}}^{-1} \I$ does, given that \eqref{f:F-mapBis} is ${\mathbf G}_F^\circ(\mathfrak O)$-equivariant. To this end, we prove that $\I_F \hat{\tilde{d}}^{-1} \I = \I \hat{\tilde{d}}^{-1} \I$ by using Lemma \ref{IC} and the inclusion \eqref{rootsubgroup} which together ensure that
\begin{equation*}
    \I\subset \I_F \Uu_F^0 \subset \I_F \hat{\tilde{ d}}^{-1} \Uu_C \hat{\tilde{d}} \subset \I_F \hat{\tilde{d}}^{-1} \I \hat{\tilde{d}} .
\end{equation*}

For the second map, the argument for the injectivity goes through literally if $k$ has characteristic different from $p$. For the characteristic $p$ case,  note, by a similar argument as in Lemma \ref{lemma:iwahori-1}, that
$\Hh'_F= (\mathbf{X}'_F)^{\I'}$ coincides with  $(\mathbf{X}'_F)^{\I}$ so that $\tau'_{d^{-1}}:  \mathbf{X}'_F\rightarrow \mathbf X'$ is injective for any $d\in \Dd_F$. The argument for surjectivity is valid with $\I'$ instead of $\I$ because $\hat d$ normalizes $\Tp^0$ for any $d\in \Dd_F$.
\end{proof}

\section{Frobenius extensions}\label{sec:Frobenius}

\subsection{}

For the convenience of the reader we recall the formalism of Frobenius extensions (for example, \cite{BF}). Let $R \subseteq S$ be an inclusion of rings and $\alpha$ an automorphism of $R$. For a left, resp.\ right, $R$-module $N$ we denote by ${_\alpha}N$, resp.\ $N_\alpha$, the $R$-module which is the pullback of $N$ along $\alpha$. Then $R \subseteq S$ is called an $\alpha$-Frobenius extension if one of the following equivalent conditions is satisfied:
\begin{itemize}
  \item[a.] The functors $N \longmapsto S \otimes_R N$ and $N \longmapsto \Hom_R(S,{_\alpha}N)$ from left $R$-modules to left $S$-modules are naturally isomorphic.
  \item[b.] $S$ is finitely generated projective as a left $R$-module, and $S \cong \Hom_R(S,{_\alpha}R)$ as $(S,R)$-bimodules.
  \item[c.] $S$ is finitely generated projective as a right $R$-module, and $S \cong \Hom_R(S,R_{\alpha^{-1}})$ as $(R,S)$-bimodules.
\end{itemize}
Moreover, an $\alpha$-Frobenius extension $R \subseteq S$ is called free if $S$ is free as a left as well as a right $R$-module.

\begin{exam}
(\cite[Example (B) after 1.2]{BF}) If $G_1$ is a subgroup of finite index in the group $G_2$ then $k[G_1] \subseteq k[G_2]$ is a free $\id_{k[G_1]}$-Frobenius extension.
\end{exam}

By \cite[Cor.\ 1.2]{BF} we have the following criterion.

\begin{lemm}\label{free-Frob}
Suppose given, for the ring extension $R \subseteq S$ and the automorphism $\alpha$ of $R$, a homomorphism of $(R,R)$-bimodules $\theta : {_\id}S_\alpha \longrightarrow R$ together with elements $x_1, \ldots, x_n, y_1, \ldots, y_n \in S$ such that $S = \sum_{i=1}^n Rx_i = \sum_{i=1}^n y_i R$ and the matrix $(\theta(x_iy_j))_{1 \leq i,j \leq n}$ has a two-sided inverse in ${\rm M}_n(R)$. Then $R \subseteq S$ is a free $\alpha$-Frobenius extension.
\end{lemm}

Our interest in Frobenius extensions comes from the following property. Although it is well known we include a sketch of proof for the convenience of the reader.

\begin{lemm}\label{Frob-injdim}
For any free $\alpha$-Frobenius extension $R \subseteq S$ the rings $R$ and $S$ have the same self-injective dimension.
\end{lemm}

\begin{proof}
First let $N$ be any left $R$-module. By c. the tensor product with $S$ over $R$ of a projective resolution of $N$ is a projective resolution of $S \otimes_R N$. This implies
\begin{equation*}
    \Ext_S^* (S \otimes_R N, S) = \Ext_R^* (N,S) \ .
\end{equation*}
Since $S \cong R^n$ as left $R$-modules we deduce that $\Ext_S^* (S \otimes_R N, S) = \Ext_R^* (N,R)^n$.

On the other hand let $M$ be a left $S$-module. By b. any projective resolution of $M$ as an $S$-module is, at the same time, a projective resolution of $M$ as an $R$-module. We deduce that
\begin{equation*}
    \Ext_S^* (M, S) = \Ext_S^*(M, \Hom_R(S, {_\alpha}R)) = \Ext_R^*(M, {_\alpha}R) = \Ext_R^*(M,R)_{\alpha^{-1}} \ .
\end{equation*}

Since the notion of a Frobenius extension is left-right symmetric an analogous argument works for right modules.
\end{proof}

We go back to the notations of the earlier sections and fix a facet $F$ in $\overline{C}$. We recall that the characteristic functions $(\tau_w)_{w\in \tilde{\W}_F^\dagger}$ form a $k$-basis of $\Hh^\dagger_F$. Moreover, they satisfy the braid relations \eqref{braid}. As before ${\rm Z}$ denotes the connected center of $\Gp$. The subgroup ${\rm Z}\Tp^0 / \Tp^1$ is normal in $\tilde{\W}_F^\dagger$. We fix an element of maximal length $w_0 \in \tilde{\W}_F$ and have the corresponding automorphism of $k$-algebras
\begin{align*}
    \alpha_{w_0} : \qquad k[{\rm Z}\Tp^0 / \Tp^1] & \longrightarrow k[{\rm Z}\Tp^0 / \Tp^1] \\
    {\rm Z}\Tp^0 / \Tp^1 \ni \xi & \longmapsto w_0 \xi w_0^{-1}
\end{align*}
The $k$-algebras $\Hh_F$ (\cite[Prop.\ 3.7]{Tin} and \cite[Thm.\ 2.4]{Saw}) and $\Hh'_F$ (compare \cite[Prop.\ 4.1]{Fay}) are Frobenius algebras. We generalize these results as follows. Let $r$ denote the rank of the torus ${\rm Z}$. Then ${\rm Z} / {\rm Z} \cap \Tp^0$ is a free abelian group of rank $r$. Since the extension $0 \rightarrow {\rm Z} \cap \Tp^0 / {\rm Z} \cap \Tp^1 \rightarrow {\rm Z} / {\rm Z} \cap \Tp^1 \rightarrow {\rm Z} / {\rm Z} \cap \Tp^0 \rightarrow 0$ splits we may fix a subgroup of finite index ${\rm Z}_0 \subseteq {\rm Z}\Tp^0 / \Tp^1$ which is free of rank $r$ and which is central in $\tilde{\W}$. Note that, because the elements in ${\rm Z}\Tp^0 / \Tp^1$ have length zero, the map  $\xi\in {\rm Z}\Tp^0 / \Tp^1\mapsto \tau_\xi$ yields  an injective morphism of $k$-algebras  $k[{\rm Z}\Tp^0 / \Tp^1]\rightarrow \Hh$. We identify $k[{\rm Z}\Tp^0 / \Tp^1]$ with its image in $\Hh$. Likewise, we identify $k[{\rm Z} / {\rm Z} \cap \Tp^0]$ with its image in $\Hh'$.

\begin{prop}\phantomsection\label{Hecke-Frob}
\begin{itemize}
\item[i.] $k[{\rm Z}\Tp^0 / \Tp^1] \subseteq \Hh_F^\dagger$ is a free $\alpha_{w_0}$-Frobenius extension.
\item[ii.] $k[{\rm Z}_0] \subseteq \Hh_F^\dagger$ is a free $\id_{k[{\rm Z}_0]}$-Frobenius extension.
\item[iii.] $k[{\rm Z} / {\rm Z} \cap \Tp^0]\subseteq {{\Hh}'_F}^{\!\!\!\dagger}$ is a free $\id_{k[{\rm Z} / {\rm Z} \cap \Tp^0]}$-Frobenius extension.
\end{itemize}
\end{prop}

\begin{proof}
i. We consider the map
\begin{align*}
    \theta : \qquad\qquad \Hh^\dagger_F & \longrightarrow k[{\rm Z}\Tp^0 / \Tp^1] \\
    \sum_{w \in \tilde{\W}_F^\dagger} a_w \tau_w & \longmapsto \sum_{\xi \in {\rm Z}\Tp^0 / \Tp^1} a_{\xi w_0} \xi \ .
\end{align*}
Using the fact that the length function on $\tilde{\W}_F^\dagger$ is $\tilde{\Omega}_F$-bi-invariant the braid relations \eqref{braid} imply that $\theta$ in fact is a homomorphism of $(k[{\rm Z}\Tp^0 / \Tp^1], (k[{\rm Z}\Tp^0 / \Tp^1])$-bimodules
\begin{equation*}
    \theta : \  {_\id}(\Hh^\dagger_F)_{\alpha_{w_0}} \longrightarrow k[{\rm Z}\Tp^0 / \Tp^1] \ .
\end{equation*}
Let $[\W_F^\dagger] \subseteq \W_F^\dagger$ denote a set of representatives for the cosets of the image of ${\rm Z}$ in $\W_F^\dagger$. This is a finite set.
We know from Section \ref{subsec:bijproof} that $(\tau_{\tilde{w}})_{w \in [\W_F^\dagger]}$ and $(\tau_{\tilde{w}^{-1} w_0})_{w \in [\W_F^\dagger]}$ are bases of $\Hh^\dagger_F$ as a left as well as a right $k[{\rm Z}\Tp^0 / \Tp^1]$-module. In order to apply Lemma \ref{free-Frob} it remains to show that the matrix $(\theta(\tau_{\tilde{v}} \tau_{\tilde{w}^{-1} w_0}))_{v,w \in [\W_F^\dagger]}$ is invertible. Since $\ell(\tilde{w}^{-1} w_0) = \ell(w_0) - \ell(\tilde{w}^{-1})$ (\cite[VI.1.6 Cor.\ 3]{Bki-LA}) it follows from \eqref{braid} that $\theta(\tau_{\tilde{w}} \tau_{\tilde{w}^{-1} w_0}) = \theta(\tau_{w_0}) = 1$. By \cite[Thm.\ 1(ii)]{Vig} (see \eqref{quadratic}) we have $\tau_\sigma^2 \in k[\Tp^0 / \Tp^1] \tau_\sigma + k[\Tp^0 / \Tp^1]$ for any lift $\sigma \in \tilde{\W}$ of an element in $S_{aff}$. This together with the braid relations implies that

\begin{equation*}
    \tau_{\tilde{v}} \tau_{\tilde{w}^{-1} w_0} \in \sum_{v' \leq v} k[\Tp^0 / \Tp^1] \tau_{\widetilde{v'w^{-1}}w_0}
\end{equation*}
where $\leq$ denotes the Bruhat order on $\W_F^\dagger$ (compare \cite[2.4.9]{Born}). We see that
\begin{equation*}
    \theta(\tau_{\tilde{v}} \tau_{\tilde{w}^{-1} w_0}) = 0 \qquad\text{unless $w \leq v$}.
\end{equation*}
In other words the matrix in question is lower triangular.

ii. This is a consequence of i., the above example, and the transitivity of Frobenius extensions (\cite[Prop.\ 1.3]{BF}).

iii. The proof of iii. is a somewhat simpler copy of the proof of i. Here we identify $w_0$ with its image  in $\W_F$. The map
\begin{align*}
    \theta' : \qquad\qquad {{\Hh}'_F}^{\!\!\!\dagger} & \longrightarrow k[{\rm Z}/{\rm Z} \cap \Tp^0] \\
    \sum_{w \in \W_F^\dagger} a_w \tau'_w & \longmapsto \sum_{\xi \in {\rm Z}/{\rm Z} \cap \Tp^0} a_{\xi w_0} \xi
\end{align*}
is a homomorphism of $(k[{\rm Z}/{\rm Z}\cap \Tp^0], k[{\rm Z}/{\rm Z}\cap \Tp^0])$-bimodules $    \  {_\id}({{\Hh}'_F}^{\!\!\!\dagger})_{\id} \longrightarrow k[{\rm Z}/{\rm Z}\cap \Tp^0 ].$ As in the proof of ii.,
$(\tau'_{{w}})_{w \in [\W_F^\dagger]}$ and $(\tau'_{{w}^{-1} w_0})_{w \in [\W_F^\dagger]}$ are bases of ${{\Hh}'_F}^{\!\!\!\dagger}$ as a left as well as a right $k[{\rm Z}/{\rm Z}\cap \Tp^0]$-module,  and we verify that the matrix $(\theta'(\tau'_{{v}} \tau'_{{w}^{-1} w_0}))_{v,w \in [\W_F^\dagger]}$ is invertible. Using  \eqref{braid-iwahori} we obtain $\theta'(\tau'_{{w}} \tau'_{ {w}^{-1} w_0}) = \theta'(\tau'_{w_0}) = 1$. By the quadratic relation \eqref{quadratic-iwahori}, we have $(\tau'_s)^2 \in k \tau'_s + k$ for any $s\in S_{aff}$. This together with the braid relations implies that $\tau'_{ {v}} \tau'_{ {w}^{-1} w_0} \in \sum_{v' \leq v} k \tau'_{{v'w^{-1}}w_0}$ where $\leq$ denotes the Bruhat order on $\W_F^\dagger$. Therefore, we have $\theta'(\tau'_{ {v}} \tau'_{ {w}^{-1} w_0}) = 0$ unless $w \leq v$, and the matrix in question is lower triangular.
\end{proof}

\begin{prop}\label{F-injdim}
Let $F$ be a facet of $\mathscr{X}$. The $k$-algebras $\Hh_F^\dagger$ and  ${{\Hh}'_F}^{\!\!\!\dagger}$ are left and right noetherian and have self-injective dimension $r$.
\end{prop}
\begin{proof}
We may assume that $F$ is contained in $\overline{C}$.  Recall that $\Hh_F^\dagger$ (respectively ${{\Hh}'_F}^{\!\!\!\dagger}$)
is a free $\id_{k[{\rm Z}_0]}$-Frobenius extension (respectively a free $\id_{k[{\rm Z} / {\rm Z} \cap \Tp^0]}$-Frobenius extension) by Proposition \ref{Hecke-Frob}. But  $k[{\rm Z}_0]$ and  $k[{\rm Z} / {\rm Z} \cap \Tp^0]$ are Laurent polynomial rings in $r$ variables over $k$. Hence they are both  noetherian and regular of global dimension $r$. The former implies that $\Hh_F^\dagger$ and ${{\Hh}'_F}^{\!\!\!\dagger}$ are  noetherian. The latter implies that the  self-injective dimension of $k[{\rm Z}_0]$  and of $k[{\rm Z} / {\rm Z} \cap \Tp^0]$ is equal to $r$ (compare \cite[Bem.\ 10.16]{Sch}). We conclude  that $\Hh_F^\dagger$ and ${{\Hh}'_F}^{\!\!\!\dagger}$ have  self-injective dimension $r$ using Lemma \ref{Frob-injdim}.
\end{proof}

\subsection{\label{subsec:explicit-computation}}

If we assume that $\Gp$ is semisimple then ${\rm Z} = 1$ and Proposition \ref{Hecke-Frob}.ii says that $\Hh_F^\dagger$ is a Frobenius algebra. More precisely, assuming that $F \subseteq \overline{C}$  the proof of Proposition \ref{Hecke-Frob} tells us that, fixing an element $w_F$ of maximal length in $\tilde\W_F^\dagger$ and using the linear form
\begin{align*}
    \delta_{w_F} : \qquad\qquad \Hh_F^\dagger & \longrightarrow k \\
    \sum_{w \in \tilde\W_F^\dagger} a_w \tau_w & \longmapsto a_{w_F} \ ,
\end{align*}
we have the isomorphism of left $\Hh_F^\dagger$-modules $\Hh_F^\dagger \cong \Hom_k(\Hh_F^\dagger,k)$ sending $1$ to $\delta_{w_F}$. Composing with the adjunction isomorphism
\begin{align*}
    \Hom_{\Hh_F^\dagger} (M, \Hom_k(\Hh_F^\dagger,k)) & \xrightarrow{\;\cong\;}\Hom_k(M,k) \\
    f & \longrightarrow [x \mapsto f(x)(1)]
\end{align*}
we obtain, for any left $\Hh_F^\dagger$-module $M$, the $k$-linear isomorphism
\begin{equation}\label{f:k-dual}
    \Hom_{\Hh_F^\dagger} (M, \Hh_F^\dagger) \cong \Hom_{\Hh_F^\dagger} (M, \Hom_k(\Hh_F^\dagger,k)) \cong \Hom_k(M,k) \ , \ f \longmapsto \delta_{w_F} \circ f \ .
\end{equation}
For later purposes we need to explicitly determine the inverse of this isomorphism. This will make use of  the defining relations in $\Hh$ and the canonical involution $\upiota$ introduced in \ref{subsec:defining-relations}.

\begin{lemm}\label{explicit-inverse}
Let $F$ be a facet in $\overline{C}$, $M$ be any left $\Hh_F^\dagger$-module, and $f \in \Hom_{\Hh_F^\dagger} (M, \Hh_F^\dagger)$; setting $f_0 := \delta_{w_F} \circ f$ we have
\begin{equation*}
    f(x) = \sum_{w \in \tilde\W_F^\dagger} f_0(\tau_{w w_F^{-1}}^\ast x) \tau_w =
    \sum_{w \in \tilde\W_F^\dagger} f_0((-1)^{\ell(w_F)-\ell(w)} \upiota(\tau_{w_F w^{-1}}) x) \tau_w  \qquad\text{for any $x \in M$}.
\end{equation*}
\end{lemm}
\begin{proof}
Write $f(x) = \sum_{w \in \tilde\W_F^\dagger} f_w(x) \tau_w$. For $\omega \in \tilde\Omega_F$ we compute
\begin{equation*}
    \sum_{w \in \tilde\W_F^\dagger} f_w(\tau_{\omega^{-1}} x) \tau_w = f(\tau_{\omega^{-1}} x) = \tau_{\omega^{-1}} f(x)
    = \sum_{w \in \tilde\W_F^\dagger} f_w(x) \tau_{\omega^{-1}} \tau_w
    = \sum_{w \in \tilde\W_F^\dagger} f_w(x) \tau_{\omega^{-1} w}
    = \sum_{w \in \tilde\W_F^\dagger} f_{\omega w}(x) \tau_w
\end{equation*}
and obtain
\begin{equation}\label{f:inverse1}
    f_{\omega w}(x) = f_w(\tau_{\omega^{-1}} x) = f_w(\tau_\omega^\ast x) \ .
\end{equation}
For $s \in S_{aff}$ we observe that $\tau_{\tilde{s}^{-1}} \tau_{\tilde{s}} = q + \tau_{\tilde{s}^{-1}} \theta_s = q + \theta_s \tau_{\tilde{s}}$ and we compute
\begin{align*}
    & \sum_{w \in \tilde\W_F^\dagger} f_w(\tau_{\tilde{s}^{-1}} x) \tau_w  = f(\tau_{\tilde{s}^{-1}} x) = \tau_{\tilde{s}^{-1}} f(x)
    = \sum_{w \in \tilde\W_F^\dagger} f_w(x) \tau_{\tilde{s}^{-1}} \tau_w \\
    & = \sum_{w \in \tilde\W_F^\dagger, \ell(\tilde{s} w) \geq \ell(w)} f_w(x) \tau_{\tilde{s}^{-1} w}
    + \sum_{w \in \tilde\W_F^\dagger, \ell(\tilde{s} w) < \ell(w)} f_w(x) \tau_{\tilde{s}^{-1}} \tau_{\tilde{s}} \tau_{\tilde{s}^{-1} w} \\
    & = \sum_{w \in \tilde\W_F^\dagger, \ell(\tilde{s} w) \geq \ell(w)} f_w(x) \tau_{\tilde{s}^{-1} w}
    + \sum_{w \in \tilde\W_F^\dagger, \ell(\tilde{s} w) < \ell(w)} f_w(x) (q + \theta_s \tau_{\tilde{s}}) \tau_{\tilde{s}^{-1} w} \\
    & = \sum_{w \in \tilde\W_F^\dagger} f_w(x) \tau_{\tilde{s}^{-1} w}
    + (q-1) \sum_{w \in \tilde\W_F^\dagger, \ell(\tilde{s} w) < \ell(w)} f_w(x)  \tau_{\tilde{s}^{-1} w} +
    \sum_{w \in \tilde\W_F^\dagger, \ell(\tilde{s} w) < \ell(w)} f_w(x) \theta_s \tau_w \\
    & =
    \sum_{w \in \tilde\W_F^\dagger} f_{\tilde{s} w}(x) \tau_w
    + (q-1) \sum_{w \in \tilde\W_F^\dagger, \ell(\tilde{s} w) > \ell(w)} f_{\tilde{s} w}(x)  \tau_w +
    \sum_{w \in \tilde\W_F^\dagger, \ell(\tilde{s} w) < \ell(w)} f_w(x) \theta_s \tau_w \\
    & =
    \sum_{w \in \tilde\W_F^\dagger} f_{\tilde{s} w}(x) \tau_w
    + (q-1) \sum_{w \in \tilde\W_F^\dagger, \ell(\tilde{s} w) > \ell(w)} f_{\tilde{s} w}(x)  \tau_w +
    \sum_{w \in \tilde\W_F^\dagger, \ell(\tilde{s} w) < \ell(w)} \sum_{t \in \Tp_s(\mathbb{F}_q)} f_w(x) \tau_{tw} \\
    & =
    \sum_{w \in \tilde\W_F^\dagger} f_{\tilde{s} w}(x) \tau_w
    + (q-1) \sum_{w \in \tilde\W_F^\dagger, \ell(\tilde{s} w) > \ell(w)} f_{\tilde{s} w}(x)  \tau_w +
    \sum_{w \in \tilde\W_F^\dagger, \ell(\tilde{s} w) < \ell(w)} (\sum_{t \in \Tp_s(\mathbb{F}_q)} f_{t^{-1}w}(x)) \tau_w \ .
\end{align*}
It follows that, if $\ell(\tilde{s} w) < \ell(w)$, then we have
\begin{equation}\label{f:inverse2}
\begin{split}
    f_{\tilde{s} w}(x) & = f_w(\tau_{\tilde{s}^{-1}} x) - \sum_{t \in \Tp_s(\mathbb{F}_q)} f_{t^{-1}w}(x) = f_w(\tau_{\tilde{s}^{-1}} x) - \sum_{t \in \Tp_s(\mathbb{F}_q)} f_w(\tau_t x) \\
    & = f_w((\tau_{\tilde{s}^{-1}} - \theta_s)x) = f_w(\tau_{\tilde{s}}^\ast x) \ .
\end{split}
\end{equation}
where the second identity uses \eqref{f:inverse1}. The equations \eqref{f:inverse1} and \eqref{f:inverse2} together imply inductively the assertion.
\end{proof}

\begin{rema}\label{rema:iwahori-4}
We remark that, replacing $w_F$ with its image in $\W_F^\dagger$, we get an analogous linear form $\delta'_{w_F}: {{\Hh}'_F}^{\!\!\!\dagger} \rightarrow k$ as well as a formula for ${{\Hh}'_F}^{\!\!\!\dagger}$-modules which is analogous to the one in Lemma \ref{explicit-inverse}. The calculation makes use of the classical quadratic relations \eqref{quadratic-iwahori}.
\end{rema}

\section{Modules, projective dimension, and duality}\label{subsec:comparison}

\subsection{\label{subsec:GP}}

In the course of this paper we have established two kinds of functorial exact resolutions.

$\mathbf{1.}$  For any left $\Hh$-module $\mathfrak{m}$ we have  as a consequence of Theorems \ref{I-resolution} and  \ref{theo:freeresolution}, the exact sequence of $\Hh$-modules
      \begin{equation}\label{f:M-resolution}
    0 \longrightarrow C_c^{or} (\mathscr{X}_{(d)}, \cX)^\I \otimes_\Hh \mathfrak{m} \xrightarrow{\partial \otimes id} \ldots \xrightarrow{\partial \otimes id} C_c^{or} (\mathscr{X}_{(0)}, \cX)^\I \otimes_\Hh \mathfrak{m} \xrightarrow{\epsilon \otimes id} \mathfrak{m} \longrightarrow 0 \ .
\end{equation}
Moreover, by Theorem \ref{theo:freeresolution}, each term in this resolution is a finite direct sum of $\Hh$-modules of the form $\Hh (\epsilon_F) \otimes_{\Hh_F^\dagger} \mathfrak{m}$.

$\mathbf{2.}$  In the case that the left $\Hh$-module $\mathfrak{m}$ is of the form $\mathfrak{m} = \mathbf V^\I$ for some smooth representation $\mathbf V$ of $\Gp$ we have, by Theorem \ref{I-resolutionV}, the exact sequence of $\Hh$-modules
\begin{equation}\label{f:VI-resolution}
    0 \longrightarrow C_c^{or} (\mathscr{X}_{(d)}, \cV)^\I \xrightarrow{\;\partial\;} \ldots \xrightarrow{\;\partial\;} C_c^{or} (\mathscr{X}_{(0)}, \cV)^\I \xrightarrow{\;\epsilon\;} \mathfrak{m} = \mathbf V^\I \longrightarrow 0 \ .
\end{equation}

We will show that \eqref{f:VI-resolution} is naturally isomorphic to \eqref{f:M-resolution}.\\

The $\Gp$-equivariant map
\begin{align*}
    \mathbf{X} \otimes_\Hh \mathbf{V}^\I & \longrightarrow \mathbf{V} \\
    \1_{g\I} \otimes v & \longmapsto gv
\end{align*}
(where $\1_{g\I}$ denotes the characteristic function of the coset $g\I$) induces a homomorphism of $\Gp$-equivariant coefficient systems $\{ \zeta_F : \mathbf X^{\I_F} \otimes_\Hh\mathbf V^{\I} \rightarrow \mathbf V^{\I_F} \}_F$ and therefore a $\Gp$-equivariant map of complexes
\begin{equation}\label{eq:noniso}
C_c^{or} (\mathscr{X}_{(i)}, \cX)\otimes_\Hh \mathbf V^\I = C_c^{or} (\mathscr{X}_{(i)}, \cX \otimes_\Hh \mathbf V^\I) \longrightarrow  C_c^{or} (\mathscr{X}_{(i)}, \cV)\ .
\end{equation}
As the subsequent remark shows, this map is not an isomorphism in general.

\begin{rema}
\begin{itemize}
\item[1.] The map $\zeta_C$ is the natural identification $\Hh \otimes_\Hh\mathbf V^{\I}\simeq \V^\I$.
\item[2.] The map $\zeta_F$ is not surjective in general. Suppose that $F \subseteq \overline{C}$. By Proposition \ref{F-mapBis} we have $\mathbf{X} ^{\I_F}\otimes_\Hh \mathbf{V}^\I \cong \mathbf{X}_F \otimes_{\Hh_F} \mathbf{V}^\I$ which implies that the image of $\zeta_F$ is the sub-${\mathbf G}_F^\circ(\mathfrak O)$-representation of $\mathbf V^{\I_F}$ generated by $\mathbf V^{\I}$. For example, if $k$ has  characteristic $p$ and  $\Gp={\rm GL}_2(\mathbb Q_p)$ choose the same supersingular representation  $\mathbf V$ as in Remark \ref{rema:about-exactness} and  $F$  to be the vertex $x_0$. Then
$\zeta_F$ is not surjective.

\item[3.] Let $F\subseteq \overline C$.
 If the cardinality of the finite reductive group
${\mathbf G}_F^\circ(\mathfrak O)/\I_F$
 is invertible in $k$, it is classical to establish that
the functors $\mathbf{X}_F\otimes_{\Hh_F}  .$ and $\Hom_{{\mathbf G}_F^\circ(\mathfrak O)}( \mathbf{X}_F, \, .\,)$ are quasi-inverse functors between the category of $\Hh_F$-modules and the category of representations of ${\mathbf G}_F^\circ(\mathfrak O)$ generated by their $\I$-invariant subspace (note that the latter actually are representations of the finite reductive group
${\mathbf G}_F^\circ(\mathfrak O)/\I_F$). The map $\zeta_F$ is  then injective.

Suppose $k$ has   characteristic $p$ and $\Gp$ has type $A_n$.
The above mentioned functors are quasi-inverse equivalences if and only if  $F$ has codimension $0$, or $F$ has codimension $1$ and $q=p$,  or $F$ has codimension $2$ and $q=2$ (\cite[Thm 4.17]{OS}). In those cases, $\zeta_F$ is injective. In the other cases, the functor
${\mathbf X}_F\otimes _{\Hh_F}.$ is not even exact.
\end{itemize}
\end{rema}

On the other hand there is the following trivial observation.

\begin{rema}\label{chamberiso}
For any chamber $D$ of $\mathscr{X}$ the map $\zeta_D$ is an isomorphism.
\end{rema}
\begin{proof}
Pick an element $g \in \Gp$ such that $D = gC$. Then $\mathbf{X}^{\I_D} = g \mathbf{X}^\I$ and $\mathbf{V}^{\I_D} = g \mathbf{V}^\I$. Modulo the identification $\mathbf{X}^\I \otimes_\Hh \mathbf{V}^\I = g \mathbf{X}^\I$ the map $\zeta_D$ is the identity.
\end{proof}

Of course, the map \eqref{eq:noniso} restricts to a homomorphism of complexes of $\Hh$-modules from \eqref{f:M-resolution} to \eqref{f:VI-resolution}.

\begin{prop}\label{prop:isoV}
The exact complexes
\begin{align*}\eqref{f:M-resolution}\qquad &
    0 \longrightarrow C_c^{or} (\mathscr{X}_{(d)}, \cX)^\I \otimes_\Hh \mathbf V^\I \xrightarrow{\;\partial\otimes id\;} \ldots \xrightarrow{\;\partial\otimes id\;} C_c^{or} (\mathscr{X}_{(0)}, \cX)^\I \otimes_\Hh \mathbf V^\I \xrightarrow{\;\epsilon\otimes id\;} \mathbf V^\I \longrightarrow 0 \cr \textrm{and}\qquad&\cr
\eqref{f:VI-resolution}\qquad&
    0 \longrightarrow C_c^{or} (\mathscr{X}_{(d)}, \cV)^\I \xrightarrow{\;\partial\;} \ldots \xrightarrow{\;\partial\;} C_c^{or} (\mathscr{X}_{(0)}, \cV)^\I \xrightarrow{\;\epsilon\;} \mathbf V^\I \longrightarrow 0\
\end{align*}
are isomorphic.
\end{prop}

\begin{proof}
The isomorphism in Proposition \ref{prop:isocomplex} is functorial. It therefore suffices to show that the corresponding homomorphism of complexes
\begin{equation*}
    C_c^{or} (\mathscr{A}_{(i)}, \cX^\I) \otimes_\Hh \mathbf V^\I = C_c^{or} (\mathscr{A}_{(i)}, \cX^\I \otimes_\Hh \mathbf V^\I) \longrightarrow C_c^{or} (\mathscr{A}_{(i)}, \cV^\I)
\end{equation*}
is bijective. As we have seen in the proof of Theorem \ref{I-resolutionV} the coefficient systems $\cX^\I$ and $\cV^\I$ on $\mathscr{A}$ only involve subspaces of fixed vectors with respect to the groups $\I_D$ where $D$ runs over the chambers in $\mathscr{A}$. The bijectivity therefore is an immediate consequence of Remark \ref{chamberiso}.
\end{proof}

\subsection{}

Left and right noetherian rings which are of finite left and right self-injective dimension are also called \emph{Gorenstein rings}. A left module $M$ over a Gorenstein ring $A$ is called \emph{Gorenstein projective} if $\Ext_A^i(M, P) = 0$ for any projective $A$-module $P$ and any $i \geq 1$. Obviously any projective module is Gorenstein projective.

We recall from \cite[Theorem 4]{Vig} that $\Hh$ is left and right noetherian and hence is a Gorenstein ring by Theorem \ref{theo:Gorenstein}.

\begin{lemm}\label{Gorproj-res}
Suppose that $\Gp$ is semisimple. For any left $\Hh$-module $\mathfrak{m}$ the exact resolution \eqref{f:M-resolution} consists of Gorenstein projective modules.
\end{lemm}
\begin{proof}
By Theorem \ref{theo:freeresolution} it suffices to show that the left $\Hh$-modules
\begin{equation*}
    \Hh (\epsilon_F) \otimes_{\Hh_F^\dagger} \Hh \otimes_\Hh \mathfrak{m} = \Hh (\epsilon_F) \otimes_{\Hh_F^\dagger} \mathfrak{m}
\end{equation*}
are Gorenstein projective for any facet $F\subseteq \overline C$. As in the proof of Lemma \ref{induced}.ii (using Remark \ref{rema:libre}) we have
\begin{equation*}
    \Ext_\Hh^i(\Hh (\epsilon_F) \otimes_{\Hh_F^\dagger} \mathfrak{m}, P) = \Ext_{\Hh_F^\dagger}^i((\epsilon_F)\mathfrak{m},P)
\end{equation*}
for any $i \geq 0$. By our assumption on the center of $\Gp$ the algebra $\Hh_F^\dagger$ is a Frobenius algebra (Proposition \ref{Hecke-Frob}.ii). Over such an algebra every projective module is injective (cf.\ \cite[Thm. 15.9]{Lam}). Since $\Hh$ is free as a left $\Hh_F^\dagger$-module (Proposition \ref{prop:free}), $P$ viewed as an $\Hh_F^\dagger$-module is projective as well, and hence is injective. It follows that $\Ext_{\Hh_F^\dagger}^i((\epsilon_F)\mathfrak{m},P) = 0$ for any $i \geq 1$.
\end{proof}

\begin{rema}\label{rema:Gpr-iwahori}
Similarly,  \eqref{f:H'-H'-bimod}  being a resolution of $\Hh'$ by free right $\Hh'$-modules, tensoring it by  a left $\Hh'$-module $\m$ yields an exact resolution
\begin{equation}\label{f:m'-resolution}
    0 \longrightarrow \bigoplus_{F \in \mathscr{F}_d} \Hh' (\epsilon_F) \otimes_{\Hh_F'^\dagger} \mathfrak{m} \xrightarrow{\ \partial\ } \ldots \xrightarrow{\ \partial\ } \bigoplus_{F \in \mathscr{F}_0} \Hh' (\epsilon_F) \otimes_{\Hh_F'^\dagger} \mathfrak{m} \longrightarrow \mathfrak{m} \longrightarrow 0
\end{equation}
for $\m$. Suppose that $\Gp$ is semisimple.
Then the argument of the previous lemma goes through using Proposition \ref{Hecke-Frob}.iii (and Remark \ref{rema:iwahori-3} for the freeness of $\Hh'$ over $\Hh_F'^\dagger$ for any facet $F\subseteq\overline C$). Thus,  \eqref{f:m'-resolution}  is an exact resolution of $\m$ by Gorenstein projective modules.
\end{rema}

\subsection{Differentials}\label{subsec:differentials}

Note that this paragraph does not require the hypothesis of semi-simplicity for $\Gp$. Let $\m$ be a left $\Hh$-module. In this section we explicitly compute the differentials in the exact resolution
\begin{equation}\label{f:m-resolution}
    0 \longrightarrow \bigoplus_{F \in \mathscr{F}_d} \Hh (\epsilon_F) \otimes_{\Hh_F^\dagger} \mathfrak{m} \xrightarrow{\ \partial\ } \ldots \xrightarrow{\ \partial\ } \bigoplus_{F \in \mathscr{F}_0} \Hh (\epsilon_F) \otimes_{\Hh_F^\dagger} \mathfrak{m} \longrightarrow \mathfrak{m} \longrightarrow 0
\end{equation}
of $\mathfrak{m}$ derived from \eqref{f:M-resolution}. We emphasize that the isomorphism between \eqref{f:m-resolution} and \eqref{f:M-resolution} obviously depends on the choice of the sets $\mathscr{F}_i$ but also on the choice of an orientation $(F,c_F)$ of any $F \in \mathscr{F}_i$.

For any facet $F \subseteq \overline{C}$ of dimension $i \geq 1$ we define
\begin{equation*}
    \mathscr{F}_{i-1}(F) := \{F' \in \mathscr{F}_{i-1} : F'\ \text{is $\W$-equivalent to a facet contained in}\ \overline{F} \}.
\end{equation*}
By Remark \ref{rema:Omega} any two $\W$-equivalent facets contained in $\overline{C}$ already are $\Omega$-equivalent. Therefore, the facets of dimension $i-1$ contained in $\overline{F}$ are of the form
\begin{equation*}
    \omega  F' \qquad\text{for $F' \in \mathscr{F}_{i-1}(F)$ and certain $\omega\in \Omega/\Omega_{F'}$.}
\end{equation*}
We pick, for any $F' \in \mathscr{F}_{i-1}(F)$, a subset $\Omega(F,F') \subseteq \Omega$ of elements which are pairwise distinct modulo $\Omega_{F'}$ and such that
\begin{equation*}
    \{F'' \subseteq \overline{F} :\ \text{$F''$ has dimension $i-1$} \} = \{ \omega F' : F' \in \mathscr{F}_{i-1}(F) , \omega \in \Omega(F,F') \}.
\end{equation*}
For any element $\omega F'$ in the right hand side we introduce the sign $\epsilon(F,F',\omega) \in \{\pm 1\}$ by the requirement that
\begin{equation*}
    c_{F,\omega F'} = \epsilon(F,F',\omega) \omega c_{F'} \ ,
\end{equation*}
where $c_{F,\omega F'}$ is the orientation induced by $(F,c_F)$ on $\omega F'$.

\begin{prop}\label{differential}
The differential $\partial$ in \eqref{f:m-resolution} is given on $\Hh (\epsilon_F) \otimes_{\Hh_F^\dagger} \mathfrak{m}$, for any $1 \leq i \leq d$ and any $F \in \mathscr{F}_i$, by
\begin{equation*}
    \partial (1 \otimes x) = \sum_{F' \in \mathscr{F}_{i-1}(F)} \sum_{\omega \in \Omega(F,F')} \epsilon(F,F',\omega) \tau_{\tilde \omega}\otimes \tau_{\tilde \omega^{-1}} x \qquad\text{for any $x \in \mathfrak{m}$.}
\end{equation*}
\end{prop}
\begin{proof}
(One easily checks that each summand on the right hand side only depends on $F'$ and the coset $\omega \Omega_{F'}$.) It suffices to treat the universal case $\mathfrak{m} = \Hh$ and to show that
\begin{equation*}
    \partial (1 \otimes 1) = \sum_{F' \in \mathscr{F}_{i-1}(F)} \sum_{\omega \in \Omega(F,F')} \epsilon(F,F',\omega) \tau_{\tilde \omega}\otimes \tau_{\tilde \omega^{-1}} \ .
\end{equation*}
By section {\ref{subsubsec:iso}} the embedding $\Hh (\epsilon_F) \otimes_{\Hh_F^\dagger} \Hh \hookrightarrow C_c^{or}(\mathscr{X}_{(i)}, \cX)^\I$ sends $1 \otimes 1$ to the unique oriented chain $f_{(F,c_F)}$ supported on $F$ and with value ${\rm char}_{\I}$ at $(F,c_F)$. The definition of the differential in the oriented chain complex $C_c^{or}(\mathscr{X}_{(i)}, \cX)$ implies
\begin{equation*}
    \partial(f_{(F,c_F)})((F'',c'')) =
    \begin{cases}
    \pm {\rm char}_{\I} & \text{if $F'' \subseteq \overline{F}$ and $c'' = \pm c_{F,F''}$}, \\
    0 & \text{otherwise}.
    \end{cases}
\end{equation*}
 From our initial discussion we deduce
\begin{equation*}
    \partial(f_{(F,c_F)}) = \sum_{F' \in \mathscr{F}_{i-1}(F)} \sum_{\omega \in \Omega(F,F')} \tilde{f}_{F',\omega} = \sum_{F' \in \mathscr{F}_{i-1}(F)} \sum_{\omega \in \Omega(F,F')} \epsilon(F,F',\omega) f_{F',\omega}
\end{equation*}
where $\tilde{f}_{F',\omega}$, resp.\ $f_{F',\omega}$, is the unique oriented chain supported on $\omega F'$ with value ${\rm char}_\I$ on $(\omega F', c_{F,\omega F'}) = (\omega F', \epsilon(F,F',\omega) \omega c_{F'})$, resp.\ on $\omega(F',c_{F'})$. One checks that $f_{F',\omega}$ is the image of $\tau_{\tilde \omega}\otimes \tau_{\tilde \omega^{-1}}$ under the embedding $\Hh (\epsilon_{F'}) \otimes_{\Hh_{F'}^\dagger} \Hh \hookrightarrow C_c^{or}(\mathscr{X}_{(i-1)}, \cX)^\I$.
\end{proof}

The description of the highest differential can be somewhat simplified by making, without loss of generality, the following choice of orientations
\begin{equation*}
    c_F := c_{C,F} \qquad\text{for any $F \in \mathscr{F}_{d-1}$}.
\end{equation*}

\begin{coro}\label{highest-diff}
Under the above simplifying assumption we have
\begin{align*}
    \Hh (\epsilon_C) \otimes_{\Hh_C^\dagger} \mathfrak{m} & \xrightarrow{\ \partial\ }  \bigoplus_{F \in \mathscr{F}_{d-1}} \Hh (\epsilon_F) \otimes_{\Hh_F^\dagger} \mathfrak{m} \\
    1 \otimes x & \longmapsto \sum_{F \in \mathscr{F}_{d-1}} \sum_{\omega \in \Omega/\Omega_F} \epsilon_C(\hat{\tilde{\omega}}) \tau_{\tilde \omega}\otimes \tau_{\tilde \omega^{-1}} x \ .
\end{align*}
\end{coro}
\begin{proof}
Obviously $\mathscr{F}_{d-1}(C) = \mathscr{F}_{d-1}$, and $\Omega(C,F)$ is a set of representatives for all cosets in $\Omega/\Omega_F$. The defining equation for the sign becomes $c_{C,\omega F} = \epsilon(C,F,\omega) \omega c_{C,F}$. But $\omega c_{C,F}$ is induced by $\omega c_C = \epsilon_C(\hat{\tilde{\omega}}) c_C$. It follows that $\epsilon(C,F,\omega) = \epsilon_C(\hat{\tilde{\omega}})$.
\end{proof}

\begin{rema}\label{rema:restric}
The argument in the proof of Corollary \ref{highest-diff} shows that $\epsilon_C | \Omega_F = \epsilon_F$ for any $F \subseteq \overline{C}$ of codimension one.
\end{rema}

\subsection{\label{subsec:duality}Duality}

Throughout this section we assume that the group $\Gp$ is semisimple. Let $\m$ be a left $\Hh$-module. By Lemma \ref{Gorproj-res} the exact resolution \eqref{f:m-resolution} of $\mathfrak{m}$ consists of Gorenstein projective $\Hh$-modules. We abbreviate it from now on by \begin{equation}\label{notation:Gpr}
Gpr_\bullet(\mathfrak{m}) \longrightarrow \mathfrak{m} \ .
\end{equation}
We always impose the condition that the corresponding choice of orientations satisfies $c_F = c_{C,F}$ for any $F \in \mathscr{F}_{d-1}$.

By definition, Gorenstein projective modules are $\Hom_\Hh(\,.\,,\Hh)$-acyclic. It follows that
\begin{equation*}
    \Ext^i_\Hh(\mathfrak{m},\Hh) = h^i(\Hom_\Hh(Gpr_\bullet(\mathfrak{m}),\Hh)) \qquad \text{for any $i \geq 0$.}
\end{equation*}

From now on, we suppose that  $\mathfrak{m}$ is  a left $\Hh$-module of finite length.

\begin{lemm}\label{finitedim}
The $\Hh$-module $\mathfrak{m}$ has finite $k$-dimension.
\end{lemm}
\begin{proof}
(Note that the subsequent argument remains valid even if $\Gp$ is not semisimple.) It follows from \cite[Theorem 4]{Vig} and \cite[Cor.\ 13.1.13(ii)]{MC} that $\Hh$ is a PI affine $k$-algebra in the sense of \cite[13.10.1]{MC}. Therefore any simple $\Hh$-module has finite $k$-dimension by \cite[Theorem 13.10.3(i)]{MC}.
\end{proof}

It follows from the lemma that  each term in the resolution $Gpr_\bullet(\mathfrak{m})$ is a finitely generated $\Hh$-module. We have
\begin{align*}
     & \Hom_\Hh(Gpr_\bullet (\mathfrak{m}), \Hh) = \Hom_\Hh(\oplus_{F \in \mathscr{F}_\bullet} \Hh (\epsilon_F) \otimes_{\Hh^\dagger_F} \mathfrak{m}, \Hh) \\
    & = \oplus_{F \in \mathscr{F}_\bullet} \Hom_\Hh(\Hh  \otimes_{\Hh^\dagger_F} (\epsilon_F) \mathfrak{m}, \Hh) = \oplus_{F \in \mathscr{F}_\bullet} \Hom_{\Hh^\dagger_F}( (\epsilon_F) \mathfrak{m}, \Hh) \\
    & = \oplus_{F \in \mathscr{F}_\bullet} \Hom_{\Hh^\dagger_F}( (\epsilon_F) \mathfrak{m}, \Hh^\dagger_F) \otimes_{\Hh^\dagger_F} \Hh
\end{align*}
where the second, third, and last identity uses Remark \ref{rema:sym}, Frobenius reciprocity, and the facts that $\mathfrak{m}$ is finite dimensional and $\Hh$ is free over $\Hh^\dagger_F$ (cf.\ Prop.\ \ref{prop:free}.i), respectively. Hence the $\Ext^i_\Hh(\mathfrak{m},\Hh)$ are the cohomology groups of the complex
\begin{equation}\label{f:dual-complex}
    \oplus_{F \in \mathscr{F}_0} \Hom_{\Hh^\dagger_F}( (\epsilon_F) \mathfrak{m}, \Hh^\dagger_F) \otimes_{\Hh^\dagger_F} \Hh \xrightarrow{\;\partial^\ast\;} \ldots \xrightarrow{\;\partial^\ast\;}  \Hom_{\Hh^\dagger_C}( (\epsilon_C) \mathfrak{m}, \Hh^\dagger_C) \otimes_{\Hh^\dagger_C} \Hh
\end{equation}
with $\partial^\ast := \Hom_\Hh(\partial,\Hh)$.

In the following we will construct an augmentation map at the right end of this complex. But first we have to introduce a certain class of automorphisms of the algebra $\Hh$. Let $\xi : \Gp \longrightarrow k^\times$ be any character which is trivial on $\I$. Analogously as in section {\ref{subsubsec:epsilonF}}, multiplying an element of $\Hh$, viewed as an $\I$-bi-invariant function on $\Gp$, by the function $\xi$ defines an automorphism $j_\xi$ of the algebra $\Hh$. It satisfies
\begin{equation*}
    j_\xi(\tau_w) = \xi(\hat{w}) \tau_w \qquad\text{for any $w \in \tilde{\W}$}.
\end{equation*}
As recalled in section \ref{subsec:bruhat} we have in $\Gp$ the normal subgroup $\Gp_{aff}$ such that $\Gp /\Gp_{aff} = \Omega$. Hence any character of $\Omega$ can be viewed as a character of $\Gp$ trivial on $\I$. In the following we apply this to the orientation character $\epsilon_C$ of the chamber $C$ (note that $\Omega = \Omega_C$) and obtain the involution $j_C := j_{\epsilon_C}$ of the algebra $\Hh$, which is the identity on the subalgebra $\Hh_{aff}$ generated by all $\tau_w$ with $w$ in the preimage of $\W_{aff}$ in $\tilde{\W}$ and satisfies $j_C(\tau_{\tilde{\omega}}) = \epsilon_{C}(\hat{\tilde{\omega}}) \tau_{\tilde{\omega}}$ for any $\omega \in \Omega$. Of course, it extends to $\Hh$ the involution $j_C$ on $\Hh_C^\dagger$ which we had introduced in section {\ref{subsubsec:epsilonF}}.

In particular, we see that $(\epsilon_C)\mathfrak{m}$ makes sense as a (left) $\Hh$-module (where the action is the action on $\mathfrak{m}$ composed with $j_C$). But, in fact, we introduce the automorphism
\begin{equation*}
    \iota_C := \upiota \circ j_C
\end{equation*}
of $\Hh$, where $\upiota$ is the canonical involution recalled in section \ref{sec:Frobenius}. One easily checks from the definitions that the involutions $\upiota$ and $j_C$ commute. Hence $\iota_C$ is an involution as well. We let $\iota_C^\ast \mathfrak{m}$ denote $\mathfrak{m}$ with the new $\Hh$-action through the automorphism $\iota_C$, and we form the right $\Hh$-module
\begin{equation*}
    \mathfrak{m}^d := \Hom_k(\iota_C^\ast\mathfrak{m},k) \ .
\end{equation*}
We obtain the exact contravariant functor $\mathfrak{m} \longmapsto \mathfrak{m}^d$ from (left) finite length $\Hh$-modules to (right) finite length $\Hh$-modules. Of course, there is a corresponding functor from right to left finite length $\Hh$-modules. Clearly, we have
\begin{equation*}
    (\mathfrak{m}^d)^d = \mathfrak{m} \ .
\end{equation*}
This implies that the exact functor $(\, .\,)^d$ maps simple modules to simple modules.

We observe that $\Hh_C^\dagger = k[\Pp_C^\dagger /\I] = k[\tilde\Omega]$ is the group algebra of a finite group and hence is a symmetric Frobenius algebra. It is well known that, using the $k$-linear form
\begin{align*}
    \delta_1 : \qquad k[\tilde\Omega] & \longrightarrow k \\
    \sum_{\omega \in \tilde\Omega} c_\omega \omega & \longmapsto c_1 \ ,
\end{align*}
one obtains the isomorphism of right $\Hh_C^\dagger$-modules
\begin{align*}
    \Delta : \Hom_{\Hh^\dagger_C}( (\epsilon_C) \mathfrak{m}, \Hh^\dagger_C) & \longrightarrow \Hom_k((\epsilon_C)\mathfrak{m},k) \\
    f & \longmapsto \delta_1 \circ f \ .
\end{align*}
Since $\upiota|\Hh_C^\dagger = \id$, hence $\iota_C |\Hh_C^\dagger = j_C | \Hh_C^\dagger$, we may view $\Delta$ as an isomorphism onto the right $\Hh_C^\dagger$-module $\mathfrak{m}^d$. As such it
extends to the surjective homomorphism of right $\Hh$-modules
\begin{align*}
    augm : \Hom_{\Hh^\dagger_C}( (\epsilon_C) \mathfrak{m}, \Hh^\dagger_C) \otimes_{\Hh^\dagger_C} \Hh & \longrightarrow \Hom_k(\iota_C^\ast\mathfrak{m},k) = \mathfrak{m}^d \\
    f \otimes \tau & \longmapsto (\delta_1 \circ f)\tau = \delta_1(f(\iota_C(\tau) .)) \ .
\end{align*}
This is the envisaged augmentation map.

Next we compute the image of the last differential in the complex \eqref{f:dual-complex}. If $F \subseteq \overline{C}$ has codimension one then $S_F = \{s_F\}$ (\cite[1.2.10 and 2.1.1]{BT1}).

\begin{lemm}\label{dual-image}
We have
\begin{equation*}
    ((\Delta \otimes \id_\Hh) \circ \partial^\ast) \big( \oplus_{F \in \mathscr{F}_{d-1}} \Hom_{\Hh^\dagger_F}( (\epsilon_F) \mathfrak{m}, \Hh^\dagger_F) \otimes_{\Hh^\dagger_F} \Hh \big) = \sum_{s \in S_{aff}} \mathfrak{m}^d (1 \otimes \tau_{\tilde{s}} -  \tau_{\tilde{s}} \otimes 1)\Hh \ .
\end{equation*}
\end{lemm}
\begin{proof}
Let $F \in \mathscr{F}_{d-1}$. Obviously, the image  $\partial^\ast ( \Hom_{\Hh^\dagger_F}( (\epsilon_F) \mathfrak{m}, \Hh^\dagger_F) \otimes_{\Hh^\dagger_F} \Hh)$ is the right $\Hh$-submodule of $\Hom_{\Hh^\dagger_C}( (\epsilon_C) \mathfrak{m}, \Hh) = \Hom_{\Hh^\dagger_C}( (\epsilon_C) \mathfrak{m}, \Hh^\dagger_C) \otimes_{\Hh^\dagger_C} \Hh$ generated by $\partial^\ast ( \Hom_{\Hh^\dagger_F}( (\epsilon_F) \mathfrak{m}, \Hh^\dagger_F) \otimes 1)$. It follows from Corollary \ref{highest-diff} that
\begin{equation*}
    \partial^\ast (f \otimes 1) = \sum_{\omega \in \Omega/\Omega_F} \epsilon_C(\hat{\tilde{\omega}}) \tau_{\tilde \omega} f(\tau_{\tilde \omega^{-1}} .)
\end{equation*}
for any $f \in  \Hom_{\Hh^\dagger_F}( (\epsilon_F) \mathfrak{m}, \Hh^\dagger_F)$. Moreover, by Lemma \ref{explicit-inverse} we have
\begin{equation*}
    f = \sum_{w \in \tilde\W_F^\dagger} f_0(j_F(\tau_{w \tilde{s_F}^{-1}}^\ast) \,.\,) \tau_w
\end{equation*}
for some $f_0 \in \Hom_k(\mathfrak{m},k)$. If we insert the second formula into the first one then we obtain
\begin{align*}
    & \partial^\ast (f \otimes 1)  = \sum_{\omega \in \Omega/\Omega_F} \epsilon_C(\hat{\tilde{\omega}}) \tau_{\tilde \omega} \sum_{w \in \tilde\W_F^\dagger} f_0(j_F(\tau_{w \tilde{s_F}^{-1}}^\ast) \tau_{\tilde \omega^{-1}} \,.\,) \tau_w  \\
    & = \sum_{\omega \in \Omega/\Omega_F} \epsilon_C(\hat{\tilde{\omega}}) \sum_{\omega' \in \tilde{\Omega}_F} f_0(j_F(\tau_{\omega'}^\ast) \tau_{\tilde \omega^{-1}} \,.\,) \tau_{\tilde\omega \omega' \tilde{s_F}}
    + \sum_{\omega \in \Omega/\Omega_F} \epsilon_C(\hat{\tilde{\omega}}) \sum_{\omega' \in \tilde{\Omega}_F}  f_0(j_F(\tau_{\omega' \tilde{s_F}^{-1}}^\ast) \tau_{\tilde \omega^{-1}} \,.\,) \tau_{\tilde\omega \omega'}  \\
    & = \sum_{\omega \in \Omega/\Omega_F} \epsilon_C(\hat{\tilde{\omega}}) \sum_{\omega' \in \tilde{\Omega}_F} f_0(\epsilon_F(\hat{\omega'})(\tau_{\tilde\omega \omega'}^\ast) \,.\,) \tau_{\tilde\omega \omega' \tilde{s_F}}
    + \sum_{\omega \in \Omega/\Omega_F} \epsilon_C(\hat{\tilde{\omega}}) \sum_{\omega' \in \tilde{\Omega}_F}  f_0(\epsilon_F(\hat{\omega'})(\tau_{\tilde\omega \omega' \tilde{s_F}^{-1}}^\ast) \,.\,) \tau_{\tilde\omega \omega'}  \\
    & = \sum_{\omega' \in \tilde\Omega} \epsilon_C(\hat{\omega'}) (f_0(\tau_{\omega'}^\ast \,.\,) \tau_{\omega'} \tau_{\tilde{s_F}} + f_0(\tau_{\omega' \tilde{s_F}^{-1}}^\ast \,.\,) \tau_{\omega'}) \\
    & = \sum_{\omega' \in \tilde\Omega} \epsilon_C(\hat{\omega'}) (f_0(\tau_{\omega'}^\ast \,.\,) \tau_{\omega'} \otimes \tau_{\tilde{s_F}} + f_0(\tau_{\omega' \tilde{s_F}^{-1}}^\ast \,.\,) \tau_{\omega'} \otimes 1) \ .
\end{align*}
using that $j_F(\tau_{\tilde{s_F}^{-1}}^\ast) = \tau_{\tilde{s_F}^{-1}}^\ast$ in the third and Remark \ref{rema:restric} in the fourth identity. We deduce that in $\m^d\otimes_{\Hh_C^\dagger}\Hh$ we have
\begin{equation*}
    (\Delta \otimes \id)(\partial^\ast (f \otimes 1)) = f_0 \otimes \tau_{\tilde{s_F}} - f_0 (\upiota(\tau_{\tilde{s_F}}) \,.\,)  \otimes 1 = f_0 \otimes \tau_{\tilde{s_F}} - f_0 \tau_{\tilde{s_F}}  \otimes 1 \ .
\end{equation*}This implies our assertion with the sum on the right hand side over all $s_F$ with $F \in \mathscr{F}_{d-1}$. But the form \eqref{f:M-resolution} of the complex $Gpr_\bullet (\mathfrak{m})$ shows that the image of the differential $\partial^\ast$ in question is independent of the particular choice of the set $\mathscr{F}_{d-1}$. Hence we may sum over all $s \in S_{aff}$.
\end{proof}

\begin{prop}\label{augment}
The sequence
\begin{equation*}
    \Hom_\Hh(Gpr_{d-1}(\mathfrak{m}),\Hh) \xrightarrow{\ \partial^\ast\ } \Hom_\Hh(Gpr_d(\mathfrak{m}),\Hh) \xrightarrow{augm} \mathfrak{m}^d \longrightarrow 0
\end{equation*}
is exact.
\end{prop}
\begin{proof}
Because of Lemma \ref{dual-image} we have to show that the kernel of the multiplication map $\mathfrak{m}^d \otimes_{\Hh_C^\dagger} \Hh \longrightarrow \mathfrak{m}^d$ is equal to $\sum_{s \in S_{aff}} \mathfrak{m}^d (1 \otimes \tau_{\tilde{s}} -  \tau_{\tilde{s}} \otimes 1)\Hh$. Obviously the latter is contained in this kernel. Since the algebra $\Hh$ is generated by $\Hh_C^\dagger$ and the $\tau_{\tilde{s}}$ we also see that any element in $\mathfrak{m}^d \otimes_{\Hh_C^\dagger} \Hh$ modulo $\sum_{s \in S_{aff}} \mathfrak{m}^d (1 \otimes \tau_{\tilde{s}} -  \tau_{\tilde{s}} \otimes 1)\Hh$ is of the form $f \otimes 1$ with $f \in \mathfrak{m}^d$. This shows the reverse inclusion.
\end{proof}

\begin{coro}\label{duality}
For any left $\Hh$-module of finite length
there is a natural isomorphism of right $\Hh$-modules $\Ext_\Hh^d(\mathfrak{m},\Hh) \cong \mathfrak{m}^d$.
\end{coro}

\begin{rema}\label{rema:iwahori-5}
By viewing $\epsilon_C$ as an $\I'$-bi-invariant character of $\Gp$ we introduce the involution $j'_C$ of $\Hh'$ defined by $j'_C(\tau'_w)=\epsilon_C(\hat w)\tau'_w$. Set $\iota_C'= \upiota'\circ j'_C$ where $\upiota'$ is the canonical involution on $\Hh'$ described in \ref{rema:involution-iwahori} and define, for any left $\Hh'$-module $\m$, the right $\Hh'$-module
$$
\m^d:=\Hom_k({\iota'_C}^{\!\!\ast}\m, k) \ .
$$
Then the calculations leading to Corollary \ref{duality} go through literally (see also Remark \ref{rema:iwahori-4}) and we obtain, for any left $\Hh'$-module $\mathfrak{m}$ of finite length, a natural isomorphism of right $\Hh'$-modules $\Ext_{\Hh'}^d(\mathfrak{m},\Hh') \cong \mathfrak{m}^d$.
\end{rema}

One instance of this duality can be seen on the trivial character $\chi_{triv}$ and the sign character $\chi_{sign}$ of $\Hh$ defined by
$$
\chi_{sign}:  \tau_{w}\mapsto (-1)^{\ell(w)} \quad\text{and}\quad \chi_{triv}: \tau_{w}\mapsto q^{\ell(w)}
$$
for $w\in \tilde\W$, and with the convention that $0^0=1$ (\cite[Corollary 1]{Vig}). The canonical involution $\upiota$ exchanges these two characters.

\begin{coro}\phantomsection\label{characters}
\begin{itemize}
\item[i.] $\Ext_\Hh^d(\chi_{triv}, \Hh) \cong \chi_{sign} \circ j_C$ and  $\Ext_\Hh^d(\chi_{sign}, \Hh) \cong \chi_{triv} \circ j_C$.
\item[ii.] The injective dimension of $\Hh$ (resp.\ $\Hh'$) as a left as well as a right $\Hh$-module (resp.\ $\Hh'$-module) is equal to $d^1 = d$.
\end{itemize}
\end{coro}
\begin{proof}
The first assertion follows from Corollary \ref{duality}. The second assertion about the self-injective dimension of $\Hh$ follows from the first in view of Theorem \ref{theo:Gorenstein}. In view of Theorem \ref{theo:Gorenstein-iwahori} and Remark \ref{rema:iwahori-5} an analogous reasoning works for $\Hh'$.
\end{proof}

From Theorem \ref{theo:Gorenstein} we know that $\Ext^i_\Hh(\mathfrak{m},\Hh) = 0$ for any $i > d$. We will show that these groups also vanish for $i < d$. In fact, we will prove that $\Hh$ is a ring of the following kind.

A ring $S$ is called Auslander-Gorenstein if
\begin{itemize}
  \item[--] $S$ is left and right noetherian,
  \item[--] $S$ has finite injective dimension as a left as well as a right $S$-module, and
  \item[--] every finitely generated left and every finitely generated right $S$-module $M$ satisfies the Auslander condition: For every $i \geq 0$ and any submodule $N$ of $\Ext^i_S(M,S)$ we have $\Ext^j_S(N,S) = 0$ for any $j < i$.
\end{itemize}
For Auslander-Gorenstein rings the \textit{grade}
\begin{equation*}
    j(M) := \min \{i : \Ext_\Hh^i(M, \Hh) \neq 0 \}
\end{equation*}
of a finitely generated $\Hh$-module $M$ is an invariant which has all the properties of a ``codimension of support'' (compare \cite{ASZ}), i.\ e., which is a ``good'' codimension function. Nonzero modules of maximal grade are called  \textit{holonomic}.

\begin{rema}\label{F-AG}
For any facet $F$ of $\mathscr{X}$ the $k$-algebra $\Hh_F^\dagger$ is Auslander-Gorenstein.
\end{rema}
\begin{proof}
$\Hh_F^\dagger$ contains, by Prop.\ \ref{Hecke-Frob}.ii, a Laurent polynomial ring $R$ such that $R \subseteq \Hh_F^\dagger$ is a free $\id_R$-Frobenius extension. By the argument in the proof of Lemma \ref{Frob-injdim} we therefore have $\Ext^*_{\Hh_F^\dagger}(M,\Hh_F^\dagger) = \Ext^*_R(M,R)$ for any $\Hh_F^\dagger$-module $M$. This reduces the Auslander property over $\Hh_F^\dagger$ to the Auslander property over $R$. But $R$ as a regular noetherian commutative ring is Auslander-Gorenstein (compare \cite[Beispiele 1 and 2 after Satz 10.4]{Sch}).
\end{proof}

\begin{theo}\label{AG}
Suppose that the group $\Gp$ is semisimple. Then the $k$-algebra $\Hh$ is Auslander-Gorenstein. Moreover, for any finitely generated $\Hh$-module $M$, the Gelfand-Kirillov dimension and the Krull dimension of $M$ coincide and are equal to $d - j(M)$.
\end{theo}
\begin{proof}
We already know from \cite[Thm.\ 4]{Vig} and the proof of Lemma \ref{finitedim} that $\Hh$ is a noetherian PI affine $k$-algebra. We deduce from Corollary \ref{duality} and Corollary \ref{characters}.ii that $\Hh$ is injectively smooth in the sense of \cite[p.~990]{SZ}. Hence \cite[Thm.~3.10 and Lemma 4.3]{SZ} (applied with the trivial grading of $\Hh$) imply our assertion.
\end{proof}

\begin{coro}\label{vanishing}
Suppose that the group $\Gp$ is semisimple. For any $\Hh$-module of finite length $\mathfrak{m}$ we have $\Ext_\Hh^i(\mathfrak{m},\Hh) = 0$ for any $i < d$. In particular, a nonzero $\Hh$-module is holonomic if and only if it is of finite length.
\end{coro}
\begin{proof}
By Corollary \ref{duality} we have $\mathfrak{m} = (\mathfrak{m}^d)^d = \Ext_\Hh^d(\mathfrak{m}^d, \Hh)$. Hence the first assertion is immediate from the Auslander condition. By \cite[Cor.\ 1.3]{ASZ} holonomic $\Hh$-modules are of finite length.
\end{proof}

We note that any finitely generated Gorenstein projective $\Hh$-module $M$ is reflexive, which mean that it satisfies $M = \Hom_\Hh(\Hom_\Hh(M,\Hh),\Hh)$, where in addition $\Hom_\Hh(M,\Hh)$ again is finitely generated Gorenstein projective (cf.\ \cite[Lemma 4.2.2(ii)]{Buc}). Proposition \ref{augment} and Corollary \ref{vanishing} therefore imply that
\begin{equation*}
    \Hom_\Hh(Gpr_{d-\bullet}(\mathfrak{m}),\Hh) \longrightarrow \mathfrak{m}^d \ ,
\end{equation*}
for any finite length $\Hh$-module $\mathfrak{m}$, is an exact resolution of $\mathfrak{m}^d$ by finitely generated Gorenstein projective modules.

\begin{rema}\label{rema:iwahori-6}
The above assertions remain valid with $\Hh'$ instead of $\Hh$ (with no condition on $k$). In particular,  under the assumption that $\Gp$ is semisimple,
 $\Hh'$ is Auslander-Gorenstein and $\Ext^i_{\Hh'}(\mathfrak{m},\Hh') = 0$ for $i<d$ and  any $\Hh'$-module $\m$ of finite length.
\end{rema}

\subsection{\label{subsec:triv-sign}Homological dimensions of the trivial and sign characters for $\Hh$.}

\begin{prop}\label{prop:char-proj}
Let $F$ be  a facet $F\subseteq \overline C$ and $\chi\in\{\chi_{triv}, \chi_{sign}\}$. Suppose that
\begin{itemize}
\item[i.]  $\Gp$ is semisimple,
\item[ii.] $\sum_{w\in\tilde\W_F^\dagger} q^{\ell(w)} \neq 0$ in $k$.
\end{itemize}
Then the  restriction of $\chi$  to $\Hh_F^\dagger$ is  projective and injective as left and as a right $\Hh_F^\dagger$-module.
\end{prop}

\begin{proof}
First of all we note that
\begin{equation*}
    \sum_{w\in\tilde\W_F^\dagger} q^{\ell(w)} = [\Tp^0 : \Tp^1] \cdot |\Omega_F| \cdot \sum_{w\in\W_F} q^{\ell(w)} \ .
\end{equation*}
As a consequence of our assumption ii.  all three factors on the right hand side are nonzero in $k$. In particular,  we have the central idempotent  $\varepsilon_{1} := \frac{1}{[\Tp^0 : \Tp^1]} \sum_{t \in \Tp^0/\Tp^1} \tau_t \in \Hh$  and there is an isomorphism of algebras $\Hh'\cong \varepsilon_1\Hh $ given by
$\tau'_w\mapsto   \varepsilon _1 \tau_{\tilde w}$ for all $w\in \W$ (see  \ref{subsec:ourtheorem}). Since $\chi(\varepsilon_{1} )=1$, the character  $\chi$  of $\Hh$ factorizes through
$\varepsilon_1\Hh$ and we denote  by $\chi'$ the character of $\Hh'$ given by the composition of $\chi$ by
$\Hh'\cong \varepsilon_1\Hh\hookrightarrow \Hh$. If $\chi=\chi_{sign}$  or  $\chi=\chi_{triv}$,   the character $\chi'$ is respectively equal to
\begin{equation}\label{chi-iwahori}
\chi'_{sign}:  \tau'_{w}\mapsto (-1)^{\ell(w)} \quad\text{and}\quad \chi'_{triv}: \tau'_{w}\mapsto q^{\ell(w)}
\end{equation} for all $w\in \W$.
We are going to verify that the  restriction of $\chi'$  to $\Hh_F'^\dagger$ is  projective and injective as a left and as a right $\Hh_F'^\dagger$-module. Since, $\Hh'^\dagger_F$ identifies with $\varepsilon_1\Hh_F^\dagger$ in the  isomorphism $ \Hh' \cong  \varepsilon_1\Hh$, it will prove the claim of the proposition.

In a first step we show that $\chi' | \Hh_F'$ is projective.
For this we let $w_F$ denote the longest element in $\W_F$ and we introduce the element
\begin{equation*}
    e'_{F,\chi} =
    \begin{cases}
    \sum_{w \in \W_F} \tau'_w & \text{if $\chi = \chi_{triv}$}, \\
    \sum_{w \in \W_F} q^{\ell(w_F) - \ell(w)} (-1)^{\ell(w)} \tau'_w & \text{if $\chi = \chi_{sign}$}.
    \end{cases}
\end{equation*}
in $\Hh_F'$. We claim that
\begin{equation}\label{f:idempotent}
    \tau'_v e'_{F,\chi} = \chi'(\tau'_v) e'_{F,\chi} = e'_{F,\chi} \tau'_v \qquad\text{for any $v \in \W_F$}
\end{equation}
holds true: it  follows from a straightforward and completely formal computation
in the Hecke algebra of $\W_F$. In particular,  we see that $e'_{F,\chi}$ lies in the center of $\Hh_F'$. We compute
\begin{equation*}
    \chi'(e'_{F,\chi}) = \sum_{w \in \W_F} q^{\ell(w)} \neq 0
\end{equation*}
which is nonzero in $k$ by our assumption ii. Hence $\varepsilon'_{F,\chi} := \chi(e'_{F,\chi})^{-1} e'_{F,\chi}$ is a central idempotent in $\Hh'_F$, and $\Hh'_F \,\varepsilon'_{F,\chi}$ is a projective $\Hh'_F$-module, one dimensional over $k$, which realizes the character $\chi'| \Hh'_F$.

Since the length $\ell$ on $\W_F$ is invariant under conjugation by $\Omega_F$ and $\W_F^\dagger=\W_F\rtimes \Omega_F$, the idempotent $\varepsilon'_{F,\chi}$ commutes with any $\tau'_{\omega}$, for $\omega \in \Omega_F$, and one easily sees (use \eqref{braid-iwahori}) that
\begin{align*}
    \varepsilon_{F,\chi}'^\dagger & := \varepsilon'_{F,\chi} \cdot \frac{1}{|\Omega_F|} \sum_{\omega \in \Omega_F} \tau'_{\omega} \\
    & = \frac{1}{\sum_{w \in {\W}_F^\dagger} q^{\ell(w)}} \cdot
    \begin{cases}
    \sum_{w \in {\W}_F^\dagger} \tau'_w & \text{if $\chi = \chi_{triv}$}\\
    \sum_{w \in {\W}_F^\dagger} q^{\ell(w_F) - \ell(w)} (-1)^{\ell(w)} \tau'_w & \text{if $\chi = \chi_{sign}$}
    \end{cases}
\end{align*}
is a central idempotent in $\Hh'^\dagger_F$ such that the projective $\Hh'^\dagger_F$-module $\Hh'^\dagger_F\, \varepsilon_{F,\chi}'^\dagger$ realizes the character $\chi'| \Hh'^\dagger_F$.

Because of our assumption i. and Proposition \ref{Hecke-Frob}.iii the $k$-algebra $\Hh'^\dagger_F$ is a Frobenius algebra. Hence any projective $\Hh'^\dagger_F$-module is also injective.
\end{proof}

\begin{prop}\label{prop:equal-to-d}
Suppose that $\Gp$ is semisimple and that $\sum_{w\in\tilde\W_F^\dagger} q^{\ell(w)} \neq 0$ in $k$ for all facets $F\in \overline C$.
Then, as left and  right $\Hh$-modules, the characters  $\chi_{triv}$ and $\chi_{sign}$  have projective  dimension equal to $d$.
\end{prop}
\begin{proof}
Let $\chi \in \{\chi_{triv} , \chi_{sign}\}$. By symmetry it suffices to treat the case of $\chi$ as a left $\Hh$-module. Proposition \ref{prop:char-proj} implies that $Gpr_\bullet (\chi) \longrightarrow \chi$ is an exact projective resolution of the $\Hh$-module $\chi$ (see \eqref{notation:Gpr}). Hence the projective dimension of $\chi$ is $\leq d$. Equality follows from Corollary \ref{characters}.i.
\end{proof}

\begin{rema}\label{rema:proj-dim-chi'}
With no hypothesis on $k$ and $\Gp$, consider the characters   $\chi'_{triv}$ and $\chi'_{sign}$ of $\Hh'$ given by the formulas  \eqref{chi-iwahori}.   Let $\chi'\in \{\chi'_{triv}, \chi'_{sign}\}$.
Under the assumption that  $\Gp$ is semisimple and   $\sum_{w\in\W_F^\dagger} q^{\ell(w)} \neq 0$ in $k$ for all facets $F\subseteq\overline C$, we establish  in the proof of Proposition \ref{prop:char-proj}  that the  restriction of $\chi'$  to $\Hh_F'^\dagger$ is  projective  as a left and as a right $\Hh_F'^\dagger$-module. Therefore,
the exact resolution  \eqref{f:m'-resolution} of $\chi'$  is  a resolution by projective modules and  $\chi'$ has projective dimension $\leq d$ as a left $\Hh'$-module. From Remark  \ref{rema:iwahori-5}, we obtain that the projective dimension  of $\chi'$ as a left $\Hh'$-module is equal to $d$. By symmetry, the same holds for $\chi'$ as a right $\Hh'$-module.
\end{rema}

We remark that it is a general fact for Gorenstein rings $A$ that any $A$-module of finite projective dimension also has finite injective dimension and that, in this case, both these dimensions are bounded above by the self-injective dimension of $A$ (cf.\ \cite{EJ} 9.1.10). In the situation of the above Proposition \ref{prop:equal-to-d} it also follows directly from Proposition \ref{injdim}, Theorem \ref{theo:freeresolution}, and Proposition \ref{prop:char-proj} that the two $\Hh$-modules $\chi_{triv}$ and $\chi_{sign}$ have injective dimension $\leq d$.

\begin{rema}\label{rema:nonsplit}
Suppose that the assumptions of Proposition \ref{prop:char-proj} hold true for all zero dimensional facets in $\overline{C}$. Then the resolution \eqref{f:H-H-bimod} is not split as a complex of $\Hh$-bimodules.
Otherwise, any left $\Hh$-module $\m$ would be a direct summand of  $\bigoplus_{F \in \mathscr{F}_0} \Hh (\epsilon_F) \otimes_{\Hh_F^\dagger} \m$ and there would be an $F \in \mathscr{F}_0$ such that $\Hom_\Hh(\m, \Hh (\epsilon_F) \otimes_{\Hh_F^\dagger} \m)\neq 0.$ Choose $\m$ to be the trivial or the sign character of $\Hh$. Then by Proposition \ref{prop:char-proj}, the restriction of $\m$ to $\Hh_F^\dagger$ is projective so that $\m$ is isomorphic to a direct summand of $\Hh_F^\dagger$
and $\Hh (\epsilon_F) \otimes_{\Hh_F^\dagger} \m$ injects in $\Hh$ as a left $\Hh$-module. This  contradicts the fact (Corollary \ref{vanishing}) that $\Hom_{\Hh}(\m,\Hh)=0$.
\end{rema}

\section{\label{subsec:projdim}Global dimension of the (pro-$p$) Iwahori-Hecke algebra in characteristic $p$}

If $k$ has characteristic zero then by a theorem of Bernstein (compare the argument in \cite[Prop.\ 37]{Vi}) the Hecke algebras $\Hh$ and $\Hh'$ are regular, i.\ e.,  they have finite global dimension, which is bounded above by the rank of the group $\Gp$.
We will show that this is no longer true in characteristic $p$.

In this paragraph, we assume that $k$ has characteristic $p$, so that one can define idempotents in $\Hh$ as mentioned in  \ref{subsec:ourtheorem}. In particular,
$\Hh'$ (resp.\ ${{\Hh}'_F}^{\!\!\!\dagger}$, $\Hh'_F$) can  be identified with the algebra $\varepsilon_{\bf 1}\Hh$ (resp.\ $\varepsilon_{\bf 1}\Hh^\dagger_F$, $\varepsilon_{\bf 1}\Hh_F$) with unit
$\varepsilon_{\bf 1}$. It implies that the projective dimension of $\Hh$ (resp.\ $\Hh^\dagger_F$, $\Hh_F$) is bounded below by the one of $\Hh'$ (resp.\  ${{\Hh}'_F}^{\!\!\!\dagger}$, $\Hh'_F$). We will use this argument several times without repeating it.

\begin{prop}\label{prop:infinite}
Suppose that $k$ has characteristic $p$.
\begin{itemize}
\item[i.] The algebra $\Hh$ has  finite global dimension if and only if $\Hh^\dagger_F$ has finite global dimension for all facets $F \subseteq \overline C$. Otherwise, $\Hh$  has a simple module of infinite projective dimension.
\item[ii.] The algebra $\Hh'$ has finite global dimension if and only if ${{\Hh}'_F}^{\!\!\!\dagger}$ has finite global dimension for all facets $F \subseteq \overline C$. Otherwise, $\Hh$ and $\Hh'$ both have a simple module of infinite projective dimension.
\end{itemize}
\end{prop}
\begin{proof}
It follows from \cite[Theorem 4]{Vig} and \cite[Cor.\ 13.1.13(iii) and Cor.\ 13.6.6(iii)]{MC} that $\Hh$ is a fully bounded noetherian ring. For such a ring it is shown in \cite{Rai} that its global dimension is the supremum of the projective dimensions of its simple modules. But according to \cite[Theorem 9.1.10]{EJ} the projective dimension of any $\Hh$-module, provided that it is finite, is bounded above by the self-injective dimension $\leq d^1$ of $\Hh$ (cf.\ Theorem \ref{theo:Gorenstein}). It follows that if $\Hh$ has infinite global dimension, then there must exist a simple $\Hh$-module of infinite projective dimension.
The statement is valid with $\Hh'$ instead of $\Hh$ since $\Hh'$ is also a fully bounded noetherian ring (\cite[Theorem 2]{Vign}) with finite self-injective dimension by Theorem \ref{theo:Gorenstein-iwahori}.

It remains to show the equivalences. For $\Hh$,
combining Proposition \ref{prop:free} and Lemma \ref{induced}.iii gives that $\Hh$ has infinite global dimension as soon as there is a facet $F \subseteq \overline C$ with $\Hh^\dagger_F$ of infinite global dimension.
Now suppose that $\Hh^\dagger_F$ has finite global  dimension for all facets $F$ with dimension $\leq d$ contained in $\overline C$. Then using the resolution \eqref{f:H-H-bimod} together with
Proposition \ref{injdim}.iii., we obtain that any $\Hh$-module has finite projective dimension. We have proved i. By Remark \ref{rema:iwahori-3} and using the resolution \eqref{f:H'-H'-bimod}, the
same argument works for $\Hh'$ and we obtain ii.
\end{proof}

\begin{coro} \label{coro:infinite}
Suppose that $k$ has characteristic $p$.
\begin{itemize}
\item[i.] If there is a facet $F \subseteq \overline C$ such that $\Hh_F$ is not semisimple, then each of $\Hh_F$, $\Hh^\dagger_F$, and $\Hh$ has a simple module of infinite projective dimension.
\item[ii.] If there is a facet $F \subseteq \overline C$ such that $\Hh'_F$ is not semisimple, then each of $\Hh'_F$, $\Hh_F$, ${{\Hh}'_F}^{\!\!\!\dagger}$, $\Hh^\dagger_F$, $\Hh'$, and $\Hh$ has a simple module of infinite projective dimension.
\end{itemize}
\end{coro}

\begin{proof}
First of all we recall (cf.\ \cite[Thm. 1]{Aus}) that an algebra  is semisimple if and only if it has global dimension equal to zero.
Then (cf.\ \cite[Cor.\ 11, Prop.\ 14, Prop.\ 15]{Aus}): a Frobenius $k$-algebra  either is semisimple or has infinite global dimension.  In the latter case there exists a simple module of infinite projective dimension. In section \ref{sec:Frobenius} we recalled that $\Hh_F$ and $\Hh'_F$  are always  Frobenius algebras.

Suppose that $F\subseteq \overline C$ is such that $\Hh_F$ is not semisimple. Then combining Lemmas \ref{HF+HF}.ii and  \ref{induced}.iii, we see that  $\Hh_F^\dagger$ has a module of infinite projective dimension, and by the proof of Proposition \ref{F-injdim}, which shows that $\Hh^\dagger_F$ is fully bounded noetherian, it has a simple module of infinite projective dimension. We conclude the proof of i. using Proposition \ref{prop:infinite}.

The assertion ii. is proved by the same arguments (see Remark \ref{rema:iwahori-3}) and using the observation at the beginning of this paragraph.
\end{proof}

Since $k$ has characteristic $p$, the algebras $\Hh_C = k[\Tp^0 / \Tp^1]$ and $\Hh'_C=k$ are   semisimple.   As for the algebras $\Hh^\dagger_C = k[\Pp_C^\dagger / \I]$  and $\Hh'^\dagger_C = k[\Pp_C^\dagger / \I']$,   they  are semisimple if and only if $p$ does not divide the order of $\Pp_C^\dagger / \I' = \Omega$. On the other hand, as noticed in  \cite[Theorem 5]{Cab},  the Jacobson radical of $\Hh_F \otimes_k \overline{k}$ (resp.\ $\Hh_F' \otimes_k \overline{k}$), where $\overline{k}$ is an algebraic closure of $k$, contains all the commutators in $\Hh_F \otimes_k \overline{k}$ (resp.\  $\Hh_F' \otimes_k \overline{k}$) because the simple $\Hh_F \otimes_k \overline{k}$-modules (resp.\  $\Hh_F' \otimes_k \overline{k}$-modules) are one dimensional. Since $\Hh_F / \mathrm{Jac}(\Hh_F)$ maps injectively to $\Hh_F \otimes_k \overline{k} / \mathrm{Jac}(\Hh_F \otimes_k \overline{k})$ and correspondingly for $\Hh'_F$ (cf.\ \cite[Thm.\ 5.14]{Lam0}) this statement remains true for $\Hh_F$ and $\Hh'_F$. Therefore, $\Hh_F$ (resp.\ $\Hh_F'$) is not semisimple and has a simple module of infinite projective dimension whenever it is not commutative. This observation gives part of the following criterion.

\begin{lemm}\label{lemma:HF-infinitedim}
Suppose $k$ has characteristic $p$. Assume $d > 0$, and let $F$ be a facet $F \subsetneq \overline C$ of dimension $< d$. We have:
\begin{itemize}
  \item[a.]  Suppose that the root system $\Phi_F$ is not of type $A_1 \times\ldots\times A_1$. Then  $\Hh'_F$ and $\Hh_F$ are not semisimple.
  \item[b.]  Suppose that $q \neq 2$.
  Then $\Hh_F$ is not semisimple.
\end{itemize}
\end{lemm}
\begin{proof} a. It follows from the classification of root systems that $S_F$  contains two elements $s$ and $s'$  that do not commute (\cite[Ch VI, 1.3]{Bki-LA}), so that we have $\tau'_{s}\tau'_{s'} \neq  \tau'_{s'} \tau'_{s}$ in $\Hh'_F$. It implies that $\Hh'_F$ and hence $\Hh_F$ are not commutative. b. Let $s\in S_F$.
First suppose that $q \neq 2, 3$. Then there is an element $t \in \mathbb F_q^* \cong \Tp_s(\mathbb F_q)$ whose square is not trivial.
($\Tp_s(\mathbb F_q)$  was defined in Section \ref{sec:Frobenius} before Remark  \ref{rema:invertible} and identifies with a subgroup of $\Tp^0/\Tp^1$.) By definition of $n_s$, we have  $n_s t n_s^{-1} t^{-1} = t^{-2} \neq 1$ from which we deduce that $n_s t \neq  t  n_s  \bmod \Tp^1$. In $\Hh_F$, we therefore have $\tau_{\tilde s} \tau_t \neq \tau_t \tau_{\tilde s}$. It remains to consider the case $q = 3$. Then $\frac{1}{2} \theta_s$ is a nonzero idempotent in $k[\Tp^0/\Tp^1]$. Since $\tau_{\tilde s}$ and $\theta_s$ commute we obtain
\begin{equation*}
    \big[(1 - \frac{1}{2} \theta_s) \tau_{\tilde s} \big]^2 = (1 - \frac{1}{2} \theta_s) \tau_{\tilde s}^2 = (1 - \frac{1}{2} \theta_s) \tau_{\tilde s} \theta_s = 0 \ .
\end{equation*}
It follows that $(1 - \frac{1}{2} \theta_s) \tau_{\tilde s}$ is a nonzero nilpotent element in $\Hh_F$. Since $\Hh_F$ modulo its Jacobson radical is a product of fields we see that this Jacobson radical has to be nonzero.
\end{proof}

\begin{rema}
If the root system $\Phi_F$ is  of type $A_1 \times\ldots\times A_1$, then $\Hh_F'$ is  semisimple.
\end{rema}

\begin{examp}\label{rema:p=2}
Suppose that $k$ has characteristic $p$ and that $q=2$. Then $\Hh = \Hh'$, and we have:
\begin{itemize}
\item[1.] Let $\Gp = {\rm SL}_2(\Corps)$. Then the algebras $\Hh_C^\dagger = \Hh_C = k$ and $\Hh_x^\dagger = \Hh_x \cong k \times k$, for any vertex $x \in \overline C$, are semisimple. It therefore follows from the resolution \eqref{f:H-H-bimod}, Proposition \ref{injdim}.iii and Corollary \ref{characters}.i that $\Hh$ has global dimension $1$.  In this case the resolution \eqref{f:M-resolution} is a projective resolution of any $\Hh$-module $\mathfrak{m}$.
\item[2.] Let $\Gp = {\rm PGL}_2(\Corps)$. Then the algebra $\Hh_C^\dagger = k[\Omega] \cong k[\mathbb{Z}/2\mathbb{Z}]$ is not semisimple (but Frobenius). By Proposition \ref{prop:infinite},  there exists a simple $\Hh$-module of infinite projective dimension.
\item[3.] Let $\Gp = {\rm GL}_2(\Corps)$. Then the algebra $\Hh_C^\dagger \cong k[\Z]$ has global dimension $1$. Moreover, $\Hh_{x_0} \cong k \times k$ so that $\Hh_{x_0}^\dagger = \Hh_{x_0}[\Z]$  has global dimension $1$ as well. It follows from the resolution  \eqref{f:H-H-bimod} and Proposition \ref{injdim}.iii that $\Hh$ has global dimension $\leq 2$.
\end{itemize}
\end{examp}

\begin{examp}
Suppose that $k$ has characteristic $p \neq 2$, and let $\Gp = {\rm PGL}_2(\Corps)$. Then ${{\Hh}'_C}^{\!\!\!\dagger} = k[\Omega] \cong k[\mathbb{Z}/2\mathbb{Z}]$ and ${{\Hh}'_{x_0}}^{\!\!\!\dagger} = \Hh'_{x_0} \cong k \times k$ are semisimple. It follows that $\Hh'$ has global dimension $1$ (whereas $\Hh$ has infinite global dimension).
\end{examp}

\begin{coro}\label{infiniteGL2}
Let $\Gp = {\rm PGL}_2(\mathbb{Q}_p)$, and let $k$ be algebraically closed of characteristic $p$. The category of smooth $k$-representations of $\Gp$  generated by their $\I$-invariant vectors has infinite global  dimension.
\end{coro}
\begin{proof}
 The category of $\Hh$-modules is  equivalent to the category of smooth representations of $\Gp= {\rm PGL}_2(\mathbb{Q}_p)$  generated by their $\I$-invariant vectors   (\cite{Inv}). Conclude applying  Lemma \ref{lemma:HF-infinitedim}b., Remark \ref{rema:p=2}.2 and  Corollary  \ref{coro:infinite}.
\end{proof}

\section{Graded Hecke algebras\label{sec:graded}}

The elements $\tau_w$, for $w \in \tilde\W$, form a $k$-basis of $\Hh$. It is immediate from the defining relations \eqref{braid} and \eqref{quadratic} that
\begin{equation*}
    F_n \Hh := \sum_{\ell(w) \leq n} k \tau_w \qquad\text{for $n \geq 0$}
\end{equation*}
defines an increasing, discrete (\i.\ e., nonnegative), and exhaustive ring filtration of $\Hh$. We have $F_0 \Hh = k[\tilde\Omega]$. In this final section we investigate the associated graded algebra $gr_\bullet \Hh$.

For any $w \in \tilde\W$ we let $\bar{\tau}_w$ denote the principal symbol of $\tau_w$ in the graded ring $gr_\bullet \Hh$. It is clear that, for any $n \geq 0$, the set $\{\bar{\tau}_w : \ell(w) = n \}$ is a $k$-basis of $gr_n \Hh$. The multiplication in $gr_\bullet \Hh$ is determined by the rule
\begin{equation}\label{f:nil}
    \bar{\tau}_v \bar{\tau}_w =
    \begin{cases}
    \bar{\tau}_{vw} & \text{if $\ell(vw) = \ell(v) + \ell(w)$}, \\
    0 & \text{otherwise}.
    \end{cases}
\end{equation}

For any facet $F$ of $\mathscr{X}$ we equip the subalgebra $\Hh_F^\dagger$ with the induced filtration.

\begin{lemm}\phantomsection\label{filt-free}
\begin{itemize}
  \item[i.] $\Hh$ is filt-free as a left as well as a right $\Hh_F^\dagger$-module.
  \item[ii.] $gr_\bullet \Hh$ is gr-free as a left as well as a right $gr_\bullet \Hh_F^\dagger$-module.
\end{itemize}
\end{lemm}
\begin{proof}
(Compare \cite[Def. I.6.1 and p.\ 28]{LO} for the definitions of filt-free and gr-free modules.) i. As a consequence of Lemma \ref{DF+}.i we have
\begin{equation*}
    F_n \Hh = \oplus_{d \in \Dd_F^\dagger} \bar{\tau}_{\tilde d} F_{n-\ell(d)} \Hh_F^\dagger = \oplus_{d \in \Dd_F^\dagger} F_{n-\ell(d)} \Hh_F^\dagger \bar{\tau}_{{\tilde d}^{-1}} \qquad\text{and}\qquad \bar{\tau}_{\tilde d}, \bar{\tau}_{{\tilde d}^{-1}} \in F_{\ell(d)} \Hh \setminus F_{\ell(d)-1} \Hh \ .
\end{equation*}
ii. This follows from i.
\end{proof}

The involution $j_F$ of $\Hh_F^\dagger$ obviously respects the filtration. Hence the tensor product filtration on $\Hh(\epsilon_F) \otimes_{\Hh_F^\dagger} \Hh$ is well defined and makes it a filtered $(\Hh,\Hh)$-bimodule. It is clear from Proposition \ref{differential} that \eqref{f:H-H-bimod} is a complex of filtered bimodules. Alternatively the subsequent proof will explain a more conceptual reason for this fact.

\begin{theo}\label{strict-exact}
For any $n \geq 0$ the subcomplex
\begin{equation*}
    0 \longrightarrow \bigoplus_{F \in \mathscr{F}_d} F_n (\Hh (\epsilon_F) \otimes_{\Hh_F^\dagger} \Hh) \longrightarrow \ldots \longrightarrow \bigoplus_{F \in \mathscr{F}_0} F_n (\Hh (\epsilon_F) \otimes_{\Hh_F^\dagger} \Hh) \longrightarrow F_n \Hh \longrightarrow 0
\end{equation*}
of \eqref{f:H-H-bimod} is exact. In other words \eqref{f:H-H-bimod} is strict-exact.
\end{theo}

\begin{proof}
We will, in fact, prove the equivalent assertion for the isomorphic oriented chain complex
\begin{equation}\label{f:apart-complex}
    0 \longrightarrow C_c^{or} (\Aa_{(d)}, \cX^\I) \xrightarrow{\;\partial\;} \ldots \xrightarrow{\;\partial\;} C_c^{or} (\Aa_{(0)}, \cX^\I) \xrightarrow{\;\epsilon_\Aa\;} \Hh \longrightarrow 0 \ .
\end{equation}
We recall that the coefficient system of right $\Hh$-modules $\cX^\I$ on the apartment $\Aa$  is given by
\begin{equation*}
    F\longmapsto \cX^\I(F) := \mathbf{X}^{\I_{C(F)}} \qquad\textrm{for any facet $F$ in $\Aa$}
\end{equation*}
with transition maps
\begin{align*}
t_{F'}^F    :\:\:\mathbf{X}^{\I_{C(F)}} & \longrightarrow \mathbf{X}^{\I_{C(F')}}  \qquad\qquad\qquad\text{whenever $F' \subseteq \overline{F}$} \ . \\
    x & \longmapsto \sum_{g \in (\I \cap \Pp_{F'}^\dagger)/(\I \cap \Pp_F^\dagger)} gx\, .
\end{align*}
If $F \subseteq \overline{C}$ then $C(F) = C$ so that $\cX^\I(F) = \Hh$ in this case. Since $\W_{aff}$ acts transitively on the chambers in $\Aa$ we find, by Proposition \ref{representants}.i, for any facet $F'$ in $\Aa$ a facet $F \subseteq \overline{C}$ and an element $d \in \Dd_F$ such that $F' = dF$. Moreover, Proposition \ref{prop:closest}.i implies $C(F') = dC$ and hence
\begin{equation*}
    \cX^\I(F') = \mathbf{X}^{\I_{dC}} = \mathbf{X}^{\hat{d}\I \hat{d}^{-1}} = \hat{d} (\mathbf{X}^\I) = \hat{d} \Hh
\end{equation*}
(inside $\mathbf{X}$). Similarly as before, we fix here and in the further proof, for any $v \in \W$, a lifting $\tilde{v} \in \tilde\W$ of $v$ as well as a lifting $\hat{v} \in N_\Gp(\Tp)$ of $\tilde{v}$. The above right $\Hh$-module becomes a filtered $\Hh$-module through
\begin{equation*}
    F_n (\cX^\I (F')) := \hat{d} F_n \Hh \qquad\text{for $n \geq 0$}.
\end{equation*}
To see that this is well defined let $F' = d_1 F_1 = d_2 F_2$ with $F_i \subseteq \overline{C}$ and $d_i \in \Dd_{F_i}$. Write $d_2^{-1} d_1 = w_0 \omega$ with $w_0 \in \W_{aff}$ and $\omega \in \Omega$. Then $w_0 \omega F_1 = F_2$. Since $\omega F_1, F_2 \subseteq \overline{C}$ it follows from (\cite[V.3.2 property (I)]{Bki-LA} that $\omega F_1 = F_2$ (cf.\ Remark \ref{rema:transitivity}). We obtain $d_1 F_1 = d_2 \omega F_1$. Using that $d_2 \in \Dd_{F_2}$ one checks that $d_2 \omega$ has minimal length in the coset $d_2 \omega \W_{F_1} = d_2 \W_{F_2} \omega$. Hence $d_2 \omega \in \Dd_{F_1}$, and therefore $d_2 \omega = d_1$. We conclude that $\hat{d_1} F_n \Hh = \hat{d_2} \hat{\omega} F_n \Hh = \hat{d_2} \tau_{\tilde\omega} F_n \Hh = \hat{d_2} F_n \Hh$. The transition maps are maps of filtered right $\Hh$-modules of a certain degree. More precisely, we claim that
\begin{equation}\label{f:trans-filt}
    t_{F_0'}^{F'} (F_n (\cX^\I (F'))) \subseteq F_{n+\d(C,C(F')) - \d(C,C(F_0'))} (\cX^\I (F_0')) \qquad\text{whenever $F_0' \subseteq \overline{F'}$}
\end{equation}
holds true. Note that $\d(C,C(F_0')) \leq \d(C,C(F'))$. As above, we let $F \subseteq \overline{C}$ and $d \in \Dd_F$ be such that $F' = dF$. We put $F_0 := d^{-1}F_0' \subseteq \overline{F}$, and we choose $d_0 \in \Dd_{F_0}$ such that $F_0' = d_0 F_0$. By Lemma \ref{bruhat-finite} we have $d = d_0 v_0$ for some $v_0 \in \W_{F_0}^\dagger$, and \eqref{additive0} implies $\ell(d) = \ell(d_0) + \ell(v_0)$. From Proposition \ref{prop:closest}.i and \eqref{eq:distance} we deduce $\d(C,C(F')) = \d(C,C(dF)) = \d(C,dC) = \ell(d)$ and similarly $\d(C,C(F_0')) = \ell(d_0)$. We therefore have to show that
\begin{equation*}
    t_{F_0'}^{F'} (\hat{d} F_n \Hh) \subseteq \hat{d_0} F_{n + \ell(v_0)} \Hh \ .
\end{equation*}
Since the transition maps are maps of right $\Hh$-modules this reduces to the claim that
\begin{equation*}
    \sum_{g \in (\I \cap \Pp_{F_0'}^\dagger)/(\I \cap \Pp_{F'}^\dagger)} g \hat{d}\tau_1 \in \hat{d_0} F_{\ell(v_0)} \Hh
\end{equation*}
which follows from the more precise statement that
\begin{equation*}
    \hat{d_0}^{-1} (\I \cap \Pp_{F_0'}^\dagger) \hat{d} \Tp^0 \I = \hat{d_0}^{-1} (\I \cap \Pp_{F_0'}^\dagger) \hat{d_0} \hat{v_0} \Tp^0 \I  \subseteq \I \hat{v_0} \Tp^0 \I \ .
\end{equation*}
But we conclude from Proposition \ref{prop:closest} that $\I \cap \Pp_{F_0'}^\dagger = \I \cap \Pp_{d_0 F_0}^\dagger = \I_{C(d_0 F_0)} = \I_{d_0 C}$ and hence that $\hat{d_0}^{-1} (\I \cap \Pp_{F_0'}^\dagger) \hat{d_0} \subseteq \I$.

We point out that in case $C(F') = C(F'_0)$ we have $\ell(v_0) = 0$ and the transition map $F_n (\cX^\I (F')) \longrightarrow F_n (\cX^\I (F'_0))$ is the identity.

In $\mathscr{X}$, resp.\ $\Aa$, we have, for any $n \geq 0$, the subcomplex $\mathscr{X}(n)$, resp.\ $\Aa(n)$, of all facets $F$ such that $C(F)$ is of distance $\leq n$ from $C$. By $\Aa(n)_i \subseteq \mathscr{X}(n)_i$ we denote the sets of $i$-dimensional facet in these subcomplexes.

As a consequence of \eqref{f:trans-filt} we may introduce the increasing sequence $F_0 \cX^\I \subseteq \ldots \subseteq F_n \cX^\I \subseteq \cX^\I$ of coefficient subsystems defined by
\begin{equation*}
    (F_n \cX^\I)(F') :=  F_{n-\d(C,C(F'))} (\cX^\I(F'))
\end{equation*}
for all facets $F'$ in $\Aa$ (with the convention that $F_m(.) = 0$ if $m < 0$). Obviously $F_n \cX^\I$ is supported on $\Aa(n)$. It is easy to see that each right $\Hh$-module $C_c^{or} (\Aa_{(i)}, \cX^\I)$ is a filtered module with respect to the induced increasing discrete filtration
\begin{equation*}
    F_n C_c^{or} (\Aa_{(i)}, \cX^\I) := C_c^{or} (\Aa_{(i)}, F_n \cX^\I) = C_c^{or} (\Aa(n)_{(i)}, F_n \cX^\I) \ ,
\end{equation*}
which, moreover, is exhaustive. The complex \eqref{f:apart-complex} in this way becomes a filtered complex.

We claim that the isomorphism between \eqref{f:apart-complex} and \eqref{f:H-H-bimod} established in Proposition \ref{prop:isocomplex} and Theorem \ref{theo:freeresolution} is a filtered isomorphism. As a piece of notation we introduce the characteristic functions $\mathrm{char}_U$ of any compact open subset $U \subseteq \Gp$. Let $F \subseteq \overline{C}$ be an $i$-dimensional facet. Under the identification $\Hh = \ind_{\Pp_F^\dagger}^\Gp (\mathbf{X}_F^\dagger)^\I$, coming from the transitivity of induction, the element $\tau_v \in \Hh$, for $v \in \tilde\W$, corresponds to the function $\phi_v : \Gp \longrightarrow \mathbf{X}_F^\dagger$ which is left $\I$-invariant, is supported on $\I \hat{v} \Pp_F^\dagger$, and has value $\phi_v(\hat{v}) = \mathrm{char} _{\hat{v}^{-1} \I \hat{v} \I \cap \Pp_F^\dagger}$. Hence under the embedding $\Hh(\epsilon_F) \otimes_{\Hh_F^\dagger} \Hh \hookrightarrow C_c^{or} (\mathscr{X}_{(i)}, \cX)^\I$, which depends on the choice of an orientation $(F,c_F)$ of $F$, the element $\tau_v \otimes 1$ is mapped to the $\I$-invariant chain $f_v \in C_c^{or} (\mathscr{X}_{(i)}, \cX)$ defined by
\begin{align*}
    f'_v & := \sum_{g \in \Gp / \Pp_F^\dagger} g \big(\text{chain supported on $\{F\}$ with value $\epsilon_F \mathrm{char}_{g^{-1} \I \hat{v} \I \cap \Pp_F^\dagger}$ in $(F,c_F)$} \big) \\
    & = \sum_{g \in \I \hat{v} \Pp_F^\dagger / \Pp_F^\dagger} \text{chain supported on $\{gF\}$ with value $\epsilon_F(g^{-1}.) \mathrm{char}_{\I \hat{v} \I \cap g \mathscr{P}_F^\dagger}$ in $(gF,g c_F)$} \ .
\end{align*}
It follows from Lemma \ref{I-orient} that $\epsilon_F | \hat{v}^{-1} \I \hat{v} \I \cap \Pp_F^\dagger = 1$. Therefore we may characterize $f'_v$ as being the unique oriented chain such that
\begin{itemize}
  \item[--] $f'_v$ is $\I$-invariant,
  \item[--] $f'_v$ is supported on the $\I$-orbit of the facet $vF$, and
  \item[--] $f'_v((vF, vc_F)) = \mathrm{char}_{\I \hat{v} \I \cap \hat{v} \Pp_F^\dagger}$.
\end{itemize}
Since the above embedding is right $\Hh$-equivariant we obtain more generally that, for any two $v,w \in \tilde\W$, the element $\tau_v \otimes \tau_w \in \Hh(\epsilon_F) \otimes_{\Hh_F^\dagger} \Hh$ is mapped to the unique oriented chain $f'_{v,w} \in C_c^{or} (\mathscr{X}_{(i)}, \cX)$ such that
\begin{itemize}
  \item[--] $f'_{v,w}$ is $\I$-invariant,
  \item[--] $f'_{v,w}$ is supported on the $\I$-orbit of the facet $vF$, and
  \item[--] $f'_{v,w}((vF, vc_F)) = \tau_w( \mathrm{char}_{\I \hat{v} \I \cap \hat{v} \Pp_F^\dagger})$.
\end{itemize}
According to Proposition \ref{prop:free}.i the elements $\tau_{\tilde{d}}$, for $d \in \Dd_F^\dagger$, form a basis of $\Hh(\epsilon_F)$ as a right $\Hh_F^\dagger$-module. Hence $\{ \tau_{\tilde{d}} \otimes \tau_w \}_{(d,w) \in \Dd_F^\dagger \times \tilde\W}$ is a $k$-basis of $\Hh(\epsilon_F) \otimes_{\Hh_F^\dagger} \Hh$. By Proposition \ref{prop:closest} and \eqref{eq:distance} we have
\begin{equation*}
    \d(C,C(dF)) = \d(C,dC) = \ell(d) \qquad\text{and}\qquad \I \cap \Pp_{dF}^\dagger \subseteq \I_{C(dF)} = \I_{dC}
\end{equation*}
for any $d \in \Dd_F^\dagger$. The former says that $dF \in \mathscr{X}(\ell(d))_i \setminus \mathscr{X}(\ell(d)-1)_i$ and the latter that $\hat{d}^{-1} \I \hat{d} \cap \Pp_F^\dagger \subseteq \I$ and hence that $\I \hat{d} \I \cap \hat{d} \Pp_F^\dagger = \hat{d} \I$. It follows that
\begin{equation*}
    f'_{\tilde{d},w}((dF,dc_F)) = \tau_w(\mathrm{char}_{\hat{d} \I}) = \mathrm{char}_{\hat{d} \I w \I} = \hat{d} \tau_w \ .
\end{equation*}
Since the isomorphism $C_c^{or} (\mathscr{X}_{(i)}, \cX)^\I \cong C_c^{or} (\Aa_{(i)}, \cX^\I)$ is given by restricting chains we obtain that the composed embedding $\Hh(\epsilon_F) \otimes_{\Hh_F^\dagger} \Hh \hookrightarrow C_c^{or} (\Aa_{(i)}, \cX^\I)$ sends $\tau_{\tilde{d}} \otimes \tau_w$ to the unique oriented chain $f_{\tilde{d},w} \in C_c^{or} (\Aa_{(i)}, \cX^\I)$ supported on $\{dF\} \subseteq \Aa(\ell(d))_i \setminus \Aa(\ell(d)-1)_i$ with value $f_{\tilde{d},w}((dF,dc_F)) = \hat{d} \tau_w \in (F_{\ell(w)} \cX^\I)(dF)$. In particular, we have $f_{\tilde{d},w} \in F_{\ell(d) + \ell(w)} C_c^{or} (\Aa_{(i)}, \cX^\I)$. Vice versa, let $f \in F_n C_c^{or} (\Aa_{(i)}, \cX^\I)$ be an arbitrary chain. We may assume that $f$ is supported on a single facet $\{F'\}$. We choose a facet $F \subseteq \overline{C}$ and an element $d \in \Dd_F$ such that $F' = dF$. Then $\d(C,C(dF)) = \ell(d)$ and $f((F',c')) \in F_{n-\ell(d)} \cX^\I (F') = \hat{d} F_{n-\ell(d)} \Hh$. By a further decomposition we therefore may assume that $f((F',c')) = \hat{d} \tau_w$ for some $w \in \tilde\W$ with $\ell(w) \leq n - \ell(d)$. We see that $f$ up to sign  is the image of $\tau_{\tilde{d}} \otimes \tau_w \in F_n (\Hh (\epsilon_F) \otimes_{\Hh_F^\dagger} \Hh)$. This establishes our claim.

To show our assertion we therefore may equivalently prove that the complexes
\begin{equation*}
    0 \longrightarrow C_c^{or} (\Aa(n)_{(d)}, F_n \cX^\I) \longrightarrow \ldots \longrightarrow C_c^{or} (\Aa(n)_{(0)}, F_n \cX^\I) \longrightarrow F_n \Hh \longrightarrow 0
\end{equation*}
are exact for all $n\geq 0$. But this follows by exactly the same reasoning as in the proof of Theorem \ref{I-resolutionV}. The coefficient systems $F_n \cX^\I$ still have the property that, for any chamber $D$ in $\Aa$ of distance $m$ from $C$, they have constant values with identity transition maps on all facet in $\overline{D} \setminus \Aa(m-1)$.
\end{proof}

\begin{coro}\label{gr-exact}
The complex
\begin{equation}\label{f:grH-grH-bimod}
    0 \rightarrow \bigoplus_{F \in \mathscr{F}_d} gr_\bullet (\Hh (gr_\bullet(\epsilon_F)) \otimes_{gr_\bullet \Hh_F^\dagger} gr_\bullet \Hh \rightarrow \ldots \rightarrow \bigoplus_{F \in \mathscr{F}_0} gr_\bullet \Hh (gr_\bullet(\epsilon_F)) \otimes_{gr_\bullet \Hh_F^\dagger} gr_\bullet \Hh \rightarrow gr_\bullet \Hh \rightarrow 0
\end{equation}
derived from \eqref{f:H-H-bimod} is an exact sequence of $(gr_\bullet \Hh, gr_\bullet \Hh)$-bimodules and is a gr-free resolution of $gr_\bullet \Hh$ as a left as well as a right graded $gr_\bullet \Hh$-module.
\end{coro}

\begin{proof}
It follows from Theorem \ref{strict-exact} that
\begin{equation*}
    0 \longrightarrow \bigoplus_{F \in \mathscr{F}_d} gr_\bullet(\Hh (\epsilon_F) \otimes_{\Hh_F^\dagger} \Hh) \longrightarrow \ldots \longrightarrow \bigoplus_{F \in \mathscr{F}_0} gr_\bullet (\Hh (\epsilon_F) \otimes_{\Hh_F^\dagger} \Hh) \longrightarrow gr_\bullet \Hh \longrightarrow 0
\end{equation*}
is exact. Moreover, by Lemma \ref{filt-free}.i and  \cite[Lemma I.6.14]{LO}, we have
\begin{equation*}
    gr_\bullet(\Hh (\epsilon_F) \otimes_{\Hh_F^\dagger} \Hh) = gr_\bullet (\Hh (gr_\bullet(\epsilon_F)) \otimes_{gr_\bullet \Hh_F^\dagger} gr_\bullet \Hh \ .
\end{equation*}
The freeness assertion follows from Lemma \ref{filt-free}.ii.
\end{proof}

\begin{lemm}\label{gr-F-injdim}
For any facet $F$ facet of $\mathscr{X}$ the $k$-algebra $gr_\bullet \Hh_F^\dagger$ is left and right noetherian and has self-injective dimension $r$.
\end{lemm}
\begin{proof}
This is proved in exactly the same way as Proposition \ref{F-injdim}. In fact, due to the simplified relations \eqref{f:nil} the argument becomes somewhat simpler.
\end{proof}

\begin{prop}\phantomsection\label{gr-Gor}
\begin{itemize}
  \item[i.] $gr_\bullet \Hh$ is finitely generated as a left and right module over a commutative subring which is a finitely generated $k$-algebra. In particular, $gr_\bullet \Hh$ is left and right noetherian.
  \item[ii.] $gr_\bullet \Hh$ is a Gorenstein ring of self-injective dimension bounded above by the rank of the group $\Gp$.
\end{itemize}
\end{prop}
\begin{proof}
i. Let $\tilde\Lambda := \Tp / \Tp^1$. It is immediate from the relations \eqref{f:nil} that
\begin{equation*}
    \mathbb{A} := \sum_{x \in \tilde\Lambda} k \bar{\tau}_x
\end{equation*}
is a commutative graded subalgebra of $gr_\bullet \Hh$. In $\tilde\Lambda$ we have the dominant (closed) Weyl chamber
\begin{equation*}
    \tilde\Lambda_{dom} := \{ x \in \tilde\Lambda :\; < x,\alpha >\; \geq 0\ \text{for any $\alpha \in \Phi^+$}\}.
\end{equation*}
Its translates $w(\tilde\Lambda_{dom})$ for $w \in \W_0$, are the (closed) Weyl chambers in $\tilde\Lambda$. According to \cite[1.4(a)]{Lu} the length function $\ell$ restricted to $\tilde\Lambda$ can be computed explicitly by
\begin{equation*}
    \ell(x) = \sum_{\alpha \in \Phi^+} |< x,\alpha >| \qquad\text{for any $x \in \tilde\Lambda$}.
\end{equation*}
We claim:
\begin{itemize}
  \item[(a)] $\ell | \tilde\Lambda$ is $\W_0$-invariant, which is easy.
  \item[(b)] For $x,x' \in \tilde\Lambda$ we have $\ell(x+x') = \ell(x) + \ell(x')$ if and only if $x$ and $x'$ lie in a common Weyl chamber.
\end{itemize}
Obviously $\ell$ is additive on the dominant Weyl chamber. Since $\W_0$ acts transitively on the set of all Weyl chambers it follows from (a) that $\ell$ is additive on any Weyl chamber. Conversely let $x,x' \in \tilde\Lambda$ such that $\ell(x+x') = \ell(x) + \ell(x')$. By (a) we may assume that $x+x' \in \tilde\Lambda_{dom}$. Then
\begin{align*}
    \ell(x+x') & = \sum_{\alpha \in \Phi^+} |< x+x',\alpha >| \leq \sum_{\alpha \in \Phi^+} |< x,\alpha >| + \sum_{\alpha \in \Phi^+} |< x',\alpha >| \\
    & = \ell(x) + \ell(x') = \ell(x+x') \ .
\end{align*}
This easily implies that $x,x' \in \tilde\Lambda_{dom}$. (This argument in fact shows that two points $x,x' \in \tilde\Lambda$ which lie in some common Weyl chamber must lie in any Weyl chamber which contains $x+x'$.)

Any Weyl chamber $w(\tilde\Lambda_{dom})$ is a saturated subsemigroup of $\tilde\Lambda$ and hence is finitely generated by Gordon's lemma (\cite[p.~7]{KKMS}). This implies that $\mathbb{A}$ is a finitely generated $k$-algebra: If $x_1, \ldots, x_r$ generate the semigroup $\tilde\Lambda_{dom}$ then $\{\bar{\tau}_{w(x_i)}\}_{w \in \W_0, 1 \leq i \leq r}$ generate the $k$-algebra $\mathbb{A}$.

As a consequence of \eqref{f:nil} we have the decomposition
\begin{equation*}
    gr_\bullet \Hh = \oplus_{w \in \W_0} \mathbb{A}(w) \qquad\text{with}\quad \mathbb{A}(w) := \sum_{x \in \tilde\Lambda} k \bar{\tau}_{x\tilde{w}}
\end{equation*}
as left $\mathbb{A}$-modules. We claim that each $\mathbb{A}$-module $\mathbb{A}(w)$ is finitely generated. Fixing $w$, there exist, as a consequence of \cite[(1.6.3)]{Vign} (compare also \cite[Lemma 3.3.6]{Born}), finitely many elements $y_1, \ldots, y_r \in \tilde\Lambda$ such that for any $x \in \tilde\Lambda$ there is an $1 \leq i \leq r$ with $\ell(x \tilde{w}) = \ell(xy_i^{-1}) + \ell(y_i \tilde{w})$. It therefore follows from \eqref{f:nil} that $\mathbb{A}(w)$ is generated by $\bar{\tau}_{y_1\tilde{w}}, \ldots, \bar{\tau}_{y_r\tilde{w}}$. The reasoning for $gr_\bullet \Hh$ as a right $\mathbb{A}$-module is analogous.

ii. The bound on the self-injective dimension follows from Lemma \ref{filt-free}.ii, Corollary \ref{gr-exact}, and Lemma \ref{gr-F-injdim} as in the proof of Theorem \ref{theo:Gorenstein}.
\end{proof}

\begin{prop}\label{gr-AG}
Suppose that the group $\Gp$ is semisimple. Then $gr_\bullet \Hh$ is Auslander-Gorenstein.
\end{prop}
\begin{proof}
It follows from Proposition \ref{gr-Gor} and   \cite[Cor.\ 13.1.13(iii) and Cor.\ 13.6.6(iii)]{MC} that $gr_\bullet \Hh$ is a fully bounded noetherian ring of finite self-injective dimension. Since the center of $\Gp$ is finite by assumption the group $\Omega$ is finite so that $gr_0 \Hh = k[\tilde\Omega]$ is an artinian ring. Hence our assertion follows from \cite[paragraph after Cor.~6.3]{SZ}.
\end{proof}

Everything in this section remains valid for the analogously formed graded rings $gr_\bullet \Hh'$ of $\Hh'$ and $gr_\bullet \Hh_{aff}$ of the affine Hecke algebra $\Hh_{aff}$ of the Coxeter system $(\W_{aff}, S_{aff})$.

\bigskip

\def\enotesize{\normalsize}

\end{document}